\documentclass{amsart}

\usepackage{amsmath, amsthm, amssymb}
\usepackage{mathrsfs}
\usepackage[all]{xy}
\usepackage{paralist}

\theoremstyle{definition}

\makeatletter
\newtheorem*{rep@theorem}{\rep@title}
\newcommand{\newreptheorem}[2]{%
\newenvironment{rep#1}[1]{%
 \def\rep@title{#2~\ref{##1}}%
 \begin{rep@theorem}}%
 {\end{rep@theorem}}}
 
\makeatother

\newtheorem{theorem}{Theorem}[section]
\newtheorem{definition}[theorem]{Definition}
\newtheorem{lemma}[theorem]{Lemma}
\newtheorem{proposition}[theorem]{Proposition}
\newtheorem{corollary}[theorem]{Corollary}
\newtheorem{conjecture}[theorem]{Conjecture}
\newtheorem*{claim}{Claim}

\newreptheorem{theorem}{Theorem}
\newreptheorem{proposition}{Proposition}

\theoremstyle{remark}
\newtheorem{remark}[theorem]{Remark}

\begin{document}

\title[Hodge theory for combinatorial geometries]{Hodge theory for combinatorial geometries}

\author{Karim Adiprasito}
\author{June Huh}
\author{Eric Katz}


\address{Einstein Institute of Mathematics, Hebrew University of Jerusalem, Jerusalem, Israel}
\email{adiprasito@math.huji.ac.il}

\address{Institute for Advanced Study, Princeton, New Jersey, USA.}
\email{junehuh@ias.edu}

\address{Department of Mathematics, The Ohio State University, Ohio, USA.}
\email{katz.60@osu.edu}

\maketitle

\section{Introduction} 

The combinatorial theory of matroids starts with Whitney \cite{Whitney}, who introduced matroids as models for independence 
 in vector spaces and graphs. See \cite[Chapter I]{SourceBook} for an excellent historical overview.
By definition, a \emph{matroid} $\mathrm{M}$ is given by a closure operator  defined on all subsets of a finite set $E$ satisfying the Steinitz-Mac\hspace{0.5mm}Lane exchange property: 
\begin{multline*}
\text{For every subset $I$ of $E$ and  every element $a$ not in the closure of $I$,}\\
\text{if $a$ is in the closure of ${I \cup\{ b\}}$, then $b$ is in the closure of  $I \cup \{a\}$.}
\end{multline*}
The matroid is called \emph{loopless} if the empty subset of $E$ is closed, and is called a \emph{combinatorial geometry} if in addition all single element subsets of $E$ are closed.
A closed subset of $E$  is called a \emph{flat} of $\mathrm{M}$, and every subset of $E$  has a well-defined rank and corank in the poset of all flats of $\mathrm{M}$.
The notion of matroid played a fundamental role in graph theory, coding theory, combinatorial optimization, and mathematical logic; we refer to \cite{Welsh} and \cite{Oxley} for general introduction.

As a generalization of the chromatic polynomial of a graph \cite{Birkhoff,WhitneyLE}, Rota defined for an arbitrary matroid $\mathrm{M}$  the \emph{characteristic polynomial} 
\[
\chi_\mathrm{M}( \lambda)=\sum_{I \subseteq E} (-1)^{|I|}\  \lambda^{\text{crk}(I)},
\]
where the sum is over all subsets $I \subseteq E$ and $\text{crk}(I)$ is the corank of $I$ in  $\mathrm{M}$ \cite{Foundations}.
Equivalently, the characteristic polynomial of $\mathrm{M}$ is
\[
\chi_{\mathrm{M}}(\lambda)\ =\ \sum_{F} \mu(\varnothing,F)
\hspace{0.5mm} \lambda^{\text{crk}(F)},
  \]
where the sum is over all flats $F$ of $\mathrm{M}$ and $\mu$ is the  M\"obius function of the poset of flats of $\mathrm{M}$,
 see Chapters 7 and 8 of~\cite{white}. Among the problems  that withstood many advances in matroid theory are the following log-concavity conjectures formulated in the 1970s.  

Write $r+1$ for the \emph{rank} of $\mathrm{M}$, that is, the rank of $E$ in the poset of flats of $\mathrm{M}$.

\begin{conjecture}\label{ConjectureCharacteristic}
Let $w_k(\mathrm{M})$ be the absolute value of the coefficient of $\lambda^{r-k+1}$ in the characteristic polynomial of $\mathrm{M}$.
Then the sequence $w_k(\mathrm{M})$ is log-concave:
\[
w_{k-1}(\mathrm{M}) w_{k+1}(\mathrm{M}) \le w_k(\mathrm{M})^2 \ \ \text{for all $1 \le k\le r$.}
\]
In particular, 
the sequence $w_k(\mathrm{M})$ is unimodal: 
\[
w_0(\mathrm{M}) \le w_1(\mathrm{M}) \le \cdots \le w_l(\mathrm{M})  \ge \cdots \ge w_r(\mathrm{M}) \ge w_{r+1}(\mathrm{M}) \ \ \text{for some index $l$.}
\]
\end{conjecture}

We remark that the positivity of the numbers $w_k(\mathrm{M})$  is used to deduce the unimodality from the log-concavity \cite[Chapter 15]{WelshBook}.

For chromatic polynomials, the unimodality was conjectured by Read, and the log-concavity was conjectured by Hoggar \cite{Read,Hoggar}.
The prediction of Read was then extended to arbitrary matroids by Rota  and  Heron,
and the conjecture in its full generality was given by Welsh  \cite{Rota,Heron,WelshBook}.
We refer to \cite[Chapter 8]{white}  and \cite[Chapter 15]{Oxley} for  overviews and historical accounts.

A subset $I \subseteq E$ is said to be \emph{independent} in $\mathrm{M}$ if no element $i$ in $I$ is in the closure of $I \setminus \{i\}$.
A related conjecture of Welsh and Mason  concerns the number of independent subsets of $E$ of given cardinality \cite{Welsh,Mason}.

\begin{conjecture}\label{ConjectureIndependence}
Let $f_k(\mathrm{M})$ be the number of independent subsets of $E$ with cardinality $k$.
Then the sequence $f_k(\mathrm{M})$ is log-concave:
\[
f_{k-1}(\mathrm{M}) f_{k+1}(\mathrm{M}) \le f_k(\mathrm{M})^2 \ \ \text{for all $1 \le k \le r$.}
\]
In particular, the sequence $f_k(\mathrm{M})$ is unimodal: 
\[
f_0(\mathrm{M}) \le f_1(\mathrm{M}) \le \cdots \le  f_l(\mathrm{M}) \ge \cdots \ge f_r(\mathrm{M}) \ge f_{r+1}(\mathrm{M}) \ \ \text{for some index $l$.}
\]
\end{conjecture}

We prove Conjecture \ref{ConjectureCharacteristic} and Conjecture \ref{ConjectureIndependence} 
by constructing a ``cohomology ring'' of $\mathrm{M}$ that satisfies the hard Lefschetz theorem and the Hodge-Riemann relations, see Theorem \ref{MainTheoremIntroduction}.

\subsection{} 



Matroid theory has experienced a remarkable development in the past century, and has been connected to diverse areas such as topology \cite{GM}, geometric model theory~\cite{Pillay}, and noncommutative geometry \cite{Neumann}. 
The study  of  hyperplane arrangements provided a particularly strong connection, see for example \cite{Orlik-Terao,Arrangements}. 
Most important for our purposes is the work of de Concini and Procesi on certain ``wonderful'' compactifications of hyperplane arrangement complements \cite{DeConcini-Procesi}. 
The original work focused only on realizable matroids, but Feichtner and Yuzvinsky \cite{Feichtner-Yuzvinsky} defined a commutative ring associated to an arbitrary matroid that specializes to the cohomology ring of a wonderful compactification in the realizable case. 

\begin{definition}
Let $S_\mathrm{M}$ be the polynomial ring
\[
S_\mathrm{M}:=\mathbb{R}\big[x_F| \text{$F$ is a nonempty proper flat of $\mathrm{M}$}\big].
\]
The \emph{Chow ring} of $\mathrm{M}$ is defined to be the quotient 
\[
A^*(\mathrm{M})_\mathbb{R}:=S_\mathrm{M}/(I_\mathrm{M}+J_\mathrm{M}),
\]
 where $I_\mathrm{M}$ is the ideal  generated by the quadratic monomials 
\[
x_{F_1}x_{F_2}, \ \ \text{$F_1$ and $F_2$ are two incomparable nonempty proper flats  of $\mathrm{M}$,}
\]
and  $J_\mathrm{M}$ is the ideal  generated by the linear forms 
\[
\sum_{i_1 \in F} x_F - \sum_{i_2 \in F} x_F, \ \ \text{$i_1$ and $i_2$ are  distinct elements of the ground set $E$.}
\]
\end{definition}

Conjecture \ref{ConjectureCharacteristic} was proved for matroids realizable over $\mathbb{C}$ in \cite{Huh} by relating  $w_k(\mathrm{M})$ to the Milnor numbers of a hyperplane arrangement realizing $\mathrm{M}$ over $\mathbb{C}$.
Subsequently in  \cite{Huh-Katz},  using the intersection theory of wonderful compactifications and the Khovanskii-Teissier inequality \cite[Section 1.6]{Lazarsfeld}, the conjecture was verified for matroids that are realizable over some field. Lenz used this result to deduce Conjecture \ref{ConjectureIndependence} for matroids realizable over some field \cite{Lenz}.

After the completion of \cite{Huh-Katz}, it was gradually realized that the validity of the Hodge-Riemann relations  for the Chow ring of $\mathrm{M}$  is a vital ingredient for the proof of the log-concavity conjectures, see Theorem \ref{MainTheoremIntroduction} below.
While the Chow ring of $\mathrm{M}$ could be defined for arbitrary $\mathrm{M}$, it was unclear how to formulate and prove the Hodge-Riemann relations. 
From the point of view of \cite{Feichtner-Yuzvinsky}, the ring $A^*(\mathrm{M})_\mathbb{R}$ is the Chow ring of a smooth, but noncompact toric variety $X(\Sigma_{\mathrm{M}})$, and  
there is no obvious way to reduce  to the classical case of projective varieties. 
In fact, we will see that $X(\Sigma_{\mathrm{M}})$ is ``Chow equivalent'' to a smooth or mildly singular projective variety over $\mathbb{K}$ if and only if the matroid $\mathrm{M}$ is realizable over $\mathbb{K}$, see Theorem~\ref{Chow-equivalence}.

\subsection{} 

We are  nearing a difficult chasm, as there is no reason to expect a working Hodge theory beyond the case of realizable matroids. Nevertheless, there was some evidence on the existence of such a theory
 for arbitrary matroids. 
 For example, it was proved in \cite{Adiprasito}, using the method of concentration of measure, that the log-concavity conjectures hold for a class of non-realizable matroids introduced by Goresky and MacPherson in \cite[III.4.1]{SMT}.


We now state the main theorem of this paper.
A real-valued function $c$ on the set of nonempty proper subsets of $E$  is said to be \emph{strictly submodular} if 
\[
 c_{I_1}+c_{I_2} > c_{I_1 \cap I_2} +c_{I_1 \cup I_2} \ \ \text{for any two incomparable subsets $I_1,I_2 \subseteq E$,}
\]
where we replace $c_\varnothing$ and $c_E$ by zero whenever they appear in the above inequality.
We note that strictly submodular functions exist. For example, 
\[
I \longmapsto |I| |E \setminus I|
\]
is a strictly submodular function.
A strictly submodular function $c$ defines an element
\[
\ell(c):= \sum_F c_F x_F\in A^1(\mathrm{M})_\mathbb{R},
\]
where the sum is over all nonempty proper flats of $\mathrm{M}$.
Note that the rank function of \emph{any} matroid on $E$ 
can, when restricted to the set of nonempty proper subsets of $E$, be obtained as a \emph{limit} of strictly submodular functions. 
We write ``$\text{deg}$'' for the isomorphism $A^r(\mathrm{M})_\mathbb{R} \simeq \mathbb{R}$ determined by the property that
\[
\text{deg}(x_{F_1}x_{F_2}\cdots x_{F_r})=1 \ \ \text{for any flag of nonempty proper flats} \ \ F_1 \subsetneq F_2 \subsetneq  \cdots \subsetneq F_r.
\]
We refer to Section \ref{SectionDegreeMap} for the existence and the uniqueness of the linear map ``$\text{deg}$''.

\begin{theorem}\label{MainTheoremIntroduction}
Let $\ell$ be an element of $A^1(\mathrm{M})_\mathbb{R}$ associated to a strictly submodular function.
\begin{compactenum}[(1)]
\item (Hard Lefschetz theorem) For every nonnegative integer $q \le \frac{r}{2}$, the multiplication by $\ell$ defines an isomorphism
\[
L_\ell^{q}:\ A^q(\mathrm{M})_{\mathbb{R}} \longrightarrow A^{r-q}(\mathrm{M})_{\mathbb{R}}, \qquad a \longmapsto  \ell^{r-2q} \cdot a.
\]
\item (Hodge--Riemann relations) For every nonnegative integer $q \le \frac{r}{2}$, the multiplication by $\ell$ defines a symmetric bilinear form
\[
Q_\ell^{q}:\ A^q(\mathrm{M})_{\mathbb{R}} \times A^q(\mathrm{M})_{\mathbb{R}} \longrightarrow \mathbb{R},  \qquad (a_1,a_2) \longmapsto (-1)^q \ \text{deg}(a_1 \cdot  L^q_\ell \ a_2)
\]
that is positive definite on the kernel of 
$\ell \cdot L^q_\ell$.
\end{compactenum} 
\end{theorem}

In fact, we will prove that  the Chow ring of $\mathrm{M}$ satisfies the hard Lefschetz theorem and the Hodge-Riemann relations with respect to any strictly convex piecewise linear function on the tropical linear space $\Sigma_\mathrm{M}$ associated to $\mathrm{M}$, see Theorem~\ref{MainTheoremBody}.
This implies Theorem~\ref{MainTheoremIntroduction}.
Our proof of the hard Lefschetz theorem and the Hodge-Riemann relations for general matroids is inspired by an ingenious inductive proof of the analogous facts for simple polytopes given by McMullen \cite{McMullen} (compare also \cite{CM} for related ideas in a different context). To show that this program, with a considerable amount of work, extends beyond polytopes, is our main purpose~here.

In Section \ref{SectionLCConjectures}, we show that the Hodge-Riemann relations, which are in fact stronger than the hard Lefschetz theorem, imply Conjecture \ref{ConjectureCharacteristic} and Conjecture \ref{ConjectureIndependence}.
We remark that, in the context of projective toric varieties, a similar reasoning leads to the Alexandrov-Fenchel inequality on mixed volumes of convex bodies.
In this respect, broadly speaking the approach of the present paper can be viewed as following  Rota's idea
 that  log-concavity conjectures should follow from their relation with the theory of mixed volumes of convex bodies, see \cite{Kung}.
 
\subsection{} 

There are other combinatorial approaches to  intersection theory for matroids.
 Mikhalkin et. al. introduced an integral Hodge structure for arbitrary matroids modeled on the cohomology of hyperplane arrangement complements   \cite{IKMZ}.
Adiprasito and Bj\"orner showed that an analogue of the Lefschetz hyperplane section theorem holds for all smooth (i.e.\ locally matroidal) projective tropical varieties \cite{ABj-pr}. 

Theorem \ref{MainTheoremIntroduction} should be compared with the counterexample to a version of Hodge conjecture for positive currents in \cite{BH}:
The example used in \cite{BH} gives a tropical variety that satisfies Poincar\'e duality, the hard Lefschetz theorem, but not the Hodge-Riemann relations. 

Finally, we remark that Zilber and Hrushovski have worked on subjects related to intersection theory for finitary combinatorial geometries, see \cite{Hrushovski}.  At present the relationship between their approach and ours is unclear.

 \subsection{Overview over the paper.}
 Sections~2 and~3 develop basic combinatorics and geometry of order filters in the poset of nonempty proper flats of a matroid $\mathrm{M}$. The order filters and the corresponding geometric objects $\Sigma_{\mathrm{M},\mathscr{P}}$, which are related to each other by ``matroidal flips'', play a central role in our inductive approach to the main theorem \ref{MainTheoremIntroduction}.
 
Sections~4 and~5 discuss piecewise linear and polynomial functions on simplicial fans, 
and in particular those on the Bergman fan $\Sigma_\mathrm{M}$.
These sections are more conceptual than the previous sections,
and, with the exception of the important technical subsection 4.3, 
can be read immediately after the introduction.

 In Section~6, we prove that the Chow ring $A^*(\mathrm{M})$ satisfies Poincar\'e duality.
The result and the inductive scheme in its proof  will be used in the proof of the main theorem \ref{MainTheoremIntroduction}.
After some general algebraic preparation in Section~7, the Hard Lefschetz theorem and the Hodge-Riemann relations for  matroids will be proved in Section~8.

 In Section~9, we identify the coefficients of the reduced characteristic polynomial of a matroid as ``intersection numbers'' in the Chow ring of the matroid.
The identification is used to deduce the log-concavity conjectures from the Hodge-Riemann relations.

\subsection*{Acknowledgements}
The authors thank Patrick Brosnan, Eduardo Cattani, Ben Elias, Ehud Hrushovski, Gil Kalai,  and Sam Payne for valuable conversations. We thank Antoine Chambert-Loir, Chi Ho Yuen, and the anonymous referees for meticulous reading. Their valuable suggestions significantly improved the quality of the paper.
Karim Adiprasito was supported by a Minerva Fellowship from the Max Planck Society and NSF Grant DMS-1128155.
June Huh was supported by a Clay Research Fellowship and NSF Grant DMS-1128155.
Eric Katz  was supported by an NSERC Discovery grant.

\section{Finite sets and their subsets}

\subsection{}

Let $E$ be a nonempty finite set of cardinality $n+1$, say $\{0,1,\ldots,n\}$.
We write $\mathbb{Z}^E$ for the free abelian group generated by the standard basis vectors $\mathbf{e}_i$  corresponding to the elements $i \in E$.
For an arbitrary subset $I \subseteq E$, we set
\[
\mathbf{e}_I:=\sum_{i \in I} \mathbf{e}_i.
\]
We associate to the set $E$ a dual pair of rank $n$ free abelian groups
\[
\mathbf{N}_E:=\mathbb{Z}^E/ \langle\mathbf{e}_E\rangle, \qquad \mathbf{M}_E:=\mathbf{e}_E^\perp \subset \mathbb{Z}^E, \qquad \langle -,- \rangle:\mathbf{N}_E \times \mathbf{M}_E \longrightarrow \mathbb{Z}.
\]
The corresponding real vector spaces will be denoted
\[
\mathbf{N}_{E,\mathbb{R}}:=\mathbf{N}_E \otimes_\mathbb{Z} \mathbb{R}, \qquad \mathbf{M}_{E,\mathbb{R}}:=\mathbf{M}_E \otimes_\mathbb{Z} \mathbb{R}.
\]
We use the same symbols $\mathbf{e}_i$ and $\mathbf{e}_I$ to denote their images in $\mathbf{N}_E$ and $\mathbf{N}_{E,\mathbb{R}}$.

The groups $\mathbf{N}$ and $\mathbf{M}$ associated to nonempty finite sets are related to each other in a natural way. 
For example, if $F$ is a nonempty subset of $E$, then there is a surjective homomorphism
\[
\mathbf{N}_E \longrightarrow \mathbf{N}_F, \qquad \mathbf{e}_I \longmapsto \mathbf{e}_{I \cap F},
\]
and an injective homomorphism
\[
\mathbf{M}_F \longrightarrow \mathbf{M}_E, \qquad \mathbf{e}_i-\mathbf{e}_j \longmapsto \mathbf{e}_i-\mathbf{e}_j.
\]
If $F$ is a nonempty  proper subset of $E$, we have a decomposition
\[
(\mathbf{e}_F^\perp \subset \mathbf{M}_E)=(\mathbf{e}_{E\setminus F}^\perp \subset \mathbf{M}_E)=\mathbf{M}_F \oplus \mathbf{M}_{E \setminus F}.
\]
Dually, we have an isomorphism from the quotient space
\[
\mathbf{N}_E/\langle \mathbf{e}_F\rangle=\mathbf{N}_E/\langle \mathbf{e}_{E \setminus F} \rangle \longrightarrow \mathbf{N}_F \oplus \mathbf{N}_{E \setminus F}, \qquad \mathbf{e}_I \longmapsto \mathbf{e}_{I \cap F} \oplus \mathbf{e}_{I \setminus F}.
\]
This isomorphism will be used later to analyze local structure of Bergman fans. 

More generally, for any map between nonempty finite sets  $\pi:E_1 \to E_2$, there is an associated homomorphism
\[
\pi_\mathbf{N}:\mathbf{N}_{E_2} \longrightarrow \mathbf{N}_{E_1}, \quad \mathbf{e}_I \longmapsto \mathbf{e}_{\pi^{-1}(I)},
\]
and the dual homomorphism 
\[
\pi_\mathbf{M}:\mathbf{M}_{E_1} \longrightarrow \mathbf{M}_{E_2}, \qquad \mathbf{e}_i-\mathbf{e}_j \longmapsto \mathbf{e}_{\pi(i)}-\mathbf{e}_{\pi(j)}. 
\]
When $\pi$ is surjective, $\pi_\mathbf{N}$ is injective and $\pi_\mathbf{M}$ is surjective.

\subsection{}

Let $\mathscr{P}(E)$ be the poset of nonempty proper subsets of $E$.
Throughout this section the symbol $\mathscr{F}$ will stand for a totally ordered subset of $\mathscr{P}(E)$, that is, a flag of nonempty proper subsets of $E$:
\[
\mathscr{F}=\Big\{F_1 \subsetneq F_2 \subsetneq \cdots \subsetneq F_l\Big\} \subseteq \mathscr{P}(E).
\]
We write $\text{min}\ \mathscr{F}$ for the intersection of all subsets in $\mathscr{F}$.
In other words, we set
\[
\text{min}\ \mathscr{F}:=\begin{cases} F_1& \text{if $\mathscr{F}$ is nonempty,}\\E & \text{if $\mathscr{F}$ is empty.}\end{cases}
\]

\begin{definition}\label{BooleanCompatiblePair}
When $I$ is a proper subset of  $\text{min}\ \mathscr{F}$, we say that $I$ is \emph{compatible} with $\mathscr{F}$ in $E$, and write $I<\mathscr{F}$.
\end{definition}

The set of all compatible pairs in $E$ form a poset under the relation
\[
(I_1<\mathscr{F}_1) \preceq (I_2<\mathscr{F}_2) \Longleftrightarrow \text{$I_1 \subseteq I_2$ and $\mathscr{F}_1 \subseteq \mathscr{F}_2$}.
\]
We note that any maximal compatible pair $I<\mathscr{F}$ gives a basis of the group $\mathbf{N}_E$:
\[
\Big\{\text{$\mathbf{e}_i$ and $\mathbf{e}_F$ for $i \in I$ and $F \in \mathscr{F}$}\Big\} \subseteq \mathbf{N}_E.
\]
If $0$ is the unique element of $E$ not in $I$ and not in any member of $\mathscr{F}$, then the above basis of $\mathbf{N}_E$ is related to the basis $\{\mathbf{e}_1,\mathbf{e}_2,\ldots,\mathbf{e}_n\}$ by an invertible upper triangular matrix.

\begin{definition}
For each compatible pair $I<\mathscr{F}$ in $E$, we define two polyhedra
\begin{align*}
\vartriangle_{I < \mathscr{F}}&:=\text{conv}\Big\{\text{$\mathbf{e}_i$ and $\mathbf{e}_F$ for $i \in I$ and $F \in \mathscr{F}$}\Big\}\subseteq \mathbf{N}_{E,\mathbb{R}}, \\[2pt]
\sigma_{I <\mathscr{F}}&:=\text{cone}\Big\{\text{$\mathbf{e}_i$ and $\mathbf{e}_F$ for $i \in I$ and $F \in \mathscr{F}$}\Big\} \subseteq \mathbf{N}_{E,\mathbb{R}}.
\end{align*}
Here ``\text{conv}\ S'' stands for the set of convex combinations of a set of vectors $S$, and ``\text{cone}\ S'' stands for the set of nonnegative linear combinations of a set of vectors $S$.
\end{definition}

Since maximal compatible pairs give bases of $\mathbf{N}_E$, the polytope $\vartriangle_{I < \mathscr{F}}$ is a simplex, and the cone $\sigma_{I<\mathscr{F}}$ is unimodular with respect to the lattice $\mathbf{N}_E$.
When $\{i\}$ is compatible with $\mathscr{F}$, 
\[
\vartriangle_{\{i\} < \mathscr{F}}  = \vartriangle_{\varnothing < \{\{i\}\} \cup \mathscr{F}} \ \ \text{and} \ \
\sigma_{\{i\} < \mathscr{F}}  = \sigma_{\varnothing < \{\{i\}\} \cup \mathscr{F}}.
\]
Any proper subset of $E$ is compatible with the empty flag in $\mathscr{P}(E)$, and the empty subset of $E$ is compatible with any flag in $\mathscr{P}(E)$. 
Thus we may write the simplex $\vartriangle_{I<\mathscr{F}}$ as the simplicial join
\[
\vartriangle_{I<\mathscr{F}}\ =\ \vartriangle_{I<\varnothing}*\vartriangle_{\varnothing<\mathscr{F}}
\]
and the cone $\sigma_{I<\mathscr{F}}$ as the vector sum
\[
 \sigma_{I<\mathscr{F}}=\sigma_{I<\varnothing}+\sigma_{\varnothing<\mathscr{F}}.
\]
The set of all simplices of the form $\vartriangle_{I<\mathscr{F}}$ is in fact a simplicial complex. More precisely, we have
\[
\vartriangle_{I_1 < \mathscr{F}_1} \cap \vartriangle_{I_2 < \mathscr{F}_2}=\vartriangle_{I_1 \cap I_2 < \mathscr{F}_1 \cap \mathscr{F}_2} \ \ \text{if $|I_1|\neq 1$ and $|I_2|\neq 1$}.
\]

\subsection{}

An \emph{order filter} $\mathscr{P}$ of $\mathscr{P}(E)$ is a collection of nonempty proper subsets of $E$ with the following property:
\[
 \text{If $F_1 \subseteq F_2$ are nonempty proper subsets of $E$, then $F_1 \in \mathscr{P}$ implies $F_2 \in \mathscr{P}$.}
\]
We do not require that $\mathscr{P}$ is closed under intersection of subsets.
We will see in Proposition \ref{NormalFan} that any such order filter cuts out a simplicial sphere in the simplicial complex of compatible pairs.

\begin{definition}
The \emph{Bergman complex} of an order filter $\mathscr{P} \subseteq \mathscr{P}(E)$ is the collection of simplices
\[
\Delta_\mathscr{P}:=\Big\{\text{$\vartriangle_{I < \mathscr{F}}\ $  for $I \notin \mathscr{P}$ and $\mathscr{F} \subseteq\mathscr{P}$}\Big\}.
\]
The \emph{Bergman fan} of an order filter $\mathscr{P} \subseteq \mathscr{P}(E)$ is the collection of simplicial cones
\[
\Sigma_\mathscr{P}:=\Big\{ \text{$\sigma_{I < \mathscr{F}}\ $  for $I \notin \mathscr{P}$ and $\mathscr{F} \subseteq\mathscr{P}$}\Big\}.
\]
The Bergman complex $\Delta_\mathscr{P}$ is a simplicial complex because $\mathscr{P}$ is an order filter.
\end{definition}


The extreme cases $\mathscr{P}=\varnothing$ and $\mathscr{P}=\mathscr{P}(E)$ correspond to familiar geometric objects.
When $\mathscr{P}$ is empty, the collection $\Sigma_\mathscr{P}$ is the normal fan of the standard $n$-dimensional simplex
\[
\Delta_n:=\text{conv}\big\{\mathbf{e}_0,\mathbf{e}_1,\ldots,\mathbf{e}_n\big\} \subseteq \mathbb{R}^E.
\]
When $\mathscr{P}$ contains all nonempty proper subsets of $E$, the collection $\Sigma_\mathscr{P}$ is the normal fan of the $n$-dimensional permutohedron
\[
\Pi_n:=\text{conv}\Big\{(x_0,x_1,\ldots,x_n) \mid \text{$x_0,x_1,\ldots,x_n$ is a permutation of $0,1,\ldots,n$}\Big\} \subseteq \mathbb{R}^E.
\]
Proposition~\ref{NormalFan} below shows that, in general,  the Bergman complex $\Delta_\mathscr{P}$ is a simplicial sphere and $\Sigma_\mathscr{P}$ is a complete unimodular fan.

\begin{proposition}\label{NormalFan}
For any order filter $\mathscr{P} \subseteq \mathscr{P}(E)$, the collection $\Sigma_\mathscr{P}$ is  the normal fan of a polytope.
\end{proposition} 

\begin{proof}
We show that $\Sigma_\mathscr{P}$ can be obtained from $\Sigma_\varnothing$ by 
performing a sequence of stellar subdivisions.
This implies that a polytope with normal fan $\Sigma_\mathscr{P}$ can be obtained by repeatedly truncating  faces of the standard simplex $\Delta_n$.
For a detailed discussion of stellar subdivisions of normal fans and truncations of polytopes, we refer to Chapters III and V of \cite{Ewald}.
In the language of toric geometry, this shows that the toric variety of $\Sigma_\mathscr{P}$ can be obtained from the $n$-dimensional projective space
by blowing-up torus orbit closures.

Choose a sequence of order filters   obtained by adding a single subset in $\mathscr{P}$ at a time:
\[
\varnothing,\ldots, \mathscr{P}_{-},\mathscr{P}_+,\ldots,\mathscr{P} \quad \text{with} \quad \mathscr{P}_+=\mathscr{P}_- \cup \{Z\}.
\]
The corresponding sequence of $\Sigma$ interpolates between the collections $\Sigma_\varnothing$ and $\Sigma_\mathscr{P}$:
\[
\Sigma_\varnothing \rightsquigarrow \ldots \rightsquigarrow    \Sigma_{\mathscr{P}_{-}}\rightsquigarrow \Sigma_{\mathscr{P}_{+}} \rightsquigarrow   \ldots \rightsquigarrow  \Sigma_{\mathscr{P}}.
\]
The modification in the middle replaces the cones of the form $\sigma_{Z<\mathscr{F}}$  with the sums of the form
\[
\sigma_{\varnothing<\{Z\}}+\sigma_{I<\mathscr{F}}, 
\]
where $I$ is any proper subset of $Z$.
In other words, the modification is the stellar subdivision of $\Sigma_{\mathscr{P}_-}$ relative to the cone $\sigma_{Z<\varnothing}$.
Since a stellar subdivision of the normal fan of a polytope is the normal fan of a polytope, by induction we know that the collection $\Sigma_\mathscr{P}$ is  the normal fan of a polytope.
\end{proof}

Note that, in the notation of the preceding paragraph, $\Sigma_{\mathscr{P}_-}=\Sigma_{\mathscr{P}_+}$ if $Z$ has cardinality $1$.

\section{Matroids and their flats}

\subsection{}

Let $\mathrm{M}$ be a loopless matroid of rank $r+1$ on the ground set $E$.
We denote $\text{rk}_\mathrm{M}$, $\text{crk}_\mathrm{M}$, and $\text{cl}_\mathrm{M}$ for the rank function, the corank function, and  the closure operator of $\mathrm{M}$ respectively. 
We omit the subscripts when $\mathrm{M}$ is understood from the context.
If $F$ is a nonempty proper flat of $\mathrm{M}$, we write
\begin{align*}
\mathrm{M}^F&:=\text{the restriction of $\mathrm{M}$ to $F$, a loopless matroid on $F$ of rank $=\text{rk}_\mathrm{M}(F)$},\\
\mathrm{M}_F&:=\text{the contraction of $\mathrm{M}$ by $F$, a loopless matroid on $E \setminus F$  of rank $=\text{crk}_\mathrm{M}(F)$}.
\end{align*}
We refer to \cite{Oxley} and \cite{WelshBook} for basic notions of matroid theory.

Let $\mathscr{P}(\mathrm{M})$ be the poset of nonempty proper flats of $\mathrm{M}$.
There is an injective map from the poset of the restriction
\[
\iota^F: \mathscr{P}(\mathrm{M}^F) \longrightarrow \mathscr{P}(\mathrm{M}), \qquad G \longmapsto G,
\]
and an injective map from the poset of the contraction
\[
\iota_F:\mathscr{P}(\mathrm{M}_F) \longrightarrow \mathscr{P}(\mathrm{M}), \qquad G \longmapsto G \cup F.
\]
We identify the flats of $\mathrm{M}_F$ with the flats of $\mathrm{M}$ containing $F$ using  $\iota_F$.
If $\mathscr{P}$ is a subset of $\mathscr{P}(\mathrm{M})$, we set
\[
\mathscr{P}^F:=(\iota^F)^{-1} \mathscr{P} \ \ \text{and} \ \ \mathscr{P}_F:=(\iota_F)^{-1} \mathscr{P}.
\]

\subsection{}

Throughout this section the symbol $\mathscr{F}$ will stand for a totally ordered subset of $\mathscr{P}(\mathrm{M})$, that is, a flag of nonempty proper flats of $\mathrm{M}$:
\[
\mathscr{F}=\Big\{F_1 \subsetneq F_2 \subsetneq \cdots \subsetneq F_l\Big\} \subseteq \mathscr{P}(\mathrm{M}).
\]
As before, we write $\text{min}\ \mathscr{F}$ for the intersection of all members of $\mathscr{F}$ inside $E$.
We extend the notion of compatibility in Definition~\ref{BooleanCompatiblePair} to the case when the matroid $\mathrm{M}$ is not Boolean.

\begin{definition}
When $I$ is a subset of $\text{min}\ \mathscr{F}$ of cardinality less than $\text{rk}_\mathrm{M}(\text{min}\ \mathscr{F})$, we say that $I$ is \emph{compatible} with $\mathscr{F}$ in $\mathrm{M}$, and write $I<_\mathrm{M}\mathscr{F}$.
\end{definition}

Since any flag of nonempty proper flats of $\mathrm{M}$ has length at most $r$, 
 any cone 
 \[
 \sigma_{I<_\mathrm{M}\mathscr{F}}=\text{cone}\Big\{\text{$\mathbf{e}_i$ and $\mathbf{e}_F$ for $i \in I$ and $F \in \mathscr{F}$}\Big\}
 \]
associated to a compatible pair in $\mathrm{M}$  has dimension at most $r$.
Conversely, any such cone  is contained in an $r$-dimensional  cone of the same type:
For this one may  take 
\begin{align*}
I' &= \text{a subset that is maximal among those containing $I$ and compatible with $\mathscr{F}$ in $\mathrm{M}$,}\\
\mathscr{F}' &= \text{a flag of flats maximal among those containing $\mathscr{F}$ and compatible with $I'$ in $\mathrm{M}$},
\end{align*}
or alternatively take
\begin{align*}
\mathscr{F}' &= \text{a flag of flats maximal among those containing $\mathscr{F}$ and compatible with $I$ in $\mathrm{M}$,}\\
I' &= \text{a subset that is maximal among those containing $I$ and compatible with $\mathscr{F}'$ in $\mathrm{M}$.}
\end{align*}

We note that any subset of $E$ with cardinality at most $r$ is compatible in $\mathrm{M}$ with the empty flag of flats, and the empty subset of $E$ is compatible in $\mathrm{M}$ with any flag of nonempty proper flats of $\mathrm{M}$. 
Therefore we may write
\[
\vartriangle_{I<_\mathrm{M}\mathscr{F}}\ =\ \vartriangle_{I<_\mathrm{M}\varnothing}*\vartriangle_{\varnothing<_\mathrm{M}\mathscr{F}}\ \ \text{and} \ \ \sigma_{I<_\mathrm{M}\mathscr{F}}=\sigma_{I<_\mathrm{M}\varnothing}+\sigma_{\varnothing<_\mathrm{M}\mathscr{F}}.
\]
The set of all simplices associated to compatible pairs in $\mathrm{M}$ form a simplicial complex, that is,
\[
\vartriangle_{I_1 <_\mathrm{M} \mathscr{F}_1} \cap \vartriangle_{I_2 <_\mathrm{M} \mathscr{F}_2}=\vartriangle_{I_1 \cap I_2 <_\mathrm{M} \mathscr{F}_1 \cap \mathscr{F}_2}.
\]

\subsection{}

An \emph{order filter} $\mathscr{P}$ of $\mathscr{P}(\mathrm{M})$ is a collection of nonempty proper flats of $\mathrm{M}$ with the following property:
\[
 \text{If $F_1 \subseteq F_2$ are nonempty proper flats of $\mathrm{M}$, then $F_1 \in \mathscr{P}$ implies $F_2 \in \mathscr{P}$.}
\]
We write $\widehat{\mathscr{P}}:=\mathscr{P} \cup \{E\}$ for the order filter of the lattice of flats of $\mathrm{M}$ generated by $\mathscr{P}$.

\begin{definition}\label{BergmanFan}
The \emph{Bergman fan} of an order filter $\mathscr{P} \subseteq \mathscr{P}(\mathrm{M})$ is the set of simplicial cones
\[
\Sigma_{\mathrm{M},\mathscr{P}}:=\Big\{ \text{$\sigma_{I < \mathscr{F}}\ $  for $\text{cl}_\mathrm{M}(I) \notin \widehat{\mathscr{P}}$ and $\mathscr{F} \subseteq\mathscr{P}$}\Big\}.
\]
The \emph{reduced Bergman fan} of $\mathscr{P}$ is the subset of the Bergman fan
\[
\widetilde{\Sigma}_{\mathrm{M},\mathscr{P}}:=\Big\{ \text{$\sigma_{I <_\mathrm{M} \mathscr{F}}\ $  for $\text{cl}_\mathrm{M}(I) \notin \widehat{\mathscr{P}}$ and $\mathscr{F} \subseteq\mathscr{P}$}\Big\}.
\]
When $\mathscr{P}=\mathscr{P}(\mathrm{M})$, we omit $\mathscr{P}$ from the notation and write the Bergman fan by $\Sigma_\mathrm{M}$.
\end{definition}

We note that the \emph{Bergman complex}  and the \emph{reduced Bergman complex} $\widetilde{\Delta}_{\mathrm{M},\mathscr{P}} \subseteq \Delta_{\mathrm{M},\mathscr{P}}$, defined in analogous ways using the simplices $\vartriangle_{I<\mathscr{F}}$ and $\vartriangle_{I<_\mathrm{M}\mathscr{F}}$,  share the same set of vertices.

Two extreme cases give familiar geometric objects.
When $\mathscr{P}$ is the set of all nonempty proper flats of $\mathrm{M}$, we have
\[
\Sigma_{\mathrm{M}}=\Sigma_{\mathrm{M},\mathscr{P}}=\widetilde{\Sigma}_{\mathrm{M},\mathscr{P}}=\text{the fine subdivision of the tropical linear space of $\mathrm{M}$ \cite{Ardila-Klivans}}.
\]
When $\mathscr{P}$ is empty, we have
\[
\widetilde{\Sigma}_{\mathrm{M},\varnothing}=\text{the $r$-dimensional skeleton of the normal fan of the simplex $\Delta_n$},
\]
and $\Sigma_{\mathrm{M},\varnothing}$ is the fan whose maximal cones are  $\sigma_{F<\varnothing}$ for rank $r$ flats $F$ of $\mathrm{M}$.
We remark that
\[
\Delta_{\mathrm{M},\varnothing}=\text{the Alexander dual of the matroid complex $\text{IN}(\textrm{M}^*)$ of the dual matroid $\mathrm{M}^*$}.
\]
See \cite{Bjorner} for basic facts on the matroid complexes and \cite[Chapter 5]{Miller-Sturmfels} for the Alexander dual of a simplicial complex.

We show that, in general, the Bergman fan and the reduced Bergman fan are indeed fans, and the reduced Bergman fan is pure of dimension $r$.

\begin{proposition}\label{Subfan}
The collection $\Sigma_{\mathrm{M},\mathscr{P}}$ is a subfan of the normal fan of a polytope.
\end{proposition}

\begin{proof}
Since $\mathscr{P}$ is an order filter, any face of a cone in $\Sigma_{\mathrm{M},\mathscr{P}}$ is in $\Sigma_{\mathrm{M},\mathscr{P}}$.
Therefore it is enough to show that there is a normal fan of a polytope that contains $\Sigma_{\mathrm{M},\mathscr{P}}$ as a subset.

For this we consider the order filter of $\mathscr{P}(E)$ generated by $\mathscr{P}$, that is, the collection of sets
\[
\widetilde{\mathscr{P}}:=\big\{\text{nonempty proper subset of $E$ containing a flat in $\mathscr{P}$}\big\} \subseteq \mathscr{P}(E).
\]
If the closure of  $I \subseteq E$ in $\mathrm{M}$ is not in $\widehat{\mathscr{P}}$, then $I$ does not contain any flat in  $\mathscr{P}$,
and hence
\[
\Sigma_{\mathrm{M},\mathscr{P}} \subseteq \Sigma_{\widetilde{\mathscr{P}}}.
\]
The latter collection is the normal fan of a polytope by Proposition~\ref{NormalFan}.
\end{proof}

Since $\mathscr{P}$ is an order filter, any face of a cone in $\widetilde{\Sigma}_{\mathrm{M},\mathscr{P}}$ is in $\widetilde{\Sigma}_{\mathrm{M},\mathscr{P}}$, and hence  $\widetilde{\Sigma}_{\mathrm{M},\mathscr{P}}$ is a subfan of $\Sigma_{\mathrm{M},\mathscr{P}}$. 
It follows that the reduced Bergman fan also is a subfan of the normal fan of a polytope.

\begin{proposition}\label{Purity}
The reduced Bergman fan $\widetilde{\Sigma}_{\mathrm{M},\mathscr{P}}$ is pure of dimension $r$.
\end{proposition}

\begin{proof}
Let $I$ be a subset of $E$ whose closure is not in $\mathscr{P}$, and let $\mathscr{F}$ be a flag of flats in $\mathscr{P}$ compatible with $I$ in $\mathrm{M}$.
We show that there are $I'$  containing $I$  and $\mathscr{F}'$ containing $\mathscr{F}$ such that
\[
I'<_\mathrm{M} \mathscr{F}', \quad \text{cl}_\mathrm{M}(I') \notin \widehat{\mathscr{P}}, \quad \mathscr{F}' \subseteq \mathscr{P}, \quad \text{and} \quad |I'|+|\mathscr{F}'|=r.
\]

First choose any flag of flats $\mathscr{F}'$ that is maximal among those containing $\mathscr{F}$, contained in $\mathscr{P}$, and compatible with $I$ in $\mathrm{M}$.
Next choose any flat $F$ of $\mathrm{M}$ that is maximal among those containing $I$ and strictly contained in $\text{min}\ \mathscr{F}'$.

We note that, by the maximality of  $F$ and the maximality of $\mathscr{F}'$ respectively,
\[
\text{rk}_\mathrm{M}(F)=\text{rk}_\mathrm{M}(\text{min}\ \mathscr{F}') -1=r-|\mathscr{F}'|.
\]
Since the rank of a set is at most its cardinality, the above implies
 \[
|I| \le r-|\mathscr{F}'| \le |F|.
\]
This shows that  there is $I'$ containing $I$, contained in $F$, and with cardinality exactly $r-|\mathscr{F}'|$.
Any such $I'$ is automatically compatible with $\mathscr{F}'$ in $\mathrm{M}$.

We show that the closure of $I'$ is not in $\mathscr{P}$ by showing that 
the flat $F$ is not in $\mathscr{P}$. If otherwise, by the maximality of $\mathscr{F}'$, the set $I$ cannot be compatible in $\mathrm{M}$ with the flag $\{F\}$, meaning
\[
|I| \ge \text{rk}_\mathrm{M}(F).
\]
The above implies that the closure of $I$ in $\mathrm{M}$, which is not in $\mathscr{P}$, is equal to $F$. This gives the desired contradiction.
\end{proof}

Our inductive approach to the hard Lefschetz theorem and the Hodge-Riemann relations for matroids is modeled on the observation that any facet of a permutohedron is the product of two smaller permutohedrons.
We note below that the Bergman fan $\Sigma_{\mathrm{M},\mathscr{P}}$ has an analogous local structure when $\mathrm{M}$ has no parallel elements, that is, when no two elements of $E$ are contained in a common rank $1$ flat of $\mathrm{M}$.

Recall that the \emph{star} of a cone $\sigma$ in a fan $\Sigma$ in a vector space $\mathbf{N}_\mathbb{R}$ is the fan 
\[
\text{star}(\sigma,\Sigma):=\big\{\overline{\sigma'} \mid \text{$\overline{\sigma'} $ is the image in $\mathbf{N}_\mathbb{R}/\langle \sigma\rangle$ of a cone $\sigma'$ in $\Sigma$ containing $\sigma$}\big\}.
\]
If $\sigma$ is a ray generated by a  vector $\mathbf{e}$,
we write $\text{star}(\mathbf{e},\Sigma)$ for the star of $\sigma$ in $\Sigma$.

\begin{proposition}\label{Star}
Let $\mathrm{M}$ be a loopless matroid on $E$, and let $\mathscr{P}$ be an order filter of $\mathscr{P}(\mathrm{M})$.
\begin{enumerate}[(1)]\itemsep 5pt
\item If $F$ is a flat in $\mathscr{P}$, then the isomorphism
$
\mathbf{N}_{E}/\langle \mathbf{e}_F\rangle \rightarrow \mathbf{N}_{F} \oplus \mathbf{N}_{E \setminus F}
$
induces a bijection
\[
\text{star}(\mathbf{e}_F,\Sigma_{\mathrm{M},\mathscr{P}}) \longrightarrow \Sigma_{\mathrm{M}^F,\mathscr{P}^F} \times \Sigma_{\mathrm{M}_F}.
\]
\item If $\{i\}$ is a proper flat of $\mathrm{M}$, then the isomorphism
$
\mathbf{N}_E/\langle \mathbf{e}_i \rangle \rightarrow \mathbf{N}_{E \setminus \{i\}} 
$
induces a bijection
\[
\text{star}(\mathbf{e}_i,\Sigma_{\mathrm{M},\mathscr{P}}) \longrightarrow \Sigma_{\mathrm{M}_{\{i\}},\mathscr{P}_{\{i\}}}.
\]
\end{enumerate}
Under the same assumptions, the stars of $\mathbf{e}_F$ and $\mathbf{e}_i$ in the reduced Bergman fan $\widetilde{\Sigma}_{\mathrm{M},\mathscr{P}}$ admit analogous descriptions.
\end{proposition}

Recall that a loopless matroid  is a combinatorial geometry if all single element subsets of E are flats.
When $\mathrm{M}$ is not a combinatorial geometry, the star of $\mathbf{e}_i$ in $\Sigma_{\mathrm{M},\mathscr{P}}$ is not necessarily a product of smaller Bergman fans.
However, when $\mathrm{M}$ is a combinatorial geometry, Proposition~\ref{Star} shows that
the star of every ray in $\Sigma_{\mathrm{M},\mathscr{P}}$ is a product of at most two Bergman fans.

\section{Piecewise linear functions and their convexity}\label{SectionPLFunctions}

\subsection{}\label{SectionGroupA1}

Piecewise linear functions on possibly incomplete fans will play an important role throughout the paper. In this section, we prove several general properties concerning convexity of such functions, working with a dual pair free abelian groups 
\[
\langle -,-\rangle: \mathbf{N} \times \mathbf{M} \longrightarrow \mathbb{Z}, \quad \mathbf{N}_\mathbb{R}:=\mathbf{N} \otimes_\mathbb{Z} \mathbb{R}, \quad \mathbf{M}_\mathbb{R}:=\mathbf{M} \otimes_\mathbb{Z} \mathbb{R},
\]
and a  fan $\Sigma$ in the vector space $\mathbf{N}_\mathbb{R}$. Throughout this section we assume that $\Sigma$ is  \emph{unimodular}, that is, every cone in $\Sigma$ is generated by a part of a basis of $\mathbf{N}$.
The set of primitive ray generators of $\Sigma$ will be denoted $V_\Sigma$.

We say that a function $\ell:|\Sigma| \rightarrow \mathbb{R}$ is  \emph{piecewise linear} if it is continuous and the restriction of $\ell$ to any cone in $\Sigma$ is the restriction of a  linear function on $\mathbf{N}_{\mathbb{R}}$. 
The function $\ell$ is said to be \emph{integral} if
\[
\ell\big( |\Sigma| \cap \mathbf{N}\big) \subseteq \mathbb{Z},
\]
and the function $\ell$ is said to be \emph{positive} if
\[
\ell\big( |\Sigma| \setminus \{0\}\big) \subseteq \mathbb{R}_{>0}.
\]

An important example of a piecewise linear function on $\Sigma$ is the \emph{Courant function} $x_\mathbf{e}$ associated to a primitive ray generator $\mathbf{e}$ of $\Sigma$,
whose values at $V_\Sigma$ are given by the Kronecker delta function.
Since  $\Sigma$ is unimodular, the Courant functions are integral, and they form a basis of the group of integral piecewise linear functions on $\Sigma$:
\[
\text{PL}(\Sigma)=\Bigg\{  \sum_{\mathbf{e} \in V_\Sigma} c_\mathbf{e} \hspace{0.5mm}x_\mathbf{e}
\mid c_\mathbf{e} \in \mathbb{Z} \Bigg\} \simeq \mathbb{Z}^{V_{\Sigma}}.
\]
An integral linear function on $\mathbf{N}_\mathbb{R}$ restricts to an integral piecewise linear function on $\Sigma$, giving a homomorphism
\[
\text{res}_\Sigma: \mathbf{M} \longrightarrow \text{PL}(\Sigma), \qquad m \longmapsto \sum_{\mathbf{e} \in V_\Sigma} \langle \mathbf{e},m \rangle   \hspace{0.5mm} x_\mathbf{e}.
\]
We denote the cokernel of the restriction map by
\[
A^1(\Sigma):=\text{PL}(\Sigma)/\mathbf{M}.
\]
In general, this group may have torsion, even under our assumption that $\Sigma$ is unimodular. 
When integral piecewise linear functions $\ell$ and $\ell'$ on $\Sigma$ differ by the restriction of an integral linear function on $\mathbf{N}_\mathbb{R}$, we say that $\ell$ and $\ell'$ are \emph{equivalent} over $\mathbb{Z}$.

Note that the group of piecewise linear functions  modulo linear functions on $\Sigma$ can be identified with the tensor product
\[
A^1(\Sigma)_\mathbb{R}:=A^1(\Sigma) \otimes_\mathbb{Z} \mathbb{R}.
\]
When piecewise linear functions $\ell$ and $\ell'$ on $\Sigma$ differ by the restriction of a linear function on $\mathbf{N}_\mathbb{R}$, we say that $\ell$ and $\ell'$ are \emph{equivalent}.

We describe three basic pullback homomorphisms between the groups $A^1$.
Let $\Sigma'$ be a subfan of $\Sigma$, and  let $\sigma$ be a cone in $\Sigma$. 

\begin{enumerate}[(1)]\itemsep 5pt
\item
The restriction of functions from $\Sigma$ to $\Sigma'$ defines a surjective homomorphism
\[
\text{PL}(\Sigma) \longrightarrow \text{PL}(\Sigma'),
\]
and this descends to a surjective homomorphism
\[
\textrm{p}_{\Sigma' \subseteq \Sigma}: A^1(\Sigma) \longrightarrow A^1(\Sigma').
\]
In terms of Courant functions, $\textrm{p}_{\Sigma' \subseteq \Sigma}$ is uniquely determined by its values
\[
x_\mathbf{e} \longmapsto \begin{cases} x_{\mathbf{e}} & \text{if $\mathbf{e}$ is in $V_{\Sigma'}$,} \\ 0& \text{if otherwise}. \end{cases}
\]
\item
Any integral piecewise linear function $\ell$ on $\Sigma$ is equivalent over $\mathbb{Z}$ to an integral piecewise linear function $\ell'$ that is zero on $\sigma$, and the choice of such $\ell'$ is unique up to an integral linear function on $\mathbf{N}_\mathbb{R}/\langle \sigma\rangle$.
Therefore we have a surjective homomorphism
\[
\textrm{p}_{\sigma \in \Sigma}: A^1(\Sigma) \longrightarrow A^1(\text{star}(\sigma,\Sigma)),
\]
uniquely determined by its values on $x_{\mathbf{e}}$ for primitive ray generators $\mathbf{e}$ not contained in $\sigma$:
\[
x_\mathbf{e} \longmapsto \begin{cases} x_{\overline{\mathbf{e}}} & \text{if there is a cone in $\Sigma$ containing $\mathbf{e}$ and $\sigma$,} \\ 0& \text{if otherwise}. \end{cases}
\]
Here $\overline{\mathbf{e}}$ is the image of $\mathbf{e}$ in the quotient space $\mathbf{N}_\mathbb{R}/\langle \sigma\rangle$.
\item
A piecewise linear function on the product of two fans $\Sigma_1 \times \Sigma_2$ is the sum of its restrictions to the subfans 
\[
\Sigma_1 \times \{0\} \subseteq \Sigma_1 \times \Sigma_2 \ \ \text{and} \ \ \{0\} \times \Sigma_2 \subseteq  \Sigma_1 \times \Sigma_2.
\]
Therefore we have an isomorphism
\[
\text{PL}(\Sigma_1 \times \Sigma_2) \simeq \text{PL}(\Sigma_1) \oplus \text{PL}(\Sigma_2),
\]
and this descends to an isomorphism
 \[
\textrm{p}_{\Sigma_1,\Sigma_2}: A^1(\Sigma_1 \times \Sigma_2) \simeq A^1(\Sigma_1) \oplus A^1(\Sigma_2).
\]
\end{enumerate}

\subsection{}

We define the \emph{link} of a cone $\sigma$ in $\Sigma$ to be the collection
\[
\text{link}(\sigma,\Sigma):=\big\{\sigma' \in \Sigma \mid \text{$\sigma'$ is contained in a cone in $\Sigma$ containing $\sigma$, and $\sigma \cap \sigma'=\{0\}$} \big\}.
\]
Note that the link of $\sigma$ in $\Sigma$ is a subfan of $\Sigma$.

\begin{definition}\label{DefinitionConvexity}
Let $\ell$ be a piecewise linear function on $\Sigma$, and let $\sigma$ be a cone in  $\Sigma$.
\begin{enumerate}[(1)]\itemsep 5pt
\item The function  $\ell$ is \emph{convex} around $\sigma$ if it is equivalent to a piecewise linear function that is zero on $\sigma$ and nonnegative on the rays of the link of $\sigma$.
\item The function $\ell$ is \emph{strictly convex} around $\sigma$ if it is equivalent to a piecewise linear function that is zero on $\sigma$ and positive on the rays of the link of $\sigma$.
\end{enumerate}
The function $\ell$ is \emph{convex} if it is convex around every cone in $\Sigma$, and \emph{strictly convex} if it is strictly convex around every cone in $\Sigma$.
\end{definition}

When $\Sigma$ is complete, the function $\ell$ is convex in the sense of Definition~\ref{DefinitionConvexity} if and only if  it is convex in the usual sense:
\[
\ell(\mathbf{u}_1+\mathbf{u}_2) \le \ell(\mathbf{u}_1)+\ell(\mathbf{u}_2) \ \ \text{for} \ \ \mathbf{u}_1, \mathbf{u}_2 \in \mathbf{N}_\mathbb{R}.
\]
In general, writing $\iota$ for the inclusion of the torus orbit closure corresponding to $\sigma$ in the toric variety of $\Sigma$, we have
\[
\text{$\ell$ is convex around $\sigma$} \Longleftrightarrow \text{$\iota^*$ of the class of the divisor associated to $\ell$ is effective}.
\]
For a detailed discussion  and related notions of convexity from the point of view of toric geometry, see \cite[Section 2]{Gibney-Maclagan}.

\begin{definition}
The \emph{ample cone} of $\Sigma$ is the open convex cone
\[
\mathscr{K}_\Sigma:=\big\{\text{classes of strictly convex piecewise linear functions on $\Sigma$}\big\} \subseteq A^1(\Sigma)_\mathbb{R}.
\]
The \emph{nef cone} of $\Sigma$ is the closed convex cone
\[
\mathscr{N}_\Sigma:=\big\{\text{classes of convex piecewise linear functions on $\Sigma$}\big\} \subseteq A^1(\Sigma)_\mathbb{R}.
\]
\end{definition}

Note that the closure of the ample cone $\mathscr{K}_\Sigma$ is contained in the nef cone $\mathscr{N}_\Sigma$.
In many interesting cases, the reverse inclusion also holds.

\begin{proposition}\label{Kleiman}
If $\mathscr{K}_\Sigma$ is nonempty, then $\mathscr{N}_\Sigma$ is the closure of $\mathscr{K}_\Sigma$.
\end{proposition}

\begin{proof}
If $\ell_1$ is a convex piecewise linear function and $\ell_2$ is strictly convex piecewise linear function on $\Sigma$, then the sum $\ell_1+\epsilon \hspace{0.5mm}  \ell_2$ is strictly convex for every positive number $\epsilon$. 
This shows that the nef cone of $\Sigma$ is in the closure of the ample cone of $\Sigma$.
\end{proof}

We record here that the various pullbacks of an ample class are ample.
The proof is straightforward from Definition~\ref{DefinitionConvexity}.

\begin{proposition}\label{PropositionAmplePullback}
Let $\Sigma'$ be a subfan of $\Sigma$, $\sigma$ be a cone in $\Sigma$, and let $\Sigma_1 \times \Sigma_2$ be a product fan.
\begin{enumerate}[(1)]\itemsep 5pt
\item
The pullback homomorphism $\textrm{p}_{\Sigma' \subseteq\Sigma}$ induces a map between the ample cones
\[
 \mathscr{K}_{\Sigma} \longrightarrow \mathscr{K}_{\Sigma'}.
\]
\item
The pullback  homomorphism $\textrm{p}_{\sigma \in \Sigma}$ induces a map between the ample cones
\[
 \mathscr{K}_\Sigma \longrightarrow \mathscr{K}_{\text{star}(\sigma,\Sigma)}.
\]
\item
The isomorphism $\textrm{p}_{\Sigma_1,\Sigma_2}$ induces a bijective map
 between the ample cones
\[
\mathscr{K}_{\Sigma_1 \times \Sigma_2} \longrightarrow \mathscr{K}_{\Sigma_1} \times \mathscr{K}_{\Sigma_2}.
\]
\end{enumerate}
\end{proposition}

Recall that the support function of a polytope is strictly convex piecewise linear function on the normal fan of the polytope. 
An elementary proof can be found in  \cite[Corollary A.19]{Oda}.
It follows from the first item of Proposition~\ref{PropositionAmplePullback} that any subfan of the normal fan of a polytope has a nonempty ample cone.
In particular, by Proposition~\ref{Subfan}, the Bergman fan  $\Sigma_{\mathrm{M},\mathscr{P}}$ has a nonempty ample cone.

Strictly convex piecewise linear functions on the normal fan of the permutohedron can be described in a particularly nice way:
A piecewise linear function on $\Sigma_{\mathscr{P}(E)}$ is strictly convex if and only if it is of the form
\[
\sum_{F \in \mathscr{P}(E)} c_F x_F, \ \ \text{$c_{F_1}+c_{F_2} > c_{F_1 \cap F_2} +c_{F_1 \cup F_2}$ for any incomparable $F_1,F_2$, with $c_\varnothing=c_E=0$}.
\]
For this and related results, see \cite{Batyrev-Blume}.
The restriction of any such \emph{strictly submodular function} gives a strictly convex function on the Bergman fan $\Sigma_\mathrm{M}$, and defines an ample class on $\Sigma_\mathrm{M}$.

\subsection{}

We specialize to the case of matroids and prove basic properties of convex piecewise linear functions on the Bergman fan $\Sigma_{\mathrm{M},\mathscr{P}}$.
We write $\mathscr{K}_{\mathrm{M},\mathscr{P}}$ for the ample cone of $\Sigma_{\mathrm{M},\mathscr{P}}$, and  $\mathscr{N}_{\mathrm{M},\mathscr{P}}$ for the nef cone of $\Sigma_{\mathrm{M},\mathscr{P}}$.

\begin{proposition}\label{MatroidKleiman}
Let $\mathrm{M}$ be a loopless matroid on $E$, and let $\mathscr{P}$ be an order filter of $\mathscr{P}(\mathrm{M})$.
\begin{enumerate}[(1)]\itemsep 5pt
\item The nef cone of $\Sigma_{\mathrm{M},\mathscr{P}}$ is equal to the closure of the ample cone of $\Sigma_{\mathrm{M},\mathscr{P}}$:
\[
\overline{\mathscr{K}_{\mathrm{M},\mathscr{P}}}=\mathscr{N}_{\mathrm{M},\mathscr{P}}.
\]
\item The ample cone of $\Sigma_{\mathrm{M},\mathscr{P}}$ is equal to the interior of the nef cone of $\Sigma_{\mathrm{M},\mathscr{P}}$:
\[
\mathscr{K}_{\mathrm{M},\mathscr{P}}=\mathscr{N}^\circ_{\mathrm{M},\mathscr{P}}.
\]
\end{enumerate}
\end{proposition}

\begin{proof}
Propositions~\ref{Subfan} shows that the ample cone $\mathscr{K}_{\mathrm{M},\mathscr{P}}$ is nonempty.
Therefore, by Proposition~\ref{Kleiman}, the nef cone  $\mathscr{N}_{\mathrm{M},\mathscr{P}}$ is equal to the closure of $\mathscr{K}_{\mathrm{M},\mathscr{P}}$.

The second assertion can be deduced from the first using the following general property of convex sets:
An open convex set is equal to the interior of its closure.
\end{proof}

The main result here is that the ample cone and its ambient vector space
\[
\mathscr{K}_{\mathrm{M},\mathscr{P}} \subseteq A^1(\Sigma_{\mathrm{M},\mathscr{P}})_\mathbb{R}
\]
depend only on $\mathscr{P}$ and the combinatorial geometry of $\mathrm{M}$, see Proposition~\ref{CombinatorialGeometry} below.
We set
\[
\overline{E}:=\big\{ A \mid \text{$A$ is a rank $1$ flat of $\mathrm{M}$}\big\}.
\]

\begin{definition}
The \emph{combinatorial geometry} of $\mathrm{M}$ is the simple matroid $\overline{\mathrm{M}}$ on $\overline{E}$
 determined by its poset of nonempty proper flats $\mathscr{P}(\overline{\mathrm{M}}) = \mathscr{P}(\mathrm{M})$.
\end{definition}

The set of primitive ray generators of $\Sigma_{\mathrm{M},\mathscr{P}}$ is the disjoint union
\[
\big\{\mathbf{e}_i \mid \text{the closure of $i$ in $\mathrm{M}$ is not in $\mathscr{P}$}\big\}
 \cup \big\{ \mathbf{e}_F \mid \text{$F$ is a flat in $\mathscr{P}$}\big\}  \subseteq \mathbf{N}_{E,\mathbb{R}},
 \]
and
the set of primitive ray generators of $\Sigma_{\overline{\mathrm{M}},\mathscr{P}}$ is the disjoint union
 \[
 \big\{\mathbf{e}_A \mid \text{$A$ is a rank $1$ flat of $\mathrm{M}$ not in $\mathscr{P}$}\big\}  
 \cup \big\{ \mathbf{e}_{F} \mid \text{$F$ is a flat in $\mathscr{P}$}\big\}\subseteq \mathbf{N}_{\overline{E},\mathbb{R}}.
\]
The corresponding Courant functions on the Bergman fans will be denoted  $x_i$,  $x_F$, and $x_A$, $x_F$ respectively.

Let $\pi$ be the surjective map between the ground sets of $\mathrm{M}$ and $\overline{\mathrm{M}}$ given by the closure operator of $\mathrm{M}$. We fix an arbitrary section $\iota$ of $\pi$ by choosing an element from each rank $1$ flat:
\[
\pi: E \longrightarrow \overline{E}, \qquad \iota: \overline{E} \longrightarrow E, \qquad \pi\circ\iota=\text{id}.
\]
The maps $\pi$ and $\iota$ induce the horizontal homomorphisms in the diagram 
\[
\xymatrixcolsep{5pc}
\xymatrixrowsep{3pc}
\xymatrix{
\text{PL}(\Sigma_{\mathrm{M},\mathscr{P}}) \ar@<0.5ex>[r]^{\pi_{\text{PL}}} & \text{PL}(\Sigma_{\overline{\mathrm{M}},\mathscr{P}}) \ar@<0.5ex>[l]^{\iota_{\text{PL}}}\\
\mathbf{M}_E \ar@<0.5ex>[r]^{\pi_\mathbf{M}} \ar[u]^{\text{res}} & \mathbf{M}_{\overline{E}}, \ar[u]_{\text{res}} \ar@<0.5ex>[l]^{\iota_\mathbf{M}}
}
\]
where the homomorphism $\pi_{\text{PL}}$ is obtained by setting
\[
x_i \longmapsto x_{\pi(i)}, \quad  x_F \longmapsto x_F,  \quad \text{for elements $i$ whose closure is not in $\mathscr{P}$, and for flats $F$ in $\mathscr{P}$,
}
\]
and the homomorphism $\iota_{\text{PL}}$ is obtained by setting
\[
x_A \longmapsto x_{\iota(A)}, \quad  x_F \longmapsto x_F,  \quad \text{for rank $1$ flats $A$ not in $\mathscr{P}$, and for flats $F$ in $\mathscr{P}$.}
\]
In the diagram above,  we have
\[
\pi_{\text{PL}} \circ \text{res}= \text{res} \circ \pi_\mathbf{M}, \quad
\iota_{\text{PL}} \circ \text{res}= \text{res} \circ \iota_\mathbf{M}, \quad
\pi_{\text{PL}} \circ \iota_{\text{PL}}=\text{id}, \quad
\pi_{\mathbf{M}} \circ \iota_{\mathbf{M}}=\text{id}.
\]

\begin{proposition}\label{PropositionSimplification}
The homomorphism $\pi_{\text{PL}}$ induces an isomorphism
\[
\underline{\pi}_{\text{PL}}:  A^1(\Sigma_{\mathrm{M},\mathscr{P}}) \longrightarrow A^1(\Sigma_{\overline{\mathrm{M}},\mathscr{P}}).
\]
The homomorphism $\iota_{\text{PL}}$ induces the inverse isomorphism
\[
\underline{\iota}_{\text{PL}}:  A^1(\Sigma_{\overline{\mathrm{M}},\mathscr{P}}) \longrightarrow A^1(\Sigma_{\mathrm{M},\mathscr{P}}).
\]
\end{proposition}

We use the same symbols  to denote the isomorphisms 
$
 A^1(\Sigma_{\mathrm{M},\mathscr{P}})_\mathbb{R} \leftrightarrows A^1(\Sigma_{\overline{\mathrm{M}},\mathscr{P}})_\mathbb{R}.
$

\begin{proof}
It is enough to check that the composition $\underline{\iota}_{\text{PL}} \circ \underline{\pi}_{\text{PL}}$ is the identity.
Let $i$ and $j$ be elements whose closures are not in $\mathscr{P}$.
Consider the linear function on $\mathbf{N}_{E,\mathbb{R}}$ given by the integral vector 
\[
\mathbf{e}_i-\mathbf{e}_j \in \mathbf{M}_E.
\] 
The restriction of this linear function to $\Sigma_{\mathrm{M},\mathscr{P}}$ is the linear combination
\[
\text{res}(\mathbf{e}_i-\mathbf{e}_j)=\Big(x_i+\sum_{i \in F \in \mathscr{P}} x_F\Big)-\Big(x_j+\sum_{j \in F \in \mathscr{P}} x_F\Big).
\]
If $i$ and $j$ have the same closure, then a flat  contains $i$ if and only if it contains $j$, and hence  the linear function witnesses that the piecewise linear functions $x_i$ and $x_j$ are equivalent over $\mathbb{Z}$.
It follows that  $\underline{\iota}_{\text{PL}} \circ \underline{\pi}_{\text{PL}}=\text{id}$.
\end{proof}

The maps $\pi$ and $\iota$ induce simplicial maps between the Bergman complexes
\[
\xymatrix{
\Delta_{\mathrm{M},\mathscr{P}} \ar@<0.5ex>[rr]^{\pi_\Delta} && \Delta_{\overline{\mathrm{M}},\mathscr{P}} \ar@<0.5ex>[ll]^{\iota_\Delta},
} \qquad
\vartriangle_{I< \mathscr{F}}\ \longmapsto\  \vartriangle_{\pi(I)< \mathscr{F}},
\quad
 \vartriangle_{\mathscr{I}< \mathscr{F}}\ \longmapsto\  \vartriangle_{\iota(\mathscr{I})< \mathscr{F}}.
\]
The simplicial map $\pi_\Delta$ collapses those simplices containing vectors of parallel elements,
and  
\[
\pi_\Delta \circ \iota_\Delta=\text{id}.
\]
The other composition $\iota_\Delta \circ \pi_\Delta$ is a deformation retraction. 
For this  note that
 \[
\vartriangle_{I<\mathscr{F}} \ \in \Delta_{\mathrm{M},\mathscr{P}} \Longrightarrow \iota_\Delta \circ \pi_\Delta (\vartriangle_{I<\mathscr{F}}) \ \cup \vartriangle_{I<\mathscr{F}} \ \subseteq \ \vartriangle_{\pi^{-1} \pi I<\mathscr{F}.}
\]
The simplex $\vartriangle_{\pi^{-1} \pi I<\mathscr{F}}$ is in $\Delta_{\mathrm{M},\mathscr{P}}$, and hence we can find a homotopy $\iota_\Delta \circ \pi_\Delta \simeq \text{id}$.

\begin{proposition}\label{CombinatorialGeometry}
The isomorphism $\underline{\pi}_{\text{PL}}$ restricts to a bijective map between the ample cones
\[
\mathscr{K}_{{\mathrm{M},\mathscr{P}}}  \longrightarrow \mathscr{K}_{{\overline{\mathrm{M}},\mathscr{P}}}.
\]
\end{proposition}

\begin{proof} 
By Proposition~\ref{MatroidKleiman}, it is enough to show that $\underline{\pi}_{\text{PL}}$ restricts to a bijective map
\[
\mathscr{N}_{{\mathrm{M},\mathscr{P}}}  \longrightarrow \mathscr{N}_{{\overline{\mathrm{M}},\mathscr{P}}}.
\]
We use the following maps corresponding to $\pi_\Delta$ and $\iota_\Delta$:
\[
\xymatrix{
\Sigma_{\mathrm{M},\mathscr{P}} \ar@<0.5ex>[rr]^{\pi_\Sigma} && \Sigma_{\overline{\mathrm{M}},\mathscr{P}} \ar@<0.5ex>[ll]^{\iota_\Sigma},
} \qquad
\sigma_{I< \mathscr{F}}\ \longmapsto\  \sigma_{\pi(I)< \mathscr{F}},
\quad
 \sigma_{\mathscr{I}< \mathscr{F}}\ \longmapsto\  \sigma_{\iota(\mathscr{I})< \mathscr{F}}.
\]

One direction is more direct:
The homomorphism $\iota_{\text{PL}}$ maps a convex piecewise linear function $\overline{\ell}$ to a convex piecewise linear function $\iota_{\text{PL}}(\overline{\ell})$.
Indeed, for any cone $\sigma_{I<\mathscr{F}}$ in $\Sigma_{\mathrm{M},\mathscr{P}}$,
\begin{multline*}
\Big(\text{$\overline{\ell}$ is zero on $\sigma_{\pi(I)<\mathscr{F}}$ and nonnegative on the link of $\sigma_{\pi(I)<\mathscr{F}}$ in $\Sigma_{\overline{\mathrm{M}},\mathscr{P}}$}\Big) \Longrightarrow\\
 \Big(\text{$\iota_{\text{PL}}(\overline{\ell})$ is zero on $\sigma_{\pi^{-1}\pi(I)<\mathscr{F}}$ and nonnegative on the link of $\sigma_{\pi^{-1}\pi(I)<\mathscr{F}}$ in $\Sigma_{\mathrm{M},\mathscr{P}}$}\Big)\\
   \Longrightarrow 
   \Big(\text{$\iota_{\text{PL}}(\overline{\ell})$ is zero on $\sigma_{I<\mathscr{F}}$ and nonnegative on the link of $\sigma_{I<\mathscr{F}}$ in $\Sigma_{\mathrm{M},\mathscr{P}}$}\Big).
\end{multline*}

Next we show the other direction:  The homomorphism $\pi_{\text{PL}}$ maps a convex piecewise linear function $\ell$ to a convex piecewise linear function $\pi_{\text{PL}}(\ell)$.
The main claim is that, for any cone $\sigma_{\mathscr{I}<\mathscr{F}}$ in $\Sigma_{\overline{\mathrm{M}},\mathscr{P}}$, 
\[
\text{$\ell$ is convex around $\sigma_{\pi^{-1}(\mathscr{I})<\mathscr{F}}$}\Longrightarrow \text{ $\pi_{\text{PL}}(\ell)$ is convex around $\sigma_{\mathscr{I}<\mathscr{F}}$.}
\]
This can be deduced from the following identities between the subfans of $\Sigma_{\mathrm{M},\mathscr{P}}$:
\begin{align*}
\pi_\Sigma^{-1}\Big( \text{the set of all faces of $\sigma_{\mathscr{I}<\mathscr{F}}$}\Big)&= 
\Big( \text{the set of all faces of $\sigma_{\pi^{-1}(\mathscr{I})<\mathscr{F}}$}\Big), \\
\pi_\Sigma^{-1} \Big( \text{the link of $\sigma_{\mathscr{I}<\mathscr{F}}$ in $\Sigma_{\overline{\mathrm{M}},\mathscr{P}}$}\Big)&=
\Big( \text{the link of $\sigma_{\pi^{-1}(\mathscr{I})<\mathscr{F}}$ in $\Sigma_{\mathrm{M},\mathscr{P}}$}\Big).
\end{align*}
It is straightforward to check the two equalities  from the definitions of $\Sigma_{\mathrm{M},\mathscr{P}}$ and $\Sigma_{\overline{\mathrm{M}},\mathscr{P}}$.
\end{proof}

\begin{remark}
Note that a Bergman fan and the corresponding reduced Bergman fan share the same set of primitive ray generators. Therefore we have isomorphisms

\[
\xymatrix{
A^1(\Sigma_{\mathrm{M},\mathscr{P}}) \ar@<0.5ex>[r] \ar@<0.5ex>[d] & A^1(\Sigma_{\overline{\mathrm{M}},\mathscr{P}})   \ar@<0.5ex>[d] \ar@<0.5ex>[l]\\
A^1(\widetilde{\Sigma}_{\mathrm{M},\mathscr{P}}) \ar@<0.5ex>[r]\ar@<0.5ex>[u]& A^1(\widetilde{\Sigma}_{\overline{\mathrm{M}},\mathscr{P}}). \ar@<0.5ex>[l] \ar@<0.5ex>[u]
}
\]
We remark that there are inclusion maps between the corresponding ample cones
\[
\xymatrix{
\mathscr{K}_{\mathrm{M},\mathscr{P}} \ar@{=}[r] \ar@{->}[d]& \mathscr{K}_{\overline{\mathrm{M}},\mathscr{P}} \ar@{->}[d] \\
\widetilde{\mathscr{K}}_{\mathrm{M},\mathscr{P}}  & \ar@{->}[l]\widetilde{\mathscr{K}}_{\overline{\mathrm{M}},\mathscr{P}}.
}
\]
In general, all three inclusions shown above may be strict. 
\end{remark}

\section{Homology and cohomology}

\subsection{}

Let $\Sigma$ be a unimodular fan in an $n$-dimensional latticed vector space $\mathbf{N}_\mathbb{R}$, and let $\Sigma_k$ be the set of $k$-dimensional cones in $\Sigma$.
If $\tau$ is a codimension $1$ face of a unimodular cone $\sigma$, we write
\[
\mathbf{e}_{\sigma/\tau}:=\text{the primitive generator of the unique $1$-dimensional face of $\sigma$ not in $\tau$}.
\]

\begin{definition}
A \emph{$k$-dimensional Minkowski weight} on $\Sigma$ is a function 
\[
\omega: \Sigma_k \longrightarrow \mathbb{Z}
\]
which satisfies the \emph{balancing condition}: For every $(k-1)$-dimensional cone $\tau$ in $\Sigma$,
\[
\sum_{\tau \subset \sigma} \omega(\sigma) \hspace{0.5mm} \mathbf{e}_{\sigma/\tau} \   \text{is contained in the subspace generated by $\tau$}.
\]
The \emph{group of Minkowski weights} on $\Sigma$ is the  group
\[
\text{MW}_*(\Sigma):=\bigoplus_{k \in \mathbb{Z}} \text{MW}_k(\Sigma),
\]
where
$\text{MW}_k(\Sigma):=\big\{\text{$k$-dimensional Minkowski weights on $\Sigma$}\big\} \subseteq \mathbb{Z}^{\Sigma_k}$.
\end{definition}

The group of Minkowski weights was studied by Fulton and Sturmfels  in the context of toric geometry  \cite{Fulton-Sturmfels}.
An equivalent notion of stress space was independently pursued by Lee in \cite{Lee}.
Both were inspired by McMullen, who introduced the notion of weights on polytopes and initiated the study of its algebraic properties \cite{PolytopeAlgebra,Weights}.
We record here some immediate properties of the group of Minkowski weights on $\Sigma$.

\begin{enumerate}[(1)]\itemsep 5pt
\item The group $\text{MW}_0(\Sigma)$ is canonically isomorphic to the group of integers:
\[
\text{MW}_0(\Sigma) =\mathbb{Z}^{\Sigma_0} \simeq \mathbb{Z}.
\]
\item The group $\text{MW}_1(\Sigma)$ is perpendicular to the image of the restriction map  from $\mathbf{M}$: 
\[
\text{MW}_1(\Sigma)=\text{im}(\text{res}_\Sigma)^\perp \subseteq \mathbb{Z}^{\Sigma_1}.
\]
\item  The group $\text{MW}_k(\Sigma)$ is trivial for $k$ negative or $k$ larger than the dimension of $\Sigma$.
\end{enumerate}

If $\Sigma$ is in addition complete, then an $n$-dimensional weight on  $\Sigma$ satisfies the balancing condition if and only if it is constant. 
Therefore, in this case, there is a canonical isomorphism
\[
\text{MW}_n(\Sigma) \simeq \mathbb{Z}.
\]
We show that the Bergman fan $\Sigma_{\mathrm{M}}$ has the same property with respect to its dimension $r$.

\begin{proposition}\label{PropositionMatroidBalancing}
An $r$-dimensional weight  on $\Sigma_\mathrm{M}$ satisfies the balancing condition if and only if it is constant.
\end{proposition}

It follows that there is a canonical isomorphism
$
\text{MW}_r(\Sigma_\mathrm{M}) \simeq \mathbb{Z}.
$
We begin the proof of Proposition~\ref{PropositionMatroidBalancing} with the following lemma.

\begin{lemma}\label{LemmaConnectedness}
The Bergman fan $\Sigma_\mathrm{M}$ is connected in codimension $1$. 
\end{lemma}

We remark that Lemma~\ref{LemmaConnectedness} is a direct consequence of the shellability of $\Delta_\mathrm{M}$, see \cite{Bjorner}.

\begin{proof}
The claim is that, for any two $r$-dimensional cones  $\sigma_\mathscr{F}, \sigma_\mathscr{G}$ in $\Sigma_\mathrm{M}$, there is a sequence 
\[
\sigma_\mathscr{F}= \sigma_0 \supset \tau_1 \subset \sigma_1 \supset \cdots  \subset \sigma_{l-1} \supset \tau_l \subset \sigma_l = \sigma_\mathscr{G},
\]
where $\tau_i$ is a common facet of $\sigma_{i-1}$ and $\sigma_i$ in $\Sigma_\mathrm{M}$.
We express this by writing $\sigma_\mathscr{F} \sim \sigma_\mathscr{G}$.

We prove by induction on the rank of $\mathrm{M}$. 
If $\text{min}\ \mathscr{F}=\text{min}\ \mathscr{G}$, then the induction hypothesis applied to $\mathrm{M}_{\text{min}\hspace{0.5mm} \mathscr{F}}$ shows that
\[
\sigma_\mathscr{F} \sim \sigma_{\mathscr{G}}.
\]
If otherwise,  we choose a  flag of nonempty proper flats $\mathscr{H}$ 
maximal among those satisfying $\text{min}\ \mathscr{F} \cup \text{min}\ \mathscr{G}<\mathscr{H}$.
By the induction hypothesis applied to  $\mathrm{M}_{\text{min}\hspace{0.5mm} \mathscr{F}}$, 
we have
\[
\sigma_\mathscr{F} \sim \sigma_{\{\text{min}\hspace{0.5mm} \mathscr{F}\} \cup \mathscr{H}}.
\]
Similarly, by the induction hypothesis applied to $\mathrm{M}_{\text{min}\hspace{0.5mm} \mathscr{G}}$, we have
\[
\sigma_\mathscr{G} \sim \sigma_{\{\text{min}\hspace{0.5mm} \mathscr{G}\} \cup \mathscr{H}}.
\]
Since any $1$-dimensional fan is connected in codimension $1$, this complete the induction.
\end{proof}

\begin{proof}[Proof of Proposition~\ref{PropositionMatroidBalancing}]
The proof is based on the \emph{flat partition property} for matroids $\mathrm{M}$ on $E$:
\[
\text{If $F$ is a flat of $\mathrm{M}$, then the flats of $\mathrm{M}$ that cover $F$ partition $E \setminus F$.}
\]
Let $\tau_\mathscr{G}$ be a codimension $1$ cone in the Bergman fan $\Sigma_\mathrm{M}$, and set
\[
V_{\text{star}(\mathscr{G})}:= \text{the set of primitive ray generators of the star of
 $\tau_\mathscr{G}$ in $\Sigma_\mathrm{M}$} \subseteq \mathbf{N}_{E,\mathbb{R}}/\langle \tau_\mathscr{F}\rangle.
 \]
The flat partition property applied to the restrictions of $\mathrm{M}$ shows that, first, 
the sum of all the vectors in $V_{\text{star}(\mathscr{G})}$ is zero and, second, any proper subset of $V_{\text{star}(\mathscr{G})}$ is linearly independent.
Therefore,  for an $r$-dimensional weight $\omega$ on $\Sigma_\mathrm{M}$,
\[
\text{$\omega$ satisfies the balancing condition at $\tau_\mathscr{G}$} \Longleftrightarrow
\text{$\omega$ is constant on cones containing $\tau_{\mathscr{G}}$.}
\]
By the connectedness of Lemma~\ref{LemmaConnectedness}, the latter condition 
for every $\tau_\mathscr{G}$ 
implies that $\omega$ is constant.
\end{proof}

\subsection{}

We continue to work with a unimodular fan $\Sigma$ in $\mathbf{N}_\mathbb{R}$.
As before, we write $V_\Sigma$ for the set of primitive ray generators of $\Sigma$.
Let $S_\Sigma$ be the polynomial ring over $\mathbb{Z}$ with variables indexed by $V_\Sigma$:
\[
S_\Sigma:=\mathbb{Z}[x_\mathbf{e} ]_{\mathbf{e} \in V_\Sigma}.
\]
For each $k$-dimensional cone $\sigma$ in $\Sigma$, we associate a degree $k$ square-free monomial 
\[
x_\sigma:=\prod_{\mathbf{e} \in \sigma} x_\mathbf{e}.
\]
The subgroup of $S_\Sigma$ generated by all such monomials $x_\sigma$ will be denoted
\[
Z^k(\Sigma):=\bigoplus_{\sigma \in \Sigma_k} \mathbb{Z}\hspace{0.5mm} x_\sigma.
\]
Let $Z^*(\Sigma)$ be the sum of $Z^k(\Sigma)$ over all nonnegative integers $k$.

\begin{definition}\label{ChowRingFan}
The \emph{Chow  ring} of $\Sigma$ is the commutative graded algebra
\[
A^*(\Sigma):=S_\Sigma/(I_\Sigma+J_\Sigma),
\]
where $I_\Sigma$ and $J_\Sigma$ are the ideals of $S_\Sigma$ defined by
\begin{align*}
I_\Sigma&:=\text{the ideal generated by the square-free monomials not in $Z^*(\Sigma)$,}\\
J_\Sigma&:=\text{the ideal generated by the linear forms $\sum_{\mathbf{e} \in V_\Sigma} \langle \mathbf{e}, m \rangle  \hspace{0.5mm}  x_\mathbf{e}$ for $m \in \mathbf{M}$.}
\end{align*}
We write $A^k(\Sigma)$ for the degree $k$ component of $A^*(\Sigma)$, and set 
\[
A^*(\Sigma)_\mathbb{R}:=A^*(\Sigma) \otimes_\mathbb{Z} \mathbb{R} \quad \text{and}  \quad A^k(\Sigma)_\mathbb{R}:=A^k(\Sigma)\otimes_\mathbb{Z} \mathbb{R}.
\]
\end{definition}

If we identify the variables of $S_\Sigma$ with the Courant functions on $\Sigma$,
then the degree $1$ component of  $A^*(\Sigma)$ agrees with the group introduced in Section~\ref{SectionPLFunctions}:
\[
A^1(\Sigma)=\text{PL}(\Sigma)/\mathbf{M}.
\]
Note that the pullback homomorphisms between $A^1$ introduced in that section uniquely extend to graded ring homomorphisms between $A^*$:
\begin{enumerate}[(1)]\itemsep 5pt
\item The homomorphism $\textrm{p}_{\Sigma' \subseteq \Sigma}$  uniquely extends to a surjective graded ring homomorphism
\[
\textrm{p}_{\Sigma' \subseteq \Sigma}: A^*(\Sigma) \longrightarrow A^*(\Sigma').
\]
\item The homomorphism $\textrm{p}_{\sigma \in \Sigma}$  uniquely extends to a surjective graded ring homomorphism
\[
\textrm{p}_{\sigma \in \Sigma}: A^*(\Sigma) \longrightarrow A^*(\text{star}(\sigma,\Sigma)).
\]
\item The isomorphism $\textrm{p}_{\Sigma_1,\Sigma_2}$ uniquely extends to a graded ring isomorphism
\[
\textrm{p}_{\Sigma_1,\Sigma_2}: A^*(\Sigma_1 \times \Sigma_2) \longrightarrow A^*(\Sigma_1) \otimes_\mathbb{Z} A^*(\Sigma_2).
\]
\end{enumerate}

We remark that the Chow ring $A^*(\Sigma)_\mathbb{R}$ can be identified with the ring of piecewise polynomial functions on $\Sigma$ modulo linear functions on $\mathbf{N}_\mathbb{R}$, see \cite{Billera}.

\begin{proposition}\label{PropositionGeneration}
The group $A^k(\Sigma)$ is generated by $Z^k(\Sigma)$ for each nonnegative integer $k$.
\end{proposition}

In particular, if $k$ larger than the dimension of $\Sigma$, then $A^k(\Sigma)=0$.

\begin{proof}
Let $\sigma$ be a cone in $\Sigma$,  let $\mathbf{e}_1,\mathbf{e}_2,\ldots,\mathbf{e}_l$ be its primitive ray generators.
and consider a degree $k$ monomial of the form
\[
x_{\mathbf{e}_1}^{k_1}x_{\mathbf{e}_2}^{k_2} \cdots x_{\mathbf{e}_l}^{k_l}, \qquad k_1 \ge k_2 \ge \cdots \ge k_l \ge 1.
\]
We show that the image of this monomial in $A^k(\Sigma)$ is in the span of $Z^k(\Sigma)$.

We do this by descending induction on the dimension of $\sigma$.
If $\text{dim}\ \sigma=k$, there is nothing to prove. 
If otherwise, we use the unimodularity of $\sigma$ to choose $m \in \mathbf{M}$ such that 
\[
\langle \mathbf{e}_1,m\rangle=-1 \ \ \text{and} \ \ \langle \mathbf{e}_2,m\rangle=\cdots=\langle \mathbf{e}_{l},m\rangle=0.
\]
This shows that, modulo the relations given by $I_\Sigma$ and $J_\Sigma$, we have
\[
x_{\mathbf{e}_1}^{k_1}x_{\mathbf{e}_2}^{k_2} \cdots x_{\mathbf{e}_l}^{k_l}=x_{\mathbf{e}_1}^{k_1-1} x_{\mathbf{e}_2}^{k_2}\cdots x_{\mathbf{e}_l}^{k_l}\ \sum_{\mathbf{e} \in \text{link}(\sigma)} \langle \mathbf{e},m\rangle \hspace{0.5mm}x_\mathbf{e},
\]
where the sum is over the set of primitive ray generators of the link of $\sigma$ in $\Sigma$.
The induction hypothesis applies to each of the terms in the expansion of the right-hand side.
\end{proof}

The group of $k$-dimensional weights on $\Sigma$ can be identified with the dual of $Z^k(\Sigma)$ under the tautological isomorphism
\[
\mathrm{t}_\Sigma: \mathbb{Z}^{\Sigma_k} \longrightarrow \text{Hom}_\mathbb{Z}(Z^k(\Sigma),\mathbb{Z}), \qquad
\omega \longmapsto \Big( x_\sigma \longmapsto \omega(\sigma)\Big).
\]
By Proposition~\ref{PropositionGeneration}, the target of $\mathrm{t}_\Sigma$ contains 
$\text{Hom}_\mathbb{Z}(A^k(\Sigma),\mathbb{Z})$ as a subgroup.

\begin{proposition}\label{PropositionBasicDuality}
The isomorphism $\mathrm{t}_\Sigma$ restricts to the bijection between the subgroups
\[
  \text{MW}_k(\Sigma) \longrightarrow  \text{Hom}_\mathbb{Z}(A^k(\Sigma),\mathbb{Z}). 
\]
\end{proposition}

The bijection in Proposition~\ref{PropositionBasicDuality}  is an analogue of the Kronecker duality homomorphism in algebraic topology.
We use it to define the \emph{cap product} 
\[
A^l(\Sigma) \times \text{MW}_k(\Sigma) \longrightarrow \text{MW}_{k-l}(\Sigma), \qquad \xi \cap \omega \hspace{0.5mm}(\sigma) := \mathrm{t}_\Sigma\hspace{0.5mm} \omega \hspace{0.5mm}(\xi \cdot x_\sigma).
\]
This makes the group $\text{MW}_*(\Sigma)$ a graded module over the Chow ring $A^*(\Sigma)$.

\begin{proof}
The homomorphisms from $A^k(\Sigma)$ to $\mathbb{Z}$ bijectively correspond to the homomorphisms from  $Z^k(\Sigma)$ to $\mathbb{Z}$ which vanish on the subgroup
\[
 Z^k(\Sigma)   \cap  (I_\Sigma+J_\Sigma) \subseteq Z^k(\Sigma).
\]
The main point is that this subgroup is generated by polynomials of the form
\[
 \Bigg( \sum_{\mathbf{e} \in \text{link}(\tau)} \langle \mathbf{e}, m \rangle  \hspace{0.5mm}  x_\mathbf{e}\Bigg) x_\tau, 
\]
where $\tau$ is a $(k-1)$-dimensional cone of $\Sigma$ and $m$ is an element perpendicular to $\langle \tau \rangle$. 
This is a special case of \cite[Theorem 1]{FMSS}.
It follows that
a $k$-dimensional weight $\omega$ corresponds to a homomorphism $A^k(\Sigma) \to \mathbb{Z}$ if and only if
\[
 \sum_{\tau \subset \sigma} \omega(\sigma)\hspace{0.5mm} \langle \mathbf{e}_{\sigma/\tau}, m \rangle =0 \ \text{for all $m \in \langle \tau \rangle^\perp$},
\]
where the sum is over all $k$-dimensional cones $\sigma$ in $\Sigma$ containing $\tau$.
Since $\langle\tau \rangle^{\perp\perp}=\langle \tau \rangle$, the latter condition is equivalent to the balancing condition on $\omega$ at $\tau$.
\end{proof}

\subsection{}\label{SectionDegreeMap}

The ideals $I_\Sigma$ and $J_\Sigma$ have a particularly simple description when $\Sigma=\Sigma_\mathrm{M}$.
In this case, we label the variables of  $S_\Sigma$  by the nonempty proper flats of $\mathrm{M}$, and write
\[
S_{\Sigma}=\mathbb{Z}[x_F]_{F \in \mathscr{P}(\mathrm{M})}.
\]
For a flag  of nonempty proper flats $\mathscr{F}$, we set
$
x_\mathscr{F}= \prod_{F \in \mathscr{F}} x_F.
$

\begin{enumerate}[(1)]\itemsep 5pt
\item[(Incomparability relations)] The ideal $I_\Sigma$ is generated by the  quadratic monomials 
\[
x_{F_1}x_{F_2},
\]
where $F_1$ and $F_2$ are two incomparable nonempty proper flats  of $\mathrm{M}$.
\item[(Linear relations)] The ideal $J_\Sigma$ is generated by the linear forms 
\[
\sum_{i_1 \in F} x_F - \sum_{i_2 \in F} x_F,
\]
where $i_1$ and $i_2$ are  distinct elements of the ground set $E$.
\end{enumerate}
The quotient ring $A^*(\Sigma_\mathrm{M})$ and its generalizations were studied  by Feichtner and Yuzvinsky in \cite{Feichtner-Yuzvinsky}.

\begin{definition}\label{DefinitionAlphaBeta}
To an element $i$ in $E$, we associate linear forms
\[
\alpha_{\mathrm{M},i}:=\sum_{i \in F} x_F, \quad \beta_{\mathrm{M},i}:=\sum_{i \notin F} x_F.
\]
Their classes  in $A^*(\Sigma_\mathrm{M})$, which are independent of $i$,  will be written $\alpha_\mathrm{M}$ and $\beta_\mathrm{M}$ respectively.
\end{definition}

We show that $A^r(\Sigma_\mathrm{M})$ is generated by the  element $\alpha_\mathrm{M}^r$, where $r$ is the dimension of $\Sigma_\mathrm{M}$.

\begin{proposition}\label{PropositionFundamentalClass}
Let $F_1 \subsetneq F_2 \subsetneq \cdots \subsetneq F_k$ be any flag of nonempty proper flats of $\mathrm{M}$.
\begin{enumerate}[(1)]\itemsep 5pt
\item If the rank of $F_m$ is not $m$ for some  $m \le k$, then
\[
x_{F_1}x_{F_2} \cdots x_{F_{k}} \hspace{0.5mm}\alpha_\mathrm{M}^{r-k}=\hspace{0.5mm} 0 \hspace{0.5mm}\in A^r(\Sigma_\mathrm{M}).
\]
\item If the rank of $F_m$ is $m$ for all  $m \le k$, then 
\[
x_{F_1}x_{F_2} \cdots x_{F_{k}} \hspace{0.5mm}\alpha_\mathrm{M}^{r-k}=\alpha_\mathrm{M}^r \in A^r(\Sigma_\mathrm{M}).
\]
\end{enumerate}
\end{proposition}

In particular, for any two maximal flags of nonempty proper flats $\mathscr{F}_1$ and $\mathscr{F}_2$ of $\mathrm{M}$,
 \[
x_{\mathscr{F}_1}=x_{\mathscr{F}_2} \in A^r(\Sigma_\mathrm{M}).
 \]
Since $\text{MW}_r(\Sigma_\mathrm{M})$ is isomorphic to $\mathbb{Z}$, this implies that
$A^r(\Sigma_\mathrm{M})$ is  isomorphic to $\mathbb{Z}$, see Proposition~\ref{PropositionDegreeMap}.

\begin{proof}
As a general observation, we note that for any element $i$ not in a nonempty proper flat $F$,
\[
x_F \ \alpha_\mathrm{M}= x_F  \ \Big( \sum_G x_G\Big) \in A^*(\Sigma_\mathrm{M}),
\]
where the sum is over all proper flats containing $F$ and  $\{i\}$. 
In particular, if the rank of $F$ is $r$, then the product is zero.

We prove the first assertion by descending induction on $k$, which is necessarily less than $r$.
If $k=r-1$, then the rank of $F_k$ should be $r$, and hence the product is zero.
For general $k$, we choose an element $i$ not in  $F_k$.
By the observation made above, we have
\[
x_{F_1}x_{F_2} \cdots x_{F_{k}} \hspace{0.5mm}\alpha_\mathrm{M}^{r-k}=
x_{F_1}x_{F_2} \cdots x_{F_{k}} \hspace{0.5mm} \Big(\sum_{G} x_G\Big) \hspace{0.5mm}\alpha_\mathrm{M}^{r-k-1},
\]
where the sum is over all proper flats  containing $F_k$ and $\{i\}$. 
The right-hand side is zero by the induction hypothesis for $k+1$ applied to each of the terms in the expansion.

We prove the second assertion by ascending induction on $k$. 
When $k=1$, we choose an element $i$ in $F_1$, and consider the corresponding representative of $\alpha_{\mathrm{M}}$ to write
\[
\alpha_\mathrm{M}^{r}=\Big(\sum_G x_G\Big)\alpha_\mathrm{M}^{r-1},
\]
where the sum is over all  proper flats containing $i$.
By the first part of the proposition for $k=1$, only one term in the expansion of the right-hand side is nonzero, and this gives
\[
\alpha_\mathrm{M}^{r}=x_{F_1} \hspace{0.5mm}\alpha_\mathrm{M}^{r-1}.
\]
For general $k$, we start from the induction hypothesis
\[
\alpha_\mathrm{M}^{r}=x_{F_1} x_{F_2}\cdots x_{F_{k-1}}  \alpha_\mathrm{M}^{r-(k-1)}.
\]
Choose an element $i$ in $F_k \setminus F_{k-1}$ and use the general observation made above to write
\[
\alpha_\mathrm{M}^{r}=x_{F_1} x_{F_2}\cdots x_{F_{k-1}}  \Big(\sum_G x_G\Big)\alpha_\mathrm{M}^{r-k},
\]
where the sum is over all proper flats containing $F_{k-1}$ and $\{i\}$.
By the first part of the proposition for $k$, only one term in the expansion of the right-hand side is nonzero, and we get
\[
\alpha_\mathrm{M}^{r}=x_{F_1} x_{F_2}\cdots x_{F_{k-1}}  x_{F_k}\alpha_\mathrm{M}^{r-k}.
\]
\end{proof}

When $\Sigma$ is complete, Fulton and Sturmfels  showed in \cite{Fulton-Sturmfels} that there is an isomorphism
\[
A^{k}(\Sigma) \longrightarrow \text{MW}_{n-k}(\Sigma), \qquad \xi \longmapsto \big( \sigma \longmapsto \text{deg} \ \xi \cdot x_\sigma \big),
\]
where $n$ is the dimension of $\Sigma$ and  $``\text{deg}"$ is the degree map of the complete toric variety of $\Sigma$.
In Theorem~\ref{PoincareDuality}, we show that there is an  isomorphism  for the Bergman fan
\[
A^{k}(\Sigma_{\mathrm{M}}) \longrightarrow \text{MW}_{r-k}(\Sigma_{\mathrm{M}}), \qquad \xi \longmapsto \big( \sigma_\mathscr{F} \longmapsto \text{deg} \ \xi \cdot x_\mathscr{F} \big),
\]
where $r$ is the dimension of $\Sigma_\mathrm{M}$ and
$``\text{deg}"$ is a homomorphism constructed from $\mathrm{M}$.
These isomorphisms are analogues of the Poincar\'e duality homomorphism in algebraic topology.

\begin{definition}\label{MatroidDegreeMap}
The \emph{degree map} of $\mathrm{M}$ is the homomorphism
obtained by taking the cap product 
\[
\text{deg}: A^r(\Sigma_\mathrm{M}) \longrightarrow \mathbb{Z}, \qquad \xi \longmapsto \xi \cap 1_\mathrm{M},
\]
where $1_\mathrm{M}=1$ is the constant $r$-dimensional Minkowski weight on  $\Sigma_\mathrm{M}$.
\end{definition}

By Proposition~\ref{PropositionGeneration},  the homomorphism $\text{deg}$ is uniquely determined by its property
\[
\text{deg}(x_\mathscr{F})=1 \ \ \text{for all monomials $x_\mathscr{F}$ corresponding to an $r$-dimensional cone in $\Sigma_\mathrm{M}$.}
\]

\begin{proposition}\label{PropositionDegreeMap}
The degree map of $\mathrm{M}$ is an isomorphism.
\end{proposition}

\begin{proof}
The second part of Proposition~\ref{PropositionFundamentalClass} shows that 
$A^r(\Sigma_\mathrm{M})$ is generated by the element $\alpha_\mathrm{M}^r$,
and that  $\text{deg} (\alpha_\mathrm{M}^r)=\text{deg}(x_\mathscr{F})=1$.
\end{proof}

\subsection{}

We remark on algebraic geometric properties of Bergman fans,
working over a fixed field $\mathbb{K}$.
For basics on toric varieties, we refer to \cite{Fulton}.
The results of this subsection will be independent from the remainder of the paper.

The main object is the  smooth toric variety $X(\Sigma)$  over $\mathbb{K}$ associated to a unimodular fan $\Sigma$ in $\mathbf{N}_\mathbb{R}$:
\[
X(\Sigma):=\bigcup_{\sigma \in \Sigma} \text{Spec} \ \mathbb{K}[\sigma^\vee \cap M].
\]
It is known that the Chow ring of $\Sigma$ is naturally isomorphic to the Chow  ring of $X(\Sigma)$:
\[
A^*(\Sigma) \longrightarrow A^*(X({\Sigma})), \qquad x_\sigma \longmapsto [X({\text{star}(\sigma)})].
\]
 See \cite[Section 10]{Danilov} for the proof when $\Sigma$ is complete, and see \cite{Bifet-DeConcini-Procesi} and \cite{Brion} for the general case.

\begin{definition}
A morphism between smooth algebraic varieties $X_1 \to X_2$ is a \emph{Chow equivalence}
if the induced homomorphism between the Chow rings $A^*(X_2) \to A^*(X_1)$ is an isomorphism.
\end{definition}

In fact, the results of this subsection will be valid for any variety that is locally a quotient of a manifold by a finite group so that $A^*(X) \otimes_\mathbb{Z}\mathbb{Q}$ has the structure of a graded algebra over $\mathbb{Q}$.
Matroids provide nontrivial examples of Chow equivalences. 
For example, consider the subfan
${\widetilde{\Sigma}_{\mathrm{M},\mathscr{P}}}\subseteq {\Sigma_{\mathrm{M},\mathscr{P}}}$
and the corresponding open subset
\[
X({\widetilde{\Sigma}_{\mathrm{M},\mathscr{P}}})\subseteq X({\Sigma_{\mathrm{M},\mathscr{P}}}).
\]
In Proposition~\ref{PropositionOpenInclusion}, we show that
 the above inclusion is a Chow equivalence for any $\mathrm{M}$ and  $\mathscr{P}$.

We remark that, when  $\mathbb{K}=\mathbb{C}$,  a Chow equivalence need not induce an isomorphism between  singular cohomology rings. 
For example, consider any  line in a projective plane minus two points
\[
\mathbb{C}\mathbb{P}^1 \subseteq \mathbb{C}\mathbb{P}^2 \setminus \{p_1,p_2\}.
\]
The inclusion is a Chow equivalence for any two distinct points $p_1,p_2$ outside $\mathbb{C}\mathbb{P}^1$,
but the two spaces have different singular cohomology rings.

We show that the notion of Chow equivalence can be used to characterize the realizability of matroids.

\begin{theorem}\label{Chow-equivalence}
There is a Chow equivalence from a smooth projective variety over $\mathbb{K}$ to $X(\Sigma_{\mathrm{M}})$ if and only if the matroid $\mathrm{M}$ is realizable over  $\mathbb{K}$.
\end{theorem}

\begin{proof}
This is a classical variant of the tropical characterization of the realizability of matroids in \cite{Katz-Payne}. We write $r$ for the dimension of $\Sigma_\mathrm{M}$, and $n$ for the dimension of  $X({\Sigma_\mathrm{M}})$. 
As before, the ground set of  $\mathrm{M}$ will be $E=\{0,1,\ldots,n\}$.

The ``if" direction follows from the construction of De Concini-Procesi wonderful models \cite{DeConcini-Procesi}.
Suppose that the loopless matroid $\mathrm{M}$ is realized by a spanning set of nonzero vectors
\[
\mathscr{R}=\{f_0,f_1,\ldots,f_n\} \subseteq V/\mathbb{K}.
\]
The realization $\mathscr{R}$ gives an injective linear map between two projective spaces
\[
L_\mathscr{R}: \mathbb{P}(V^\vee) \longrightarrow X(\Sigma_\varnothing), \qquad L_\mathscr{R}=[f_0:f_1:\cdots:f_n],
\]
where $\Sigma_\varnothing$ is the complete fan in $\mathbf{N}_{E,\mathbb{R}}$ corresponding to the empty order filter of $\mathscr{P}(E)$.
Note that the normal fan of the $n$-dimensional permutohedron $\Sigma_{\mathscr{P}(E)}$ can be obtained from  the normal fan of the $n$-dimensional simplex $\Sigma_\varnothing$ by performing a sequence of  stellar subdivisions.
In other words, there is a morphism between toric varieties
\[
\pi: X(\Sigma_{\mathscr{P}(E)}) \longrightarrow X(\Sigma_\varnothing),
\]
which is the composition of blowups of torus-invariant subvarieties.
To be explicit, consider a sequence of order filters of $\mathscr{P}(E)$  obtained by adding a single subset at a time:
\[
\varnothing,\ldots, \mathscr{P}_{-},\mathscr{P}_+,\ldots,\mathscr{P}(E) \quad \text{with} \quad \mathscr{P}_+=\mathscr{P}_- \cup \{Z\}.
\]
The corresponding sequence of $\Sigma$ interpolates between the collections $\Sigma_\varnothing$ and $\Sigma_{\mathscr{P}(E)}$:
\[
\Sigma_\varnothing \rightsquigarrow \ldots \rightsquigarrow    \Sigma_{\mathscr{P}_{-}}\rightsquigarrow \Sigma_{\mathscr{P}_{+}} \rightsquigarrow   \ldots \rightsquigarrow  \Sigma_{\mathscr{P}(E)}.
\]
The modification in the middle replaces the cones of the form $\sigma_{Z<\mathscr{F}}$  with the sums of the form
\[
\sigma_{\varnothing<\{Z\}}+\sigma_{I<\mathscr{F}}, 
\]
where $I$ is any proper subset of $Z$.
The wonderful model $Y_\mathscr{R}$ associated to $\mathscr{R}$ is by definition the strict transform of $\mathbb{P}(V^\vee)$ under the composition of toric blowups $\pi$.
The torus-invariant prime divisors of $X(\Sigma_{\mathscr{P}(E)})$ correspond to nonempty proper subsets of $E$,
and those divisors intersecting $Y_\mathscr{R}$ exactly correspond to nonempty proper flats of $\mathrm{M}$.
Therefore, the smooth projective variety $Y_\mathscr{R}$  is contained in the open subset
\[
X(\Sigma_\mathrm{M}) \subseteq X(\Sigma_{\mathscr{P}(E)}).
\]
The inclusion $Y_\mathscr{R} \subseteq X(\Sigma_\mathrm{M})$ is a Chow equivalence \cite[Corollary 2]{Feichtner-Yuzvinsky}.

The ``only if'' direction follows from  computations in $A^*(\Sigma_\mathrm{M})$ made in the previous subsection.
Suppose that there is a Chow equivalence from a smooth projective variety 
\[
f: Y \longrightarrow X({\Sigma_\mathrm{M}}).
\]
Proposition~\ref{PropositionGeneration} and Proposition~\ref{PropositionDegreeMap} show that
\[
A^r(Y)  \simeq A^r(\Sigma_\mathrm{M})\simeq \mathbb{Z} \quad \text{and} \quad A^k(Y) \simeq A^k(\Sigma_\mathrm{M}) \simeq 0 \ \ \text{for all $k$ larger than $r$.}
\] 
Since $Y$ is complete, the above implies that the dimension of $Y$ is $r$.
Let $g$ be the composition
\[
\xymatrixcolsep{2.5pc}
\xymatrix{
Y \ar[r]^<<<<<{f}  & X(\Sigma_\mathrm{M}) \ar[r]^<<<<<{\pi_\mathrm{M}}
&X(\Sigma_\varnothing) \simeq \mathbb{P}^n,
}
\]
where $\pi_\mathrm{M}$ is the restriction of the composition of toric blowups  $\pi$.
We use Proposition~\ref{PropositionFundamentalClass} to compute the degree of the image $g(Y)\subseteq \mathbb{P}^n$.

For this we note that,  for any element $i \in E$, we have
\[
\pi_\mathrm{M}^{-1} \{z_i=0\}=\bigcup_{i \in F} D_F,
\]
where $z_i$ is the homogeneous coordinate of $\mathbb{P}^n$ corresponding to $i$
and $D_F$ is the torus-invariant prime divisor of $X(\Sigma_\mathrm{M})$ corresponding to a  nonempty proper flat $F$.
All the components of $\pi_\mathrm{M}^{-1} \{z_i=0\}$ appear with multiplicity $1$, and hence
\[
\pi_\mathrm{M}^*\ \mathcal{O}_{\mathbb{P}^n}(1)=\alpha_\mathrm{M} \in A^1(\Sigma_\mathrm{M}).
\]
Hence, under the isomorphism $f^*$ between the Chow rings, the $0$-dimensional cycle $(g^* \mathcal{O}_{\mathbb{P}^n}(1))^r$ is the image of the generator
\[
(\pi_\mathrm{M}^{*} \ \mathcal{O}_{\mathbb{P}^n}(1))^r=\alpha_\mathrm{M}^r \in A^r(\Sigma_\mathrm{M})\simeq  \mathbb{Z}.
\]
By the projection formula, the above implies that the degree of the image of $Y$ in $\mathbb{P}^n$ is $1$. 
In other words, $g(Y)\subseteq \mathbb{P}^n$ is an $r$-dimensional linear subspace defined over $\mathbb{K}$.
We  express the inclusion  in the form
\[
L_\mathscr{R}: \mathbb{P}(V^\vee) \longrightarrow \mathbb{P}^n, \qquad L_\mathscr{R}=[f_0:f_1:\cdots:f_n].
\]

Let $\mathrm{M}'$ be the loopless matroid on $E$ defined by the set of nonzero vectors $\mathscr{R} \subseteq V/\mathbb{K}$.
The image of $Y$ in $X(\Sigma_\mathrm{M})$ is the wonderful model $Y_\mathscr{R}$, and hence
\[
X(\Sigma_{\mathrm{M}'}) \subseteq X(\Sigma_{\mathrm{M}}).
\]
Observe that none of the torus-invariant prime divisors of $X(\Sigma_{\mathrm{M}})$ are rationally equivalent to zero.
Since $f$ is a Chow equivalence,
the observation implies that the torus-invariant prime divisors of
$X(\Sigma_{\mathrm{M}'})$  and $X(\Sigma_{\mathrm{M}})$
bijectively correspond to each other.
Since a matroid is determined by its set of nonempty proper flats, this shows that $\mathrm{M}=\mathrm{M}'$.
\end{proof}

\section{Poincar\'e duality for matroids}\label{SectionDecompositionTheorem}

\subsection{}

The principal result of this section is an analogue of Poincar\'e duality for $A^*(\Sigma_{\mathrm{M},\mathscr{P}})$, see Theorem~\ref{PoincareDuality}.
We give an alternative description of the Chow ring  suitable for this purpose.

\begin{definition}\label{DefinitionIntermediateChow}
Let $S_{E \cup \mathscr{P}}$ be the polynomial ring  over $\mathbb{Z}$ with variables indexed by $E \cup \mathscr{P}$:
\[
S_{E \cup \mathscr{P}}:=\mathbb{Z}[x_i,x_F]_{i \in E,F\in \mathscr{P}}.
\]
The \emph{Chow ring} of  $(\mathrm{M},\mathscr{P})$ is the commutative graded algebra
\[
A^*(\mathrm{M},\mathscr{P}):=S_{E \cup\mathscr{P}}/(\mathscr{I}_1+\mathscr{I}_2+\mathscr{I}_3+\mathscr{I}_4),
\]
where $\mathscr{I}_1$, $\mathscr{I}_2$, $\mathscr{I}_3$, $\mathscr{I}_4$ are the ideals of $S_{E \cup \mathscr{P}}$ defined below.
\begin{enumerate}[(1)]\itemsep 5pt
\item[(Incomparability relations)] The ideal $\mathscr{I}_{1}$ is generated by the  quadratic monomials 
\[
x_{F_1}x_{F_2},
\]
where $F_1$ and $F_2$ are two incomparable flats  in the order filter $\mathscr{P}$.
\item[(Complement relations)] The ideal $\mathscr{I}_2$ is generated by the  quadratic monomials 
\[
x_{i}\hspace{0.5mm}x_{F},
\]
where $F$ is a flat in the order filter $\mathscr{P}$ and $i$ is an element in the complement $E \setminus F$.
\item[(Closure relations)] The ideal $\mathscr{I}_3$ is generated by the monomials
\[
\prod_{i \in I} x_i,
\]
where $I$ is an independent set  of $\mathrm{M}$ whose closure is in $\mathscr{P} \cup \{E\}$.
\item[(Linear relations)] The ideal $\mathscr{I}_4$ is generated by the linear forms 
\[
\Big(x_{i}+\sum_{i \in F} x_F\Big) -\Big(x_{j}+ \sum_{j \in F} x_F\Big),
\]
where $i$ and $j$ are  distinct elements of $E$ and the sums are over flats $F$ in $\mathscr{P}$.
\end{enumerate}
When $\mathscr{P}=\mathscr{P}(\mathrm{M})$, we omit $\mathscr{P}$ from the notation and write the Chow ring by $A^*(\mathrm{M})$.
\end{definition}

When $\mathscr{P}$ is empty, the relations in $\mathscr{I}_4$ show that all $x_i$ are equal in the Chow ring, and hence
\[
A^*(\mathrm{M},\varnothing) \simeq \mathbb{Z}[x]/(x^{r+1}).
\]
When $\mathscr{P}$ is $\mathscr{P}(\mathrm{M})$, the relations in $\mathscr{I}_3$ show that  all  $x_i$ are zero in the Chow ring, and hence
\[
A^*(\mathrm{M}) \simeq A^*(\Sigma_{\mathrm{M}}).
\]
In general, if $i$ is an element whose closure is in $\mathscr{P}$, then  $x_i$ is zero in the Chow ring. 
The square-free monomial relations in the remaining set of variables correspond bijectively to the non-faces of the Bergman complex $\Delta_{\mathrm{M},\mathscr{P}}$, and hence
\[
A^*(\mathrm{M},\mathscr{P})\simeq A^*(\Sigma_{\mathrm{M},\mathscr{P}}).
\]
More precisely, in the notation of Definitions \ref{ChowRingFan} and \ref{DefinitionIntermediateChow}, for $\Sigma=\Sigma_{\mathrm{M},\mathscr{P}}$ we have
\[
\mathscr{I}_1+\mathscr{I}_2+\mathscr{I}_3=I_{\Sigma} \quad \text{and} \quad \mathscr{I}_4=J_{\Sigma}.
\]

We show that the Chow ring of $(\mathrm{M},\mathscr{P})$ is also isomorphic to the Chow ring of 
the reduced Bergman fan $\widetilde{\Sigma}_{\mathrm{M},\mathscr{P}}$.

\begin{proposition}\label{PropositionOpenInclusion}
Let $I$ be a subset of $E$, and let $F$ be a flat in an order filter $\mathscr{P}$ of $\mathscr{P}(\mathrm{M})$.
\begin{enumerate}[(1)] \itemsep 5pt
\item If $I$ has  cardinality at least the rank of $F$, then
\[
\Big( \prod_{i \in I} x_i\Big) x_F =0 \in A^*(\mathrm{M},\mathscr{P}).
\]
\item If $I$ has cardinality at least $r+1$, then
\[
\prod_{i \in I} x_i=0 \in A^*(\mathrm{M},\mathscr{P}).
\]
\end{enumerate}
\end{proposition}

In other words, the inclusion of the open subset
$X({\widetilde{\Sigma}_{\mathrm{M},\mathscr{P}}})\subseteq X({\Sigma_{\mathrm{M},\mathscr{P}}})$
 is a Chow equivalence.
Since the reduced Bergman fan  has dimension $r$, this implies that
\[
A^k(\mathrm{M},\mathscr{P})=0 \ \ \text{for $k>r$.}
\]

\begin{proof}
For the first assertion, we use complement relations in $\mathscr{I}_2$ to reduce to the case when $I \subseteq F$.
We prove by induction on the difference between the rank of $F$ and the rank of $I$.

When the difference is zero,   $I$ contains a basis of $F$, and the desired vanishing follows from a closure relation in $\mathscr{I}_3$.
When the difference is positive, we choose a subset $J \subseteq F$ with
\[
 \text{rk}(J)=\text{rk}(I)+1, \quad I \setminus J=\{i\} \ \ \text{and} \ \ J \setminus I=\{j\}.
\]
From the linear relation in $\mathscr{I}_4$ for $i$ and $j$, we deduce that
\[
x_{i}+\sum_{\substack{i \in G \\ j \notin G}} x_G=
x_{j}+\sum_{\substack{j \in G \\ i \notin G}} x_G,
\]
where the sums are over flats $G$ in $\mathscr{P}$.
Multiplying both sides by 
$\Big(\prod_{i \in I \cap J} x_i \Big) x_F$,
we get
\[
\Big( \prod_{i \in I} x_i\Big) x_F=\Big( \prod_{j \in J} x_j\Big) x_F.
\]
Indeed, a term involving $x_G$ in the expansions of the products is zero in the Chow ring by
\begin{enumerate}[(1)]\itemsep 5pt
\item  an incomparability relation  in $\mathscr{I}_1$, if $G \nsubseteq F$,
\item  a complement relation in  $\mathscr{I}_2$, if $I \cap J \nsubseteq G$, 
\item  the induction hypothesis for  $I \cap J \subseteq G$, if otherwise. 
\end{enumerate}
The right-hand side of the equality is zero
by the induction hypothesis for $J \subseteq F$.

The second assertion can be proved in the same way, by descending induction on the rank of $I$, using the first part of the proposition.
\end{proof}

We record here that the isomorphism of Proposition~\ref{PropositionSimplification} uniquely extends  to an isomorphism between the corresponding Chow rings.

\begin{proposition}
The homomorphism $\pi_{\text{PL}}$ induces an isomorphism of graded rings
\[
\underline{\pi}_{\text{PL}}:  A^*(\mathrm{M},\mathscr{P}) \longrightarrow A^*(\overline{\mathrm{M}},\mathscr{P}).
\]
The homomorphism $\iota_{\text{PL}}$ induces the inverse isomorphism of graded rings
\[
\underline{\iota}_{\text{PL}}:  A^*(\overline{\mathrm{M}},\mathscr{P}) \longrightarrow A^*(\mathrm{M},\mathscr{P}).
\]
\end{proposition}

\begin{proof}
Consider the extensions of  $\pi_{\text{PL}}$ and $\iota_{\text{PL}}$ to the polynomial rings 
\[
\xymatrix{
S_{E \cup \mathscr{P}} \ar@<0.5ex>[rr]^{\widetilde{\pi}_\text{PL}} &&S_{\overline{E} \cup \mathscr{P}}\ar@<0.5ex>[ll]^{\widetilde{\iota}_\text{PL}}.
}
\]
The result follows from the observation that $\widetilde{\pi}_{\text{PL}}$ and $\widetilde{\iota}_{\text{PL}}$ preserve the monomial relations in $\mathscr{I}_1$, $\mathscr{I}_2$, and $\mathscr{I}_3$.
\end{proof}

\subsection{}\label{SectionMatroidalFlip}

Let $\mathscr{P}_-$ be an order filter of $\mathscr{P}(\mathrm{M})$, and let $Z$ be a flat  maximal in  $\mathscr{P}(\mathrm{M}) \setminus \mathscr{P}_-$.
We set
\[
\mathscr{P}_+:=\mathscr{P}_- \cup \{Z\} \subseteq \mathscr{P}(\mathrm{M}).
\]
The collection $\mathscr{P}_+$ is an order filter of $\mathscr{P}(\mathrm{M})$.
 
\begin{definition}
The \emph{matroidal flip} from $\mathscr{P}_-$ to $\mathscr{P}_+$ is the modification of fans $ \Sigma_{\mathrm{M},\mathscr{P}_{-}} \hspace{-1.5mm}\rightsquigarrow \Sigma_{\mathrm{M},\mathscr{P}_{+}}$.
\end{definition}

The  flat $Z$  will be called the \emph{center} of the matroidal flip. 
The matroidal flip removes the  cones  
\[
\sigma_{I < \mathscr{F}} \ \ \text{with} \ \  \text{cl}_\mathrm{M}(I)=Z  \ \ \text{and} \ \ \text{min}\ \mathscr{F}\neq Z,
\]
and replaces them with the cones 
\[
\sigma_{I < \mathscr{F}} \ \ \text{with} \ \ \text{cl}_\mathrm{M}(I) \neq Z \ \ \text{and} \ \ \text{min}\ \mathscr{F}=Z.
\]
The center $Z$ is necessarily minimal in $\mathscr{P}_+$, and we have
\begin{align*}
&\text{star}(\ \sigma_{Z<\varnothing}\ ,\ \Sigma_{\mathrm{M},\mathscr{P}_-}) \simeq \Sigma_{\mathrm{M}_Z}, \\
&\text{star}(\sigma_{\varnothing<\{Z\}},\Sigma_{\mathrm{M},\mathscr{P}_+}) \simeq \Sigma_{\mathrm{M}^Z,\varnothing} \times \Sigma_{\mathrm{M}_Z}.
\end{align*}

\begin{remark}
The matroidal flip preserves the homotopy type of the underlying simplicial complexes $\Delta_{\mathrm{M},\mathscr{P}_-}$ and $\Delta_{\mathrm{M},\mathscr{P}_+}$.
To see this, consider the inclusion 
\[
\Delta_{\mathrm{M},\mathscr{P}_+} \subseteq \Delta^{*}_{\mathrm{M},\mathscr{P}_-}:=\text{the stellar subdivision of $\Delta_{\mathrm{M},\mathscr{P}_-}$ relative to  $\vartriangle_{Z<\varnothing}$}.
\]
We claim that the left-hand side is a deformation retract of the right-hand side.
More precisely, there is a sequence of compositions of elementary collapses
\begin{multline*}
\Delta^{*}_{\mathrm{M},\mathscr{P}_-} =\Delta^{1,1}_{\mathrm{M},\mathscr{P}_-} \ \rightsquigarrow \ \Delta^{1,2}_{\mathrm{M},\mathscr{P}_-}\  \rightsquigarrow \   \cdots \ \rightsquigarrow\ \Delta^{1,\text{crk}(Z)-1}_{\mathrm{M},\mathscr{P}_-} \ \rightsquigarrow \ \\
\Delta^{1,\text{crk}(Z)}_{\mathrm{M},\mathscr{P}_-} =\Delta^{2,1}_{\mathrm{M},\mathscr{P}_-} \ \rightsquigarrow\ \Delta^{2,2}_{\mathrm{M},\mathscr{P}_-} \ \rightsquigarrow\  \cdots \ \rightsquigarrow \  \Delta^{2,\text{crk}(Z)-1}_{\mathrm{M},\mathscr{P}_-} \ \rightsquigarrow \ \hspace{35mm}\\
\Delta^{2,\text{crk}(Z)}_{\mathrm{M},\mathscr{P}_-} =\Delta^{3,1}_{\mathrm{M},\mathscr{P}_-} \ \rightsquigarrow\ \Delta^{3,2}_{\mathrm{M},\mathscr{P}_-} \ \rightsquigarrow \
 \cdots \ \rightsquigarrow\ \Delta^{3,\text{crk}(Z)-1}_{\mathrm{M},\mathscr{P}_-}\  \rightsquigarrow\ \cdots \ \rightsquigarrow \ \Delta_{\mathrm{M},\mathscr{P}_+}, \hspace{10mm}
\end{multline*}
where  $\Delta^{m,k+1}_{\mathrm{M},\mathscr{P}_-}$ is the subcomplex of $\Delta^{m,k}_{\mathrm{M},\mathscr{P}_-}$ obtained by collapsing all the faces $\vartriangle_{I<\mathscr{F}}$ with
\[
 \text{cl}_\mathrm{M}(I) = Z, \quad  \text{min}\ \mathscr{F}\neq Z, \quad  |Z \setminus I|=m,\quad |\mathscr{F}|=\text{crk}_\mathrm{M}(Z)-k.
\]
The faces $\vartriangle_{I<\mathscr{F}}$ satisfying the above conditions can be collapsed in $\Delta^{m,k}_{\mathrm{M},\mathscr{P}_-}$ because
\[
\text{link}(\vartriangle_{I<\mathscr{F}},\Delta^{m,k}_{\mathrm{M},\mathscr{P}_-})=\{\mathbf{e}_Z\}.
\]
It follows that the homotopy type of the Bergman complex $\Delta_{\mathrm{M},\mathscr{P}}$ is independent of $\mathscr{P}$.
For basics of elementary collapses of simplicial complexes, see \cite[Chapter 6]{Kozlov}.
The special case that $\Delta_{\mathrm{M},\varnothing}$ is homotopic  to $\Delta_\mathrm{M}$ is 
an elementary consequence of the nerve theorem and gives a homotopy version of the usual crosscut theorem \cite[Chapter 13]{Kozlov}.
\end{remark}

We construct  homomorphisms associated to the matroidal flip, the \emph{pullback homomorphism} and the \emph{Gysin homomorphism}.

\begin{proposition}
There is a graded ring homomorphism between the Chow rings
\[
\Phi_Z: A^*(\mathrm{M},\mathscr{P}_-) \longrightarrow A^*(\mathrm{M},\mathscr{P}_+)
\]
uniquely determined by the property
\[
x_F \longmapsto x_F \quad \text{and} \quad x_i \longmapsto \begin{cases} x_i+x_Z &  \text{if $i \in Z$,}\\  x_i & \text{if $i \notin Z$.} \end{cases}
\] 
\end{proposition}

The map $\Phi_Z$ will be called the \emph{pullback homomorphism} associated to the matroidal flip from $\mathscr{P}_-$ to $\mathscr{P}_+$.
We will show that the pullback homomorphism is injective in Theorem~\ref{DecompositionTheorem}.

\begin{proof}
Consider the homomorphism between the polynomial rings 
\[
\phi_Z: S_{E \cup \mathscr{P}_-} \longrightarrow S_{E \cup \mathscr{P}_+}
\]
defined by the same rule determining $\Phi_Z$.
We claim that
\[
\phi_Z(\mathscr{I}_1) \subseteq \mathscr{I}_1, \quad
\phi_Z(\mathscr{I}_2) \subseteq \mathscr{I}_1 + \mathscr{I}_2, \quad
\phi_Z(\mathscr{I}_3) \subseteq \mathscr{I}_2 + \mathscr{I}_3, \quad
\phi_Z(\mathscr{I}_4) \subseteq \mathscr{I}_4.
\]
The first and the last inclusions are straightforward to verify.

We check the second inclusion. 
For an element $i$  in $E \setminus F$, we have
\[
\phi_Z(x_ix_F) = \begin{cases}x_ix_F+x_Zx_F &\text{if $i \in Z$,}\\ x_ix_F & \text{if $i \notin Z$.} \end{cases} 
\]
If $i$ is in  $Z \setminus F$, then the monomial $x_Z x_F$ is in $\mathscr{I}_1$ because  $Z$ is minimal in $\mathscr{P}_+$.

We check the third inclusion. 
For an independent set $I$ whose closure is in $\mathscr{P}_- \cup \{E\}$, 
\[
\phi_Z\Big(\prod_{i \in I} x_i\Big) =\prod_{i \in I \setminus Z} x_i \ \prod_{i \in I \cap Z} (x_i+x_Z).
\]
The term $\prod_{i \in I} x_i$ in the expansion of the right-hand side is in $\mathscr{I}_3$. 
Since $Z$ is minimal in $\mathscr{P}_+$,
 there is an element in $I \setminus Z$,  and hence all the remaining terms in the expansion are in $\mathscr{I}_2$.
\end{proof}

\begin{proposition}
The pullback homomorphism $\Phi_Z$ is an isomorphism when $\text{rk}_\mathrm{M}(Z)=1$.
\end{proposition}

\begin{proof}
Let $j_1$ and $j_2$ be distinct elements of $Z$.
If $Z$ has rank $1$, then a flat contains  $j_1$ if and only if it contains  $j_2$.
It follows from the linear relation in $S_{E \cup \mathscr{P}_-}$ for $j_1$ and $j_2$ that
\[
x_{j_1}=x_{j_2} \in A^*(\mathrm{M},\mathscr{P}_-).
\]
We choose an element $j \in Z$, and construct the inverse
$\Phi_Z'$ of $\Phi_Z$ by setting 
\[
x_Z \longmapsto x_{j}, \quad x_F \longmapsto x_F, \quad \text{and} \quad
x_i \longmapsto \begin{cases} 0 &  \text{if $i \in Z$,}\\  x_i & \text{if $i \notin Z$.} \end{cases}
\]
It is straightforward to check that $\Phi_Z'$ is well-defined, and that $\Phi_Z'=\Phi_Z^{-1}$.
\end{proof}

As before, we identify the flats of $\mathrm{M}_Z$ with the flats of $\mathrm{M}$ containing $Z$, and identify the flats of $\mathrm{M}^Z$ with the flats of $\mathrm{M}$ contained in $Z$.

\begin{proposition}\label{PropositionGysin}
Let $p$ and $q$ be positive integers.
\begin{enumerate}[(1)]\itemsep 5pt
\item There is a group homomorphism
\[
\Psi_{Z}^{p,q}: A^{q-p}(\mathrm{M}_{Z}) \longrightarrow A^{q}(\mathrm{M},\mathscr{P}_+)
\]
uniquely determined by the property
$
x_\mathscr{F} \longmapsto x_{Z}^p \ x_{\mathscr{F}}.
$
\item There is a group homomorphism
\[
\Gamma^{p,q}_Z:A^{q-p}(\mathrm{M}^Z) \longrightarrow A^{q}(\mathrm{M})
\]
uniquely determined by the property $x_\mathscr{F} \longmapsto x_Z^p \ x_\mathscr{F}$.
\end{enumerate}
\end{proposition}

The map $\Psi^{p,q}_Z$ will be called the \emph{Gysin homomorphism} of type $p,q$ associated to the matroidal flip from $\mathscr{P}_-$ to $\mathscr{P}_+$.
We will  show that the Gysin homomorphism is injective when $p < \text{rk}_\mathrm{M}(Z)$ in Theorem~\ref{DecompositionTheorem}.

\begin{proof}
It is clear that  the Gysin homomorphism $\Psi_{Z}^{p,q}$ respects the incomparability relations.
We check that $\Psi_{Z}^{p,q}$  respects the linear relations.

Let $i_1$ and $i_2$ be elements in $E \setminus Z$,
and consider the linear relation in $S_{E \cup\mathscr{P}_+}$ for $i_1$ and $i_2$:
\[
\Big(x_{i_1}+\sum_{i_1 \in F} x_F\Big)-\Big(x_{i_1}+\sum_{i_2 \in F} x_F\Big) \in \mathscr{I}_4.
\]
Since $i_1$ and $i_2$ are not in $Z$, multiplying the linear relation with $x_Z^p$ gives
\[
x_Z^p \Big(\sum_{Z \cup\{i_1\} \subseteq F} x_F - \sum_{Z \cup \{i_2\} \subseteq F} x_F \Big) \in \mathscr{I}_1+\mathscr{I}_2+\mathscr{I}_4.
\]

The second statement on $\Gamma^{p,q}_Z$ can be proved in the same way, using $i_1$ and $i_2$ in $Z$.
\end{proof}

Let $\mathscr{P}$ be any order filter of $\mathscr{P}(\mathrm{M})$.
We choose a sequence of order filters of the form
\[
\varnothing,\mathscr{P}_1,\mathscr{P}_2,\ldots, \mathscr{P},\ldots,\mathscr{P}(\mathrm{M}),
\]
where an order filter in the sequence is obtained from the preceding one by adding a single flat.
The corresponding sequence of matroidal flips interpolates between $\Sigma_{\mathrm{M},\varnothing}$ and $\Sigma_{\mathrm{M}}$:
\[
\Sigma_{\mathrm{M},\varnothing} \rightsquigarrow\Sigma_{\mathrm{M},\mathscr{P}_1} \rightsquigarrow \ldots \rightsquigarrow    \Sigma_{\mathrm{M},\mathscr{P}} \rightsquigarrow   \ldots \rightsquigarrow  \Sigma_{\mathrm{M}}.
\]

\begin{definition}\label{DefinitionIntermediateDegree}
We write $\Phi_\mathscr{P}$ and $\Phi_{\mathscr{P}^c}$ for the compositions of pullback homomorphisms 
\[
\Phi_\mathscr{P}:A^*(\mathrm{M},\varnothing) \longrightarrow A^*(\mathrm{M},\mathscr{P}) \quad \text{and} \quad 
\Phi_{\mathscr{P}^c}:A^*(\mathrm{M},\mathscr{P}) \longrightarrow A^*(\mathrm{M}). 
\]
\end{definition}

Note that $\Phi_\mathscr{P}$ and $\Phi_{\mathscr{P}^c}$ depend only on $\mathscr{P}$ and not on the chosen sequence of matroidal flips.
The composition of all the pullback homomorphisms $\Phi_{\mathscr{P}^c} \circ \Phi_{\mathscr{P}}$
is uniquely determined by its property 
\[
\Phi_{\mathscr{P}^c} \circ \Phi_{\mathscr{P}}\ (x_i) = \alpha_\mathrm{M}.
\]

\subsection{}

Let $\mathscr{P}_-$ and $\mathscr{P}_+$ be as before, and let $Z$ be the center of the matroidal flip from $\mathscr{P}_-$ to $\mathscr{P}_+$. 
For positive integers $p$ and $q$,
we consider  the pullback homomorphism in degree $q$ 
\[
\Phi^q_Z:A^q(\mathrm{M},\mathscr{P}_-) \longrightarrow A^q(\mathrm{M},\mathscr{P}_+)
\]
and the  Gysin homomorphism of type $p,q$ 
\[
\Psi_{Z}^{p,q}: A^{q-p}(\mathrm{M}_{Z}) \longrightarrow A^{q}(\mathrm{M},\mathscr{P}_+).
\]

\begin{proposition}\label{PropositionSurjective}
For any positive integer $q$, the sum of the pullback homomorphism and Gysin homomorphisms 
\[
\Phi_{Z}^q \oplus \bigoplus_{p=1}^{\text{rk}(Z)-1} \Psi_Z^{p,q}
\]
is a surjective group homomorphism to $A^q(\mathrm{M},\mathscr{P}_+)$.
\end{proposition}

The proof is given below Lemma~\ref{LemmaBasisReduction}.
In Theorem~\ref{DecompositionTheorem}, we will show that the sum is in fact an isomorphism.

\begin{corollary}\label{CorollaryTopIsomorphism}
The pullback homomorphism $\Phi_Z$ is an isomorphism in degree $r$:
\[
\Phi_Z^r:A^r(\mathrm{M},\mathscr{P}_-) \simeq A^r(\mathrm{M},\mathscr{P}_+).
\]
\end{corollary}

Repeated application of the corollary shows that, for any order filter $\mathscr{P}$,
the homomorphisms  $\Phi_\mathscr{P}$ and $\Phi_{\mathscr{P}^c}$  are isomorphisms in degree $r$:
\[
\Phi_\mathscr{P}^r: A^r(\mathrm{M},\varnothing) \simeq A^r(\mathrm{M},\mathscr{P}) \quad \text{and} \quad \Phi_{\mathscr{P}^c}^r:  A^r(\mathrm{M},\mathscr{P}) \simeq A^r(\mathrm{M}).  
\]

\begin{proof}[Proof of Corollary~\ref{CorollaryTopIsomorphism}]
The contracted matroid $\mathrm{M}_Z$ has rank $\text{crk}_\mathrm{M}(Z)$,
and hence
\[
\Psi_Z^{p,q}=0 \ \ \text{when $p <\text{rk}_\mathrm{M}(Z)$ and $q=r$.}
\]
Therefore,  Proposition~\ref{PropositionSurjective} for $q=r$ says that the homomorphism $\Phi_Z$ is surjective in degree $r$.

Choose a sequence of matroidal flips
\[
\Sigma_{\mathrm{M},\varnothing} \rightsquigarrow \ldots \rightsquigarrow    \Sigma_{\mathrm{M},\mathscr{P}_{-}}\rightsquigarrow \Sigma_{\mathrm{M},\mathscr{P}_{+}} \rightsquigarrow   \ldots \rightsquigarrow  \Sigma_{\mathrm{M}},
\]
and consider the corresponding group homomorphisms
\[
\xymatrixcolsep{3pc}
\xymatrix{
A^r(\mathrm{M},\varnothing) \ar[r]^{\Phi_{\mathscr{P}_-}}& A^r(\mathrm{M},\mathscr{P}_-) 
 \ar[r]^{\Phi_{\mathscr{P}_Z}}& A^r(\mathrm{M},\mathscr{P}_+)  \ar[r]^{\Phi_{\mathscr{P}_+^c}}& A^r(\mathrm{M}).
}
\]
Proposition~\ref{PropositionSurjective} applied to each matroidal flips in the sequence shows that
all three homomorphisms are surjective.
The first group is clearly isomorphic to $\mathbb{Z}$, and by Proposition~\ref{PropositionDegreeMap}, the last group is also isomorphic to $\mathbb{Z}$. 
It follows that all three homomorphisms are isomorphisms.
\end{proof}

Let $\beta_{\mathrm{M}_Z}$ be the element $\beta$ in Definition~\ref{DefinitionAlphaBeta} for the contracted matroid $\mathrm{M}_Z$.
The first part of Proposition~\ref{PropositionGysin} shows that the expression $x_Z\hspace{0.5mm} \beta_{\mathrm{M}_Z}$ defines an element in $A^*(\mathrm{M},\mathscr{P}_+)$.

\begin{lemma}\label{MinimalFlatRelation}
For any element $i$ in $Z$, we have
\[
x_ix_Z+x_Z^2+x_Z\beta_{\mathrm{M}_Z} =0\in A^*(\mathrm{M},\mathscr{P}_+).
\]
\end{lemma}

\begin{proof}
We choose an element $j$ in $E \setminus Z$, and consider the linear relation in $S_{E \cup \mathscr{P}_+}$ for $i$ and $j$:
\[
\Bigg(x_i+\sum_{\substack{i \in F \\ j \notin F}} x_F\Bigg)-\Bigg(x_j+\sum_{\substack{j \in F \\ i \notin F}} x_F\Bigg) \in \mathscr{I}_4.
\]
Since $i$ is in  $Z$, and $Z$ is minimal in $\mathscr{P}_+$,  multiplying the linear relation with $x_Z$ gives
\[
x_Zx_i+x_Z^2+\Bigg(\sum_{Z \subsetneq F \subsetneq F \cup \{j\}} x_Zx_F\Bigg) \in \mathscr{I}_1+\mathscr{I}_2+\mathscr{I}_4.
\]
The sum in the parenthesis is the image of $\beta_{\mathrm{M}_Z}$ under the  homomorphism $\Psi^{1,2}_Z$.
\end{proof}

Let $\alpha_{\mathrm{M}^Z}$ be the element $\alpha$ in Definition~\ref{DefinitionAlphaBeta} for the restricted matroid $\mathrm{M}^Z$.
The second part of Proposition~\ref{PropositionGysin} shows that the expression $x_Z\hspace{0.5mm} \alpha_{\mathrm{M}^Z}$ defines an element in $A^*(\mathrm{M})$.

\begin{lemma}\label{LemmaAlphaRelation}
If $Z$ is maximal among flats strictly contained in a proper flat $\widetilde{Z}$, then
\[
x_Z x_{\widetilde{Z}}(x_Z+\alpha_{\mathrm{M}^Z})=0 \in A^*(\mathrm{M}).
\]
If $Z$ is maximal among flats strictly contained in the flat $E$, then
\[
x_Z (x_Z+\alpha_{\mathrm{M}^Z})=0 \in A^*(\mathrm{M}).
\]
\end{lemma}

\begin{proof}
We justify the first statement; the second statement can be proved in the same way.

Choose an element $i$ in $Z$ and an element $j$  in $\widetilde{Z} \setminus Z$.
The linear relation for $i$ and $j$ shows that
\[
\sum_{\substack{i \in F \\ j \notin F}} x_F=\sum_{\substack{j \in F \\ i \notin F}} x_F \in A^*(\mathrm{M}).
\]
Multiplying both sides by the monomial $x_Z \ x_{\widetilde{Z}}$, the incomparability relations give
\[
x_Z^2 \ x_{\widetilde{Z}}+\Big( \sum_{i \in F \subsetneq Z}  x_F \ x_Z\Big) \ x_{\widetilde{Z}}=0 \in A^*(\mathrm{M}).
\]
The sum in the parenthesis is the image of $\alpha_{\mathrm{M}^Z}$ under the homomorphism $\Gamma^{1,2}_Z$.
\end{proof}

\begin{lemma}\label{LemmaIdeal}
The sum of the images of Gysin homomorphisms is the ideal  generated by $x_Z$:
\[
\sum_{p>0} \sum_{q>0} \text{im}\ \Psi_Z^{p,q}=x_Z \ A^*(\mathrm{M},\mathscr{P}_+).
\]
\end{lemma}

\begin{proof}
It is enough to prove that the right-hand side is contained in the left-hand side.
Since $Z$ is minimal in $\mathscr{P}_+$, the incomparability relations in $\mathscr{I}_1$ and the complement relations in $\mathscr{I}_2$ show that any nonzero degree $q$ monomial in the ideal generated by $x_Z$ is of the form
\[
x_Z^{k} \  \prod_{F\in \mathscr{F}} x_F^{k_F} \ \prod_{i \in I} x_i^{k_i},  \qquad I \subseteq Z < \mathscr{F},
\]
where the sum of the exponents is $q$.
Since the exponent $k$ of $x_Z$ is positive, Lemma~\ref{MinimalFlatRelation} shows that this monomial is  in the sum
\[
\text{im}\ \Psi_Z^{k,q}+\text{im}\ \Psi_Z^{k+1,q}+\cdots + \text{im}\ \Psi_Z^{q,q}. \qedhere
\]
\end{proof}

\begin{lemma}\label{LemmaInclusions}
For positive integers $p$ and $q$, we have
\[
x_Z \ \text{im}\ \Phi_Z^q \subseteq \text{im}\ \Psi_Z^{1,q+1}
\quad \text{and} \quad
x_Z \ \text{im}\ \Psi_Z^{p,q} \subseteq   \text{im}\ \Psi_Z^{p+1,q+1}.
\]
If $F$ is a proper flat strictly containing $Z$, then
\[
x_F \ \text{im}\ \Phi_Z^q \subseteq \text{im}\ \Phi_Z^{q+1}
\quad \text{and} \quad
x_F \ \text{im}\ \Psi_Z^{p,q} \subseteq   \text{im}\ \Psi_Z^{p,q+1}.
\]
\end{lemma}

\begin{proof}
Only the first inclusion is nontrivial. 
Note that the left-hand side is generated by  elements  of the form
\[
\xi=x_Z \ \prod_{F\in \mathscr{F}} x_F^{k_F}  \prod_{i \in I \setminus Z} x_i^{k_i}  \prod_{i \in I \cap Z} (x_i+x_Z)^{k_i}, 
\]
where $I$ is a subset of $E$ and $\mathscr{F}$ is a flag in $\mathscr{P}_-$.
When $I$ is contained in $Z$,
 Lemma~\ref{MinimalFlatRelation} shows that
\[
\xi=x_Z \prod_{F\in \mathscr{F}} x_F^{k_F} \ \prod_{i \in I} (-\beta_{\mathrm{M}_Z})^{k_i} \in \text{im}\ \Psi_Z^{1,q+1}.
\]
When $I$ is not contained in $Z$, a complement relation in  $S_{E \cup \mathscr{P}_+}$ shows that $\xi=0$.
\end{proof}

\begin{lemma}\label{LemmaBasisReduction}
For any integers $k \ge \text{rk}_\mathrm{M}(Z)$ and $q \ge k$, we have
\[
 \text{im}\ \Psi_{Z}^{k,q} \subseteq  \text{im} \ \Phi_{Z}^q+ \sum_{p=1}^{k-1} \text{im}\ \Psi_{Z}^{p,q}.
\]
\end{lemma}

\begin{proof}
By the second statement of Lemma~\ref{LemmaInclusions}, it is enough to prove the assertion when $q=k$:
The general case can be deduced by multiplying both sides of the inclusion by $x_\mathscr{F}$ for $Z<\mathscr{F}$.

By the first statement of Lemma~\ref{LemmaInclusions},  it is enough to justify the above when $k=\text{rk}_\mathrm{M}(Z)$:
The general case can be deduced by multiplying both sides of the inclusion by powers of $x_Z$.

We prove the assertion when $k=q=\text{rk}_\mathrm{M}(Z)$. 
For this we choose a basis $I$ of $Z$, and expand the product
\[
\prod_{i\in I} (x_i+x_Z) \in \text{im}\ \Phi_Z^k.
\]
The closure relation for $I$ shows that the  term  $\prod_{i \in I} x_i$  in the expansion is zero, 
and hence, by Lemma~\ref{MinimalFlatRelation},
\[
\prod_{i \in I} (x_i+x_Z)=(-\beta_{\mathrm{M}_Z})^k-(-x_Z-\beta_{\mathrm{M}_Z})^k \in \text{im}\ \Phi_Z^k.
\]
Expanding the right-hand side, we see that
\[
x_Z^k\in \text{im}\ \Phi_Z^k+ \sum_{p=1}^{k-1} \text{im}\ \Psi_{Z}^{p,k}.
\]
Since $\text{im}\ \Psi_{Z}^{k,k} $ is generated by $x_Z^k$, this implies the asserted inclusion.
\end{proof}

\begin{proof}[Proof of Proposition~\ref{PropositionSurjective}]
By Lemma~\ref{LemmaBasisReduction}, it is enough to show that the sum  
$
\Phi_{Z}^q \oplus \bigoplus_{p=1}^{q} \Psi_Z^{p,q}
$
is surjective.
By Lemma~\ref{LemmaIdeal}, the image of the second summand is the degree $q$ part of the ideal generated by $x_Z$.

We show that any monomial is in the image of the pullback homomorphism $\Phi_Z$ modulo the ideal generated by $x_Z$.
Note that any degree $q$ monomial not in the ideal generated by $x_Z$ is of the form
\[
  \prod_{F\in \mathscr{F}} x_F^{k_F} \ \prod_{i \in I} x_i^{k_i},  \qquad  Z \notin \mathscr{F}.
\]
Modulo the ideal generated by $x_Z$, this monomial is equal to
\[
\Phi_Z\Big(  \prod_{F\in \mathscr{F}} x_F^{k_F} \ \prod_{i \in I} x_i^{k_i}\Big)=\prod_{F\in \mathscr{F}} x_F^{k_F}  \prod_{i \in I \setminus Z} x_i^{k_i}  \prod_{i \in I \cap Z} (x_i+x_Z)^{k_i}. \qedhere
\]
\end{proof}

We use Proposition~\ref{PropositionSurjective}  to show that the Gysin homomorphism between top degrees is an isomorphism.

\begin{proposition}\label{GysinTopIsomorphism}
The Gysin homomorphism $\Psi_Z^{p,q}$ is an isomorphism when $p=\text{rk}(Z)$ and $q=r$:
\[
\Psi_Z^{p,q}: A^{\text{crk}(Z)-1}(\mathrm{M}_Z)\simeq A^r(\mathrm{M},\mathscr{P}_+).
\]
\end{proposition}

\begin{proof}
We consider the composition
\[
\xymatrixcolsep{3pc}
\xymatrix{
A^{\text{crk}(Z)-1}(\mathrm{M}_Z) \ar[r]^{\Psi^{p,q}_Z} & A^r(\mathrm{M},\mathscr{P}_+) \ar[r]^{\Phi_{\mathscr{P}_+^c}}&A^r(\mathrm{M}),
}
\qquad
x_\mathscr{F} \longmapsto x_Z^{\text{rk}(Z)} \ x_\mathscr{F}.
\]
The second map is an isomorphism by Corollary~\ref{CorollaryTopIsomorphism}, and therefore it is enough to show that the composition  is an isomorphism.

For this we choose two flags of nonempty proper flats of $\mathrm{M}$:
\begin{align*}
\mathscr{Z}_1&=\text{a flag of  flats strictly contained in $Z$ with $|\mathscr{Z}_1|=\text{rk}(Z)-1$,}\\[2pt]
\mathscr{Z}_2&=\text{a flag of  flats strictly containing $Z$ with $|\mathscr{Z}_2|=\text{crk}(Z)-1$.}
\end{align*}
We claim that the composition maps a generator to a generator:
\[
(-1)^{\text{rk}(Z)-1} \ x_Z^{\text{rk}(Z)} \ x_{\mathscr{Z}_2}= x_{\mathscr{Z}_1} \hspace{0.5mm} x_Z \hspace{0.5mm} x_{\mathscr{Z}_2} \in A^*(\mathrm{M}).
\]
Indeed, the map $\Gamma_Z^{1,\text{rk}(Z)}$ applied to the second formula of Proposition~\ref{PropositionFundamentalClass} for  $\mathrm{M}^Z$
gives
\[
 x_{\mathscr{Z}_1}\hspace{0.5mm}x_Z  \hspace{0.5mm}x_{\mathscr{Z}_2}=(\alpha_{\mathrm{M}^Z})^{\text{rk}(Z)-1} \ x_Z \hspace{0.5mm}x_{\mathscr{Z}_2}\in A^*(\mathrm{M}),
\]
and, by Lemma~\ref{LemmaAlphaRelation},  the right-hand side of the above is equal to
\[
 (-1)^{\text{rk}(Z)-1}\ x_Z^{\text{rk}(Z)} \ x_{\mathscr{Z}_2} \in A^*(\mathrm{M}).   \qedhere
\]
\end{proof}

\subsection{}

Let $\mathscr{P}_-$, $\mathscr{P}_+$, and $Z$ be as before, and let $\mathscr{P}$ be any order filter of $\mathscr{P}(\mathrm{M})$.

\begin{theorem}[Decomposition]\label{DecompositionTheorem}
For any positive integer $q$, the sum of the pullback homomorphism and the Gysin homomorphisms
\[
\Phi_{Z}^q \oplus \bigoplus_{p=1}^{\text{rk}(Z)-1} \Psi_Z^{p,q}
\]
is an isomorphism to $A^q(\mathrm{M},\mathscr{P}_+)$.
\end{theorem}

\begin{theorem}[Poincar\'e Duality]\label{PoincareDuality}
For any nonnegative integer $q \le r$, 
 the multiplication map
\[
A^{q}(\mathrm{M},\mathscr{P}) \times A^{r-q}(\mathrm{M},\mathscr{P}) \longrightarrow A^r(\mathrm{M},\mathscr{P}) 
\]
defines an isomorphism between groups
\[
A^{r-q}(\mathrm{M},\mathscr{P}) \simeq \text{Hom}_\mathbb{Z}(A^{q}(\mathrm{M},\mathscr{P}), A^r(\mathrm{M},\mathscr{P})).
\]
\end{theorem}

In particular, the groups $A^q(\mathrm{M},\mathscr{P})$ are torsion free.
We  simultaneously prove Theorem~\ref{DecompositionTheorem} (Decomposition) and Theorem~\ref{PoincareDuality} (Poincar\'e Duality) by  lexicographic induction on the rank of matroids and the cardinality of the order filters.
The proof is given below in Lemma~\ref{LemmaInduction}.

\begin{lemma}\label{LemmaUpperTriangular}
Let $q_1$ and $q_2$ be positive integers.
\begin{enumerate}[(1)]\itemsep 5pt
\item For any positive integer $p$, we have
\[
\text{im} \ \Psi_Z^{p,q_1} \cdot \text{im} \ \Phi^{q_2}_Z \subseteq \text{im} \ \Psi_Z^{p,q_1+q_2}
\]
\item For any positive integers $p_1$ and $p_2$, we have
\[
 \text{im} \  \Psi_Z^{p_1,q_1} \cdot \text{im} \  \Psi_Z^{p_2,q_2} \subseteq  \text{im} \  \Psi_Z^{p_1+p_2,q_1+q_2}.
\]
\end{enumerate}
\end{lemma}

The first inclusion shows that, when $q_1+q_2=r$ and $p$ is less than $\text{rk}(Z)$,
\[
\text{im} \ \Psi_Z^{p,q_1} \cdot \text{im} \ \Phi^{q_2}_Z = 0.
\]
The second inclusion shows  that, when $q_1+q_2=r$ and $p_1+p_2$ is less than $\text{rk}(Z)$,
\[
 \text{im} \  \Psi_Z^{p_1,q_1} \cdot  \text{im} \  \Psi_Z^{p_2,q_2} = 0.
 \]

\begin{proof}
The assertions are direct consequences of Lemma~\ref{LemmaInclusions}.
\end{proof}

\begin{lemma}\label{LemmaInduction}
Let $q$ be a positive integer, and $p_1,p_2$ be distinct positive integers less than $\text{rk}(Z)$.
\begin{enumerate}[(1)]\itemsep 5pt
\item If Poincar\'e Duality holds for $A^*(\mathrm{M},\mathscr{P}_-)$, then
\[
\text{ker} \ \Phi^q_Z=0 \quad \text{and} \quad \text{im} \ \Phi^q_Z \cap \sum_{p=1}^{\text{rk}(Z)-1} \hspace{-1mm}\text{im}\ \Psi_Z^{p,q} =0.
\]
\item  If Poincar\'e Duality  holds for $A^*(\mathrm{M}_Z)$,  then
\[
\text{ker} \ \Psi_Z^{p_1,q}=\text{ker} \ \Psi_Z^{p_2,q}=0 \quad \text{and} \quad  \text{im} \  \Psi_Z^{p_1,q} \cap \text{im} \  \Psi_Z^{p_2,q}=0
\]
\end{enumerate}
\end{lemma}

\begin{proof}
Let $\xi$ be a nonzero element in the domain of $\Phi_Z^q$. 
Since $\Phi_Z$ is an isomorphism between top degrees, 
Poincar\'e Duality for $(\mathrm{M},\mathscr{P}_-)$ implies that
\[
\Phi_Z^q(\xi) \cdot  \text{im} \ \Phi^{r-q}_Z  \neq 0.
\]
This shows that $\Phi_Z^q$ is injective.
On the other hand,  Lemma~\ref{LemmaUpperTriangular} shows that
\[
\Bigg( \sum_{p=1}^{\text{rk}(Z)-1}\text{im}\ \Psi_Z^{p,q} \Bigg) \cdot \text{im} \ \Phi^{r-q}_Z=0.
\]
This shows that the image of $\Phi_Z^q$ intersects the image of $ \oplus_{p=1}^{\text{rk}(Z)-1} \Psi_Z^{p,q} $ trivially.

Let $\xi$ be a nonzero element in the domain of $\Psi_Z^{p,q}$, where $p=p_1$ or $p=p_2$.
Since $\Psi_Z$ is an isomorphism between top degrees, 
Poincar\'e Duality for $\mathrm{M}_Z$ implies that
\[
 \Psi^{p,q}_Z(\xi) \cdot \text{im}\ \Psi^{\text{rk}(Z)-p,r-q}_Z \neq 0.
\]
This shows that $\Psi_Z^{p,q}$ is injective.
For the assertion on the intersection, we assume that $p=p_1>p_2$. 
Under this assumption Lemma~\ref{LemmaUpperTriangular} shows 
\[
 \text{im} \  \Psi_Z^{p_2,q} \cdot  \text{im} \  \Psi_Z^{\text{rk}(Z)-p,r-q} = 0.
\]
This shows that the image of $\Psi_Z^{p_1,q}$ intersects the image of $\Psi_Z^{p_2,q}$ trivially.
\end{proof}

\begin{proof}[Proofs of Theorem~\ref{DecompositionTheorem} and Theorem~\ref{PoincareDuality}]
We simultaneously prove Decomposition and Poincar\'e Duality by lexicographic induction on the rank of $\mathrm{M}$ and the cardinality of  $\mathscr{P}_-$.
Note that both statements are valid when $r=1$, and Poincar\'e Duality holds when $q=0$ or $q=r$.
Assuming that Poincar\'e Duality holds for $A^*(\mathrm{M}_Z)$,
we show the implications
\begin{multline*}
\Big(\text{Poincar\'e Duality holds for $A^*(\mathrm{M},\mathscr{P}_-)$}\Big) \Longrightarrow  \\
\Big(\text{Poincar\'e Duality holds for $A^*(\mathrm{M},\mathscr{P}_-)$ and Decomposition   holds for $\mathscr{P}_- \subseteq \mathscr{P}_+$}\Big) \\ \Longrightarrow 
\Big(\text{Poincar\'e Duality holds for $A^*(\mathrm{M},\mathscr{P}_+)$}\Big).
\end{multline*}
The base case of the induction is provided by the isomorphism
\[
A^*(\mathrm{M},\varnothing) \simeq \mathbb{Z}[x]/(x^{r+1}). 
\]
The first implication follows from Proposition~\ref{PropositionSurjective} and Lemma~\ref{LemmaInduction}.

We prove the second implication.
Decomposition for $\mathscr{P}_- \subseteq \mathscr{P}_+$ shows that, for  any positive integer $q<r$, we have
\begin{align*}
A^q(\mathrm{M},\mathscr{P}_+)&= \text{im} \  \Phi_Z^{q} \oplus  \text{im} \  \Psi_Z^{1,q} \oplus  \text{im} \  \Psi_Z^{2,q} \oplus \cdots \oplus  \text{im} \  \Psi_Z^{\text{rk}(Z)-1,q}, \ \ \text{and}
\\[4pt]
A^{r-q}(\mathrm{M},\mathscr{P}_+)&= \text{im} \  \Phi_Z^{r-q}\oplus \text{im} \  \Psi_Z^{\text{rk}(Z)-1,r-q} \oplus  \text{im} \  \Psi_Z^{\text{rk}(Z)-2,r-q} \oplus \cdots  \oplus  \text{im} \  \Psi_Z^{1,r-q}.
\end{align*}
By Poincar\'e Duality for $(\mathrm{M},\mathscr{P}_-)$ and Poincar\'e Duality  for $\mathrm{M}_Z$,  all the summands  above are torsion free.
We construct bases of the sums by choosing bases of their summands.

We use Corollary~\ref{CorollaryTopIsomorphism} and Proposition~\ref{GysinTopIsomorphism} to obtain isomorphisms
\[
A^r(\mathrm{M},\mathscr{P}_-)  \simeq A^r(\mathrm{M},\mathscr{P}_+)  \simeq A^{\text{crk}(Z)-1}(\mathrm{M}_Z) \simeq \mathbb{Z}.
\]
For a positive integer $q<r$, consider the matrices of multiplications
\begin{align*}
\mathscr{M}_+&:=\Big(A^{q}(\mathrm{M},\mathscr{P}_+) \times A^{r-q}(\mathrm{M},\mathscr{P}_+) \longrightarrow \mathbb{Z}\Big),
\\[2pt]
\mathscr{M}_-&:=\Big(A^{q}(\mathrm{M},\mathscr{P}_-) \times A^{r-q}(\mathrm{M},\mathscr{P}_-) \longrightarrow \mathbb{Z}\Big),
\end{align*}
and, for positive integers $p<\text{rk}(Z)$, 
\[
\hspace{6mm} \mathscr{M}_p:=\Big(A^{q-p}(\mathrm{M}_Z) \times A^{r-q-\text{rk}(Z)+p}(\mathrm{M}_Z) \longrightarrow \mathbb{Z}\Big).
\]
By Lemma~\ref{LemmaUpperTriangular}, under the chosen bases ordered as shown above, $\mathscr{M}_+$ is a block upper triangular matrix with block diagonals 
$\mathscr{M}_-$ and $\mathscr{M}_p$, up to signs.
It follows from  Poincar\'e Duality for $(\mathrm{M},\mathscr{P}_-)$ and  Poincar\'e Duality  for $\mathrm{M}_Z$ that
\[
\text{det}\ \mathscr{M}_+=\pm \text{det}\ \mathscr{M}_- \times \prod_{p=1}^{\text{rk}(Z)-1} \text{det}\ \mathscr{M}_p=\pm 1.
\]
This proves the second implication, completing the lexicographic induction.
\end{proof}

\section{Hard Lefschetz property and Hodge-Riemann relations}\label{SectionHLHR}

\subsection{}

Let $r$ be a nonnegative integer.
We record basic algebraic facts concerning Poincar\'e duality, the hard Lefschetz property, and the Hodge-Riemann relations.

\begin{definition}
A graded Artinian ring $R^*$ satisfies  \emph{Poincar\'e duality of dimension $r$} if
\begin{enumerate}[(1)]\itemsep 5pt
\item there are isomorphisms
$R^0 \simeq \mathbb{R}$ and $R^r \simeq \mathbb{R}$, 
\item for every integer $q>r$, we have $R^q \simeq 0$, and,
\item for every integer $q \le r$,
the multiplication defines an isomorphism
\[
R^{r-q} \longrightarrow \text{Hom}_\mathbb{R}(R^{q},R^r).
\]
\end{enumerate}
In this case, we say that $R^*$ is a \emph{Poincar\'e duality algebra of dimension $r$}.
\end{definition}

In the remainder of this subsection, we suppose that $R^*$ is a Poincar\'e duality algebra of dimension $r$.
We fix an isomorphism, called the \emph{degree map} for $R^*$, 
\[
\text{deg}: R^r \longrightarrow \mathbb{R}.
\]

\begin{proposition}\label{PropositionLocalDuality}
For any nonzero element $x$ in $R^d$,  the quotient ring
\[
R^*/\text{ann}(x), \ \ \text{where}\ \ \text{ann}(x):=\{a \in R^* \mid x \cdot a =0\},
\]
is a  Poincar\'e duality algebra of dimension $r-d$.
\end{proposition}

By definition, the degree map for $R^*/\text{ann}(x)$ \emph{induced by $x$} is the homomorphism
\[
\text{deg}(x \cdot -): R^{r-d}/\text{ann}(x) \longrightarrow \mathbb{R}, \qquad a+\text{ann}(x) \longmapsto \text{deg}(x \cdot a).
\]
The Poincar\'e duality for $R^*$ shows that the degree map  for $R^*/\text{ann}(x)$  is an isomorphism.

\begin{proof}
This is straightforward to check, see for example \cite[Corollary I.2.3]{Meyer-Smith}.
\end{proof}

\begin{definition}
Let $\ell$ be an element of $R^1$, and let $q$ be a nonnegative integer $\le \frac{r}{2}$.
\begin{enumerate}[(1)]\itemsep 5pt
\item The \emph{Lefschetz operator} on $R^q$ associated to $\ell$ is the linear map
\[
L^q_{\ell}: R^q \longrightarrow R^{r-q}, \qquad a \longmapsto \ell^{r-2q} \  a.
\]
\item The \emph{Hodge-Riemann form} on $R^q$ associated to $\ell$ is  the symmetric bilinear form
\[
Q^{q}_{\ell}: R^q \times R^q \longrightarrow \mathbb{R}, \qquad (a_1,a_2) \longmapsto (-1)^q\ \text{deg}\ ( a_1 \cdot L_{\ell}^q (a_2) ).
\]
\item The \emph{primitive subspace} of $R^q$ associated to $\ell$ is the subspace
\[
P^q_{\ell}:=\{a \in R^q \mid \ell \cdot L_\ell^q(a)=0\} \subseteq R^q.
\]
\end{enumerate}
\end{definition}

\begin{definition}[Hard Lefschetz property and Hodge-Riemann relations]
We say that
\begin{enumerate}[(1)]\itemsep 5pt
\item  $R^*$  satisfies $\text{HL}(\ell)$ if 
 the Lefschetz operator $L_{\ell}^q$ is an isomorphism on $R^q$ for all $q \le \frac{r}{2}$, and
\item  $R^*$  satisfies $\text{HR}(\ell)$ if 
the Hodge-Riemann form $Q_{\ell}^q$ is positive definite on $P_\ell^q$ for all $q \le \frac{r}{2}$.
\end{enumerate}
\end{definition}

If the Lefschetz operator $L^{q}_\ell$ is an isomorphism, then there is a decomposition
\[
R^{q+1}=P_\ell^{q+1} \oplus \ell \hspace{0.5mm} R^{q}.
\]
Consequently, when $R^*$ satisfies $\text{HL}(\ell)$, we have the \emph{Lefschetz decomposition} of $R^q$ for $q \le \frac{r}{2}$:
\[
R^q=P^q_\ell \oplus \ell\hspace{0.5mm} P^{q-1}_\ell \oplus \cdots \oplus \ell^q \hspace{0.5mm}P^0_\ell.
\]
An important basic fact is that the Lefschetz decomposition of $R^q$ is orthogonal with respect to the Hodge-Riemann form $Q^q_\ell$: For nonnegative integers $q_1<q_2 \le q$, we have
\[
Q_\ell^q\Big(\ell^{q_1} a_1,\ell^{q_2} a_2\Big)=(-1)^q \text{deg}\ \Big(\ell^{q_2-q_1} \big(\ell^{r-2(q-q_1)} a_1\big) a_2\Big)=0, \quad a_1 \in P_\ell^{q-q_1}, \ \ a_2 \in P_\ell^{q-q_2}.
\]

\begin{proposition}\label{HLCharacterization}
The following conditions are equivalent for $\ell \in R^1$:
\begin{enumerate}[(1)]\itemsep 5pt
\item $R^*$ satisfies $\text{HL}(\ell)$.
\item The Hodge-Riemann form $Q_\ell^q$ on $R^q$  is nondegenerate for all $q \le \frac{r}{2}$.
\end{enumerate}
\end{proposition}

\begin{proof}
The Hodge-Riemann form $Q_\ell^q$ on $R^q$ is nondegenerate if and only if the composition
\[
\xymatrixcolsep{2.5pc}
\xymatrix{
R^q \ar[r]^{L^q_\ell\ } & R^{r-q} \ar[r] & \text{Hom}_\mathbb{R}(R^q,R^r)
}
\]
is an isomorphism, where the second map is given by the multiplication in $R^*$.
Since $R^*$ satisfies Poincar\'e duality, the composition is an isomorphism if and only if $L^q_\ell$ is an isomorphism.
\end{proof}

If $L_\ell^q(a)=0$, then $Q_\ell^q(a,a)=0$ and $a \in P_\ell^q$.
Thus the property $\text{HR}(\ell)$ implies the property $\text{HL}(\ell)$.

\begin{proposition}\label{HRCharacterization}
The following conditions are equivalent for $\ell \in R^1$:
\begin{enumerate}[(1)]\itemsep 5pt
\item $R^*$ satisfies $\text{HR}(\ell)$.
\item The Hodge-Riemann form $Q_\ell^q$ on $R^q$  is nondegenerate and has signature
\[
\sum_{p=0}^q (-1)^{q-p} \Big(\text{dim}_\mathbb{R}R^p -\text{dim}_\mathbb{R} R^{p-1}\Big) \ \ \text{for all $q \le \frac{r}{2}$.}
\]
\end{enumerate}
\end{proposition}

Here, the signature of a symmetric bilinear form is $n_+-n_-$, where $n_+$ and  $n_-$ are the number of positive and negative eigenvalues of any matrix representation the bilinear form \cite[Section 6.3]{Jacobson}.

\begin{proof}
If $R^*$ satisfies $\text{HR}(\ell)$,
then $R^*$ satisfies $\text{HL}(\ell)$,
and therefore we have the Lefschetz decomposition
\[
R^q=P^q_\ell \oplus \ell\hspace{0.5mm} P^{q-1}_\ell \oplus \cdots \oplus \ell^q \hspace{0.5mm}P^0_\ell.
\]
Recall that the Lefschetz decomposition of $R^q$ is orthogonal with respect to  $Q^q_\ell$, and note that there is an isometry
\[
\big(P_\ell^p,Q^p_\ell\big) \simeq \big(\ell^{q-p} \hspace{0.5mm}P_\ell^p,(-1)^{q-p}Q_\ell^q\big) \ \ \text{for every nonnegative integer $p \le q$.}
\]
Therefore, the condition $\text{HR}(\ell)$ implies that
\begin{align*}
\Big(\text{signature of $Q^q_\ell$ on $R^q$}\Big)&=\sum_{p=0}^q (-1)^{q-p}\Big(\text{signature of $Q^{p}_\ell$ on $P_\ell^{p}$}\Big)\\
&=\sum_{p=0}^q (-1)^{q-p} \Big(\text{dim}_\mathbb{R}R^{p} -\text{dim}_\mathbb{R} R^{p-1}\Big). 
\end{align*}

Conversely, suppose that the Hodge-Riemann forms $Q^q_\ell$ are nondegenerate and their signatures are given by the stated formula. 
Proposition~\ref{HLCharacterization} shows that $R^*$ satisfies $\text{HL}(\ell)$, and hence  
\[
R^q=P^q_\ell \oplus \ell\hspace{0.5mm} P^{q-1}_\ell \oplus \cdots \oplus \ell^q \hspace{0.5mm}P^0_\ell.
\]
The Lefschetz decomposition of $R^q$ is orthogonal with respect to  $Q^q_\ell$, and therefore
\begin{align*}
\Big(\text{signature of $Q^q_\ell$ on $P^q_\ell$}\Big)&=
\Big(\text{signature of $Q^q_\ell$ on $R^q$}\Big)
-\Big(\text{signature of $Q^{q-1}_\ell$ on $R^{q-1}$}\Big).
\end{align*}
The assumptions on  the signatures of $Q^q_\ell$ and $Q^{q-1}_\ell$ show that the right-hand side is 
\[
\text{dim}_\mathbb{R}R^{q} -\text{dim}_\mathbb{R} R^{q-1}= \text{dim}_\mathbb{R} P^{q}_\ell.
\]
Since $Q^q_\ell$ is nondegenerate on $P^q_\ell$, this means that $Q^q_\ell$ is positive definite on $P^q_\ell$.
\end{proof}

\subsection{}

In this subsection, we show that the properties $\text{HL}$ and $\text{HR}$ are preserved under the tensor product of Poincar\'e duality algebras.

Let  $R_1^*$ and $R_2^*$ be  Poincar\'e duality algebras  of dimensions $r_1$ and $r_2$ respectively. We choose degree maps for $R_1^*$ and for $R_2^*$, denoted
\[
\text{deg}_1:R_1^{r_1} \longrightarrow \mathbb{R}, \qquad
 \text{deg}_2:R^{r_2}_2 \longrightarrow \mathbb{R}.
\]
We note that $R_1 \otimes_\mathbb{R} R_2$ is a Poincar\'e duality algebra of dimension $r_1+r_2$:
For any two graded components of the tensor product with complementary degrees
\begin{align*}
\Big( R_1^p \otimes_\mathbb{R} R_2^0\Big) \oplus \Big( R_1^{p-1} \otimes_\mathbb{R} R_2^1\Big) \oplus \cdots \oplus \Big( R_1^0 \otimes_\mathbb{R} R_2^p\Big), \\
\Big( R_1^q \otimes_\mathbb{R} R_2^0\Big) \oplus \Big( R_1^{q-1} \otimes_\mathbb{R} R_2^1\Big) \oplus \cdots \oplus \Big( R_1^0 \otimes_\mathbb{R} R_2^q\Big),
\end{align*}
the multiplication of the two can be represented by a block diagonal matrix with  diagonals
\[
\Big( R^{p-k}_1 \otimes_\mathbb{R} R_2^{k}\Big) \times \Big(R^{q-r_2+k}_1 \otimes_\mathbb{R} R_2^{r_2-k} \Big)\longrightarrow R^{r_1}_1 \otimes_\mathbb{R} R_2^{r_2}.
\]
By definition, the \emph{induced degree map} for the tensor product is the isomorphism
\[
\text{deg}_1 \otimes_\mathbb{R} \text{deg}_2:R_1^{r_1} \otimes_\mathbb{R} R_2^{r_2} \longrightarrow \mathbb{R}.
\]
We use the induced degree map whenever we discuss the property $\text{HR}$ for  tensor products.

\begin{proposition}\label{TensorHR}
Let $\ell_1$ be an element of $R^1_1$, and let $\ell_2$ be an element of $R_2^1$.
\begin{enumerate}[(1)]\itemsep 5pt
\item If $R_1^*$ satisfies $\text{HL}(\ell_1)$ and $R_2^*$ satisfies $\text{HL}(\ell_2)$, then $R_1^* \otimes_\mathbb{R} R_2^*$ satisfies $\text{HL}(\ell_1\otimes 1 +1 \otimes \ell_2)$.
\item If $R_1^*$ satisfies $\text{HR}(\ell_1)$ and $R_2^*$ satisfies $\text{HR}(\ell_2)$, then $R_1^* \otimes_\mathbb{R} R_2^*$ satisfies $\text{HR}(\ell_1\otimes 1 +1 \otimes \ell_2)$.
\end{enumerate}
\end{proposition}


We begin the proof with the following special case.

\begin{lemma} \label{LemmaPP}
Let $r_1\leq r_2$ be nonnegative integers, and consider the Poincar\'e duality algebras
\[
R^*_1= \mathbb{R}[x_1]/(x_1^{r_1+1}) \ \ \text{and} \ \
R^*_2= \mathbb{R}[x_2]/(x_2^{r_2+1})
\]
equipped with the degree maps
\begin{align*}
\text{deg}_1:R^{r_1}_1 \longrightarrow \mathbb{R}, \qquad x_1^{r_1} \longmapsto 1,\\
\text{deg}_2:R^{r_2}_2 \longrightarrow \mathbb{R}, \qquad x_2^{r_2} \longmapsto 1.
\end{align*}
Then $R_1^*$ satisfies $\text{HR}(x_1)$, $R_2$ satisfies $\text{HR}(x_2)$, and $R^*_1 \otimes_\mathbb{R} R^*_2$ satisfies $\text{HR}(x_1 \otimes 1+1 \otimes x_2)$.
\end{lemma}

The first two assertions are easy to check, and the third assertion follows from the Hodge-Riemann relations for the cohomology of the compact K\"ahler manifold $\mathbb{C}\mathbb{P}^{r_1}\times \mathbb{C}\mathbb{P}^{r_2}$.   Below we sketch a combinatorial proof using the Lindstr\"{o}m-Gessel-Viennot lemma (cf. \cite[Proof of Lemma 2.2]{McDaniel}).

\begin{proof}
For the third assertion, we identify the tensor product with
\[
R^*:=\mathbb{R}[x_1,x_2]/(x_1^{r_1+1},x_2^{r_2+1}), \ \ \text{and set} \ \ \ell:= x_1 + x_2.
\]
The induced degree map for the tensor product will be written
\[
\text{deg}: R^{r_1+r_2} \longrightarrow \mathbb{R}, \qquad x_1^{r_1}x_2^{r_2} \longmapsto 1.
\]
\begin{claim}
For some (equivalently any) choice of basis of $R^q$, we have
\[
(-1)^{\frac{q(q+1)}{2}}\text{det}\ \big(Q_\ell^q\big)>0 \ \ \text{for all nonnegative integers} \ \ q \le r_1.
\]
\end{claim}

We show that it is enough to prove the claim. The inequality of the claim implies that $Q^q_\ell$ is nondegenerate for  $q \le r_1$, and hence $L^q_\ell$ is an isomorphism for $q \le r_1$.
The Hilbert function of $R^*$ forces the dimensions of the primitive subspaces to satisfy
\[
\text{dim}_\mathbb{R} P_\ell^q=\begin{cases} 1 & \text{for $q \le r_1$,} \\ 0 & \text{for $q>r_1$,}\end{cases}
\]
and that  there is a decomposition
\[
R^q=P^q_\ell \oplus  \ell P^{q-1}_\ell \oplus \dots\oplus \ell^q P^0_\ell \ \ \text{for} \ \  q \le r_1.
\]
Every summand of the above decomposition is $1$-dimensional, and hence
\[
\Big(\text{signature of $Q_\ell^q$ on $R^q$}\Big)=\pm1 -\Big(\text{signature of $Q_\ell^{q-1}$ on $R^{q-1}$}\Big).
\]
The claim on  the determinant of $Q_\ell^q$ determines the sign of $\pm 1$ in the above equality:
\[
\Big(\text{signature of $Q_\ell^q$}\Big)=1-\Big(\text{signature of $Q_\ell^{q-1}$}\Big).
\]
It follows that the signature of $Q_\ell^q$ on $P^q_\ell$ is $1$ for $q \le r_1$,
and thus $R$ satisfies $\text{HR}(\ell)$.

To prove the claim, we consider the monomial basis 
\[
\Big\{x_1^ix_2^{q-i}\mid i=0,1,\dots, q\Big\} \subseteq R^q.
\]
The matrix $[a_{ij}]$ which represents $(-1)^q\hspace{0.5mm} Q^q_\ell$  has binomial coefficients as its entries: 
\[
[a_{ij}]
:=\Bigg[\text{deg}\Big((x_1+x_2)^{r_1+r_2-2q}x_1^{i+j}x_2^{q-i+q-j}\Big)\Bigg]
=\Bigg[\binom{r_1+r_2-2q}{r_1-i-j}\Bigg].
\]
The sign of the determinant of $[a_{ij}]$ can be determined using the Lindstr\"{o}m-Gessel-Viennot lemma:
\[
(-1)^{q(q+1)/2}\text{det}\ [a_{ij}]>0. 
\]
See  \cite[Section 5.4]{Aigner} for an exposition and similar examples.  
\end{proof}

Now we reduce  Proposition~\ref{TensorHR} to the case of Lemma~\ref{LemmaPP}.
We first introduce some useful notions to be used in the remaining part of the proof.

Let $R^*$ be a Poincar\'e duality algebra of dimension $r$, and let $\ell$ be an element of $R^1$.

\begin{definition}
Let $V^*$ be a graded subspace of $R^*$.
 We say that
\begin{enumerate}[(1)]\itemsep  5pt
\item $V^*$ satisfies $\text{HL}(\ell)$ if $Q^q_\ell$ restricted to $V^q$ is nondegenerate for all nonnegative $q \le \frac{r}{2}$.
\item $V^*$ satisfies $\text{HR}(\ell)$ if $Q^q_\ell$ restricted to $V^q$ is nondegenerate and has signature
\[
\sum_{p=0}^q (-1)^{q-p} \Big(\text{dim}_\mathbb{R} V^p - \text{dim}_\mathbb{R} V^{p-1}\Big) \ \ \text{for all nonnegative $q \le \frac{r}{2}$}.
\]
\end{enumerate}
\end{definition}

Propositions~\ref{HLCharacterization} and~\ref{HRCharacterization} show that this agrees with the previous definition when $V^*=R^*$.

\begin{definition}
Let $V_1^*$ and $V_2^*$ be graded subspaces of $R^*$.
We write 
\[
V_1^* \perp_{\text{PD}} V_2^*
\]
to mean that $V_1^* \cap V_2^* =0$ and $V_1^{r-q} \ V_2^{q}=0$ for all nonnegative integers $q \le r$, and write
\[
V_1^* \perp_{Q^*_\ell} V_2^*
\]
to mean that $V_1^* \cap V_2^* =0$ and $Q^q_\ell(V_1^q,V_2^q)=0$ for all nonnegative integers $q \le \frac{r}{2}$.
\end{definition}

We record here basic properties of the two notions of orthogonality.
Let $S^*$ be another Poincar\'e duality algebra of dimension $s$.

\begin{lemma}\label{LemmaOrthogonality}
Let $V_1^*, V_2^* \subseteq R^*$ and $W_1^*, W_2^* \subseteq S^*$ be graded subspaces.
\begin{enumerate}[(1)]\itemsep 5pt
\item If $V_1^* \perp_{Q^*_\ell} V_2^*$ and if both $V_1^*$, $V_2^*$ satisfy $\text{HL}(\ell)$, then   $V_1^* \oplus V_2^*$ satisfy $\text{HL}(\ell)$.
\item If $V_1^* \perp_{Q^*_\ell} V_2^*$ and if both $V_1^*$, $V_2^*$ satisfy $\text{HR}(\ell)$, then   $V_1^* \oplus V_2^*$ satisfy $\text{HR}(\ell)$.
\item If $V_1^* \perp_{\text{PD}} V_2^*$ and if $\ell \hspace{0.5mm}V_1^* \subseteq V_1^*$, then $V_1^* \perp_{Q^*_\ell} V_2^*$.
\item If $V_1^* \perp_{\text{PD}} V_2^*$, then $(V^*_1 \otimes_\mathbb{R} W_1^*) \perp_{\text{PD}} (V^*_2 \otimes_\mathbb{R} W_2^*)$.
\end{enumerate}
\end{lemma}

\begin{proof}
The first two assertions are straightforward.
We justify the third assertion: For any nonnegative integer $q \le \frac{r}{2}$, the assumption on $V_1^*$ implies $L_\ell^q V_1^q \subseteq V_1^{r-q}$, and hence
\[
Q_\ell^q(V_1^q,V_2^q) \subseteq \text{deg}(V_1^{r-q} V_2^q)=0.
\]
For the fourth assertion, we check that, for any nonnegative integers $p_1,p_2,q_1,q_2$ whose sum is $r+s$, 
\[
V^{p_1}_1V_2^{p_2} \otimes_\mathbb{R} W_1^{q_1}W_2^{q_2} =0.
\]
The assumption on  $V_1^*$ and $V_2^*$ shows that the  first factor is trivial if $p_1+p_2 \ge r$, and the second factor is trivial if otherwise.
\end{proof}

\begin{proof}[Proof of Proposition~\ref{TensorHR}]
Suppose that $R_1^*$ satisfies $\text{HR}(\ell_1)$ and that $R_2^*$ satisfies $\text{HR}(\ell_2)$.
We set
\[
R^*:=R^*_1 \otimes_\mathbb{R} R_2^*, \qquad \ell:=\ell_1 \otimes 1 + 1 \otimes \ell_2.
\]
We show that $R^*$ satisfy $\text{HR}(\ell)$.
The assertion on $\text{HL}$ can be proved in the same way.

For every $p \le \frac{r_1}{2}$,  choose an orthogonal basis of $P_{\ell_1}^p \subseteq R_1^p$ with respect to $Q_{\ell_1}^p$:
\[
\Big\{v^p_{1},v^p_2,\ldots,v^p_{m(p)}\Big\} \subseteq P_{\ell_1}^p.
\]
Similarly, for every $q \le \frac{r_2}{2}$,  choose an orthogonal basis of $P_{\ell_2}^q \subseteq R_2^q$ with respect to $Q_{\ell_2}^q$:
\[
\Big\{w^q_{1},w^q_2,\ldots,w^q_{n(q)}\Big\} \subseteq P_{\ell_2}^q.
\]
Here we use the upper indices  to indicate the degrees of basis elements.
To each pair of  $v^p_i$ and $w^q_j$, we associate a graded subspace of $R^*$: 
{
\setlength{\jot}{7pt}
\begin{gather*}
B^*(v_i^p,w_j^q):=B^*(v_{i}^{p}) \otimes_\mathbb{R} B^*(w_{j}^{q}), \ \ \text{where}\\
B^*(v_i^p):= \langle v_i^p\rangle \hspace{0.5mm}\oplus \hspace{0.5mm}\ell_1 \langle v_i^p\rangle\hspace{0.5mm} \oplus  \cdots \oplus\hspace{0.5mm} \ell_1^{r_1-2p} \langle v_i^q\rangle  \subseteq R^*_1,\\
B^*(w^q_j):= \langle w_j^q\rangle \oplus \ell_2 \langle w_j^q\rangle \oplus  \cdots \oplus \ell^{r_2-2q} \langle w_j^q\rangle \subseteq R_2^*,
\end{gather*}
}
Let us compare the tensor product $B^*(v_i^p,w_j^q)$ with the truncated polynomial ring
\[
S^*_{p,q}:=\mathbb{R}[x_1,x_2]/(x_1^{r_1-2p+1},x_2^{r_2-2q+1}).
\]
The properties $\text{HR}(\ell_1)$ and $\text{HR}(\ell_2)$ show that, for every  nonnegative integer $k
\le \frac{r_1+r_2-2p-2q}{2}
$, there is an isometry 
\[
\Big( B^{k+p+q}(v_{i}^{p},w_{j}^{q}), \  Q_\ell^{k+p+q} \Big)
\simeq
\Big(S^k_{p,q}, \ (-1)^{p+q} \hspace{0.5mm}Q_{x_1+x_2}^k\Big).
\]
Therefore, by Lemma~\ref{LemmaPP}, the graded subspace $B^*(v_i^p,w_j^q) \subseteq R^*$ satisfies $\text{HR}(\ell)$.

The properties $\text{HL}(\ell_1)$ and $\text{HL}(\ell_2)$ imply that  there is a direct sum decomposition
\[
R^*=\bigoplus_{p,q,i,j} B^*(v_i^p,w_j^q).
\]
It is enough to prove that the above decomposition is orthogonal with respect to $Q_{\ell}^*$:
\begin{claim}
Any two distinct summands of $R^*$  satisfy
$
B^*(v,w) \perp_{Q_\ell^*} B^*(v',w'). 
$
\end{claim}
For the proof of the claim, we may suppose that $w \neq w'$. 
The orthogonality of the Lefschetz decomposition for $R_2^*$ with respect to $Q_{\ell_2}^*$ shows that
\[
B(w) \perp_{\text{PD}} B(w').
\]
By the fourth assertion  of Lemma~\ref{LemmaOrthogonality}, the above implies
\[
B^*(v,w) \perp_{\text{PD}} B^*(v',w'). 
\]
By the third assertion of Lemma~\ref{LemmaOrthogonality}, this gives the claimed statement.
\end{proof}

\subsection{}

Let $\Sigma$ be a unimodular fan, or more generally a simplicial fan in $\mathbf{N}_\mathbb{R}$. 
The purpose of this subsection is to state and prove Propositions \ref{lHRtoHL} and \ref{ForSomeForAll}, which together support the inductive structure of the proof of Main Theorem \ref{MainTheoremBody}.

\begin{definition}
We say that $\Sigma$ satisfies  \emph{Poincar\'e duality of dimension $r$} if $A^*(\Sigma)_\mathbb{R}$ is a Poincar\'e duality algebra of dimension $r$.
\end{definition}

In the remainder of this subsection, we suppose that $\Sigma$ satisfies  Poincar\'e duality of dimension $r$.
We fix an isomorphism, called the \emph{degree map} for $\Sigma$,
\[
\text{deg}: A^r(\Sigma)_\mathbb{R} \longrightarrow \mathbb{R}.
\]
As before, we write $V_\Sigma$ for the set of primitive ray generators of $\Sigma$.

Note that for any nonnegative integer $q$ and $\mathbf{e}\in V_\Sigma$ there is a commutative diagram 
\[
\xymatrixcolsep{5pc}
\xymatrixrowsep{3pc}
\xymatrix{
A^{q}(\Sigma) \ar[r]^{\textrm{p}_{\mathbf{e}}} \ar[dr]_{x_\mathbf{e} \cdot -}& A^q(\text{star}(\mathbf{e},\Sigma)) \ar[d]^{ x_\mathbf{e} \cdot -}\\
&A^{q+1}(\Sigma),
}
\]
where $\mathrm{p}_\mathbf{e}$ is the pullback homomorphism $\mathrm{p}_{\mathbf{e}\in \Sigma}$
and $x_\mathbf{e}\cdot -$ are  the multiplications by $x_\mathbf{e}$.
It follows that  there is a surjective graded ring homomorphism
\[
\pi_{\mathbf{e}}: A^*(\text{star}(\mathbf{e},\Sigma)) \longrightarrow A^*(\Sigma)/\text{ann}(x_\mathbf{e}).
\]

\begin{proposition}
The star of $\mathbf{e}$ in $\Sigma$ satisfies  Poincar\'e duality of dimension $r-1$ if and only if $\pi_{\mathbf{e}}$ is an isomorphism:
\[
A^*(\text{star}(\mathbf{e},\Sigma)) \simeq A^*(\Sigma)/\text{ann}(x_\mathbf{e}).
\]
\end{proposition}

\begin{proof}
The ''if'' direction follows from Proposition~\ref{PropositionLocalDuality}: The quotient $A^*(\Sigma)/\text{ann}(x_\mathbf{e})$ is a Poincar\'e duality algebra of dimension $r-1$.

The ''only if'' direction follows from the observation that any surjective graded ring homomorphism between Poincar\'e duality algebras of the same dimension is an isomorphism.
\end{proof}

\begin{definition}
Let $\Sigma$ be a fan that satisfies Poincar\'e duality of dimension $r$.
We say that
\begin{enumerate}[(1)]\itemsep 5pt
\item $\Sigma$ satisfies the \emph{hard Lefschetz property} if $A^*(\Sigma)_\mathbb{R}$ satisfies $\text{HL}(\ell)$ for all $\ell \in \mathscr{K}_\Sigma$,
\item $\Sigma$  satisfies the \emph{Hodge-Riemann relations} if $A^*(\Sigma)_\mathbb{R}$ satisfies $\text{HR}(\ell)$ for all $\ell \in \mathscr{K}_\Sigma$, and
\item $\Sigma$  satisfies the \emph{local Hodge-Riemann relations}  if the Poincar\'e duality algebra 
\[
A^*(\Sigma)_\mathbb{R}/\text{ann}(x_\mathbf{e})
\]
satisfies $\text{HR}(\ell_\mathbf{e})$ with respect to the degree map induced by $x_\mathbf{e}$ for all $\ell \in \mathscr{K}_\Sigma$  and  $\mathbf{e} \in V_\Sigma$.
\end{enumerate}
Hereafter we write $\ell_\mathbf{e}$  for the image of $\ell$ in the quotient $A^*(\Sigma)_\mathbb{R}/\text{ann}(x_\mathbf{e})$.
\end{definition}

\begin{proposition}\label{lHRtoHL}
If $\Sigma$ satisfies the local Hodge-Riemann relations, then $\Sigma$ satisfies the hard Lefschetz property.
\end{proposition}

\begin{proof}
By definition, for $\ell \in \mathscr{K}_\Sigma$ there are positive real numbers $c_\mathbf{e}$ such that
\[
\ell= \sum_{\mathbf{e} \in V_\Sigma} c_\mathbf{e} \hspace{0.5mm} x_\mathbf{e} \in A^1(\Sigma)_\mathbb{R}.
\]
We need to show that the Lefschetz operator $L^q_\ell$ on $A^q(\Sigma)_\mathbb{R}$ is injective for all $q \le \frac{r}{2}$. 
Nothing is claimed when $r=2q$, so we may assume that $r-2q$ is positive.

Let $f$ be an element in the kernel of $L_\ell^q$, and write
 $f_\mathbf{e}$ for the image of $f$ in the quotient $A^q(\Sigma)_\mathbb{R}/\text{ann}(x_\mathbf{e})$.
Note that the element $f$ has the following properties:
\begin{enumerate}[(1)]\itemsep 2pt
\item For all $\mathbf{e} \in V_\Sigma$, the image $f_\mathbf{e}$ belongs to the primitive subspace $P_{\ell_\mathbf{e}}^q$, and
\item for the positive real numbers $c_\mathbf{e}$ as above, we have
\[
\sum_{\mathbf{e} \in V_\Sigma} c_\mathbf{e} \hspace{0.5mm} Q_{\ell_\mathbf{e}}^q(f_\mathbf{e},f_\mathbf{e})=Q_\ell^q(f,f)=0.
\]
\end{enumerate}
By the local Hodge-Riemann relations, the two properties above show
that all the
$f_\mathbf{e}$ are zero:
\[
x_\mathbf{e} \cdot f =0 \in A^*(\Sigma)_\mathbb{R} \ \ \text{for all $\mathbf{e} \in V_\Sigma$.}
\]
Since the elements $x_\mathbf{e}$ generate the Poincar\'e duality algebra $A^*(\Sigma)_\mathbb{R}$, this implies that $f=0$.
\end{proof}

\begin{proposition}\label{ForSomeForAll}
If $\Sigma$ satisfies the hard Lefschetz property, then the following are equivalent:
\begin{enumerate}[(1)] \itemsep 5pt
\item $A^*(\Sigma)_\mathbb{R}$ satisfies $\text{HR}(\ell)$ for some $\ell \in \mathscr{K}_\Sigma$. 
\item $A^*(\Sigma)_\mathbb{R}$ satisfies $\text{HR}(\ell)$ for all $\ell \in \mathscr{K}_\Sigma$.
\end{enumerate}
\end{proposition}

\begin{proof}
Let $\ell_0$ and $\ell_1$ be elements of $\mathscr{K}_\Sigma$, and suppose that
$A^*(\Sigma)_\mathbb{R}$ satisfies $\text{HR}(\ell_0)$.
Consider the parametrized family 
\[
\ell_t:= (1-t)\hspace{0.5mm} \ell_0+t\hspace{0.5mm} \ell_1, \qquad 0 \le t \le 1.
\]
Since $\mathscr{K}_\Sigma$ is convex, the elements $\ell_t$ are ample for all $t$.

Note that  $Q_{\ell_t}^q$ are nondegenerate on $A^q(\Sigma)_\mathbb{R}$ for all $t$ and $q \le \frac{r}{2}$ because $\Sigma$ satisfies the hard Lefschetz property.
It follows that the signatures of $Q_{\ell_t}^q$ should be independent of $t$ for all $q \le \frac{r}{2}$.
Since $A^*(\Sigma)_\mathbb{R}$ satisfies $\text{HR}(\ell_0)$, the common signature should be
\[
\sum_{p=0}^q (-1)^{q-p} \Big(\text{dim}_\mathbb{R} \ A^p(\Sigma)_\mathbb{R} -\text{dim}_\mathbb{R} \ A^{p-1}(\Sigma)_\mathbb{R} \Big).
\]
We conclude by  Proposition~\ref{HRCharacterization} that
$A^*(\Sigma)_\mathbb{R}$ satisfies $\text{HR}(\ell_1)$.
\end{proof}

\section{Proof of the main theorem}

\subsection{}

As a final preparation for the proof of the main theorem, we show that the property $\text{HR}$ is preserved by a matroidal flip for particular choices of ample classes.

Let $\mathrm{M}$ be as before, and consider the matroidal flip from $\mathscr{P}_-$ to $\mathscr{P}_+$ with center $Z$.
We will use the following homomorphisms:
\begin{enumerate}[(1)]\itemsep 5pt
\item The pullback homomorphism $\Phi_Z:A^*(\mathrm{M},\mathscr{P}_-) \longrightarrow A^*(\mathrm{M},\mathscr{P}_+)$.
\item The Gysin homomorphisms $\Psi_{Z}^{p,q}: A^{q-p}(\mathrm{M}_{Z}) \longrightarrow A^{q}(\mathrm{M},\mathscr{P}_+)$.
\item The pullback homomorphism $\mathrm{p}_{{Z}}: A^*(\mathrm{M},\mathscr{P}_-) \longrightarrow A^*(\mathrm{M}_Z)$.
\end{enumerate}
The homomorphism $\mathrm{p}_{Z}$ is obtained from the  graded ring homomorphism $\mathrm{p}_{\sigma \in \Sigma}$, where $\sigma=\sigma_{Z<\varnothing}$ and $\Sigma=\Sigma_{\mathrm{M},\mathscr{P}_-}$, making use of the  identification 
\[
 \text{star}(\sigma,\Sigma) \simeq \Sigma_{\mathrm{M}_Z}.
\]
In the remainder of this section, we fix a strictly convex piecewise linear function  $\ell_-$ on $\Sigma_{\mathrm{M},\mathscr{P}_-}$.
For nonnegative real numbers $t$, we set
\[
 \ell_+(t):=\Phi_Z(\ell_-)-tx_Z \in A^1(\mathrm{M},\mathscr{P}_+) \otimes_\mathbb{Z} \mathbb{R}.
\]
We write $\ell_Z$ for the pullback of  $\ell_-$ to the star of the cone $\sigma_{Z<\varnothing}$ in the Bergman fan $\Sigma_{\mathrm{M},\mathscr{P}_-}$:
\[
 \ell_Z:=\mathrm{p}_Z(\ell_-) \in A^1(\mathrm{M}_Z) \otimes_\mathbb{Z} \mathbb{R}.
\]
Proposition ~\ref{PropositionAmplePullback} shows that $\ell_Z$ is the class of a strictly convex piecewise linear function on $\Sigma_{\mathrm{M}_Z}$.

\begin{lemma}
$\ell_+(t)$ is  strictly convex for all sufficiently small positive  $t$.
\end{lemma}

\begin{proof}
It is enough to  show that $\ell_+(t)$ is strictly convex around a given cone $\sigma_{I<\mathscr{F}}$ in $\Sigma_{\mathrm{M},\mathscr{P}_+}$.

When $Z \notin \mathscr{F}$, the cone  $\sigma_{I<\mathscr{F}}$ is in the fan  $\Sigma_{\mathrm{M},\mathscr{P}_-}$, and hence we may suppose  that
\[
\text{$\ell_-$ is zero  on $\sigma_{I<\mathscr{F}}$ and positive on the link of $\sigma_{I<\mathscr{F}}$ in $\Sigma_{\mathrm{M},\mathscr{P}_-}$.}
\]
It is straightforward to deduce from the above that if
\[
0<t<\sum_{i \in Z \setminus I} \ell_-(\mathbf{e}_i),
\]
then 
$\ell_+(t)$ is zero on $\sigma_{I<\mathscr{F}}$ and positive on the link of $\sigma_{I<\mathscr{F}}$ in $\Sigma_{\mathrm{M},\mathscr{P}_+}$.
Note that $Z \setminus I$ is nonempty and each of the summands in the right-hand side of the above inequality is positive.

When $Z \in \mathscr{F}$,  the cone $\sigma_{Z<\mathscr{F} \setminus \{Z\}}$ is in  the fan $\Sigma_{\mathrm{M},\mathscr{P}_-}$, and hence we may suppose  that
\[
\text{$\ell_-$ is zero on $\sigma_{Z<\mathscr{F} \setminus \{Z\}}$ and positive on the link of $\sigma_{Z<\mathscr{F} \setminus \{Z\}}$ in $\Sigma_{\mathrm{M},\mathscr{P}_-}$.}
\]
Let $J$ be the flat $\text{min} \hspace{0.5mm} \mathscr{F} \setminus \{Z\}$, and let $m(t)$ be the linear function on $\mathbf{N}_E$ defined by setting
\[
\mathbf{e}_i \longmapsto 
\left\{\begin{array}{cl}
\frac{t}{|Z \setminus I|} & \text{if $i \in Z \setminus I$}, \\
\frac{-t}{|J \setminus Z|} & \text{if $i \in J \setminus Z$},\\
0  & \text{if otherwise.}
\end{array}\right.
\]
It is straightforward to deduce from the above that,  for all sufficiently small positive $t$,
\[
\text{$\ell_+(t)+m(t)$ is zero on $\sigma_{I<\mathscr{F}}$ and positive on the link of $\sigma_{I<\mathscr{F}}$ in $\Sigma_{\mathrm{M},\mathscr{P}_+}$.}
\]
More precisely, the latter statement is valid for all $t$ that satisfies the inequalities
\[
0<t< \text{min}\Big\{\ell_-(\mathbf{e}_F), \  \text{$\mathbf{e}_F$ is in the link of $\sigma_{Z<\mathscr{F} \setminus \{Z\}}$ in $\Sigma_{\mathrm{M},\mathscr{P}_-}$} \Big\}.
\]
Here  the minimum of the empty set is defined to be $\infty$.
\end{proof}

We write ``$\text{deg}$'' for the degree map of $\mathrm{M}$ and of $\mathrm{M}_Z$, and fix the degree maps
\begin{align*}
\text{deg}_+: A^r(\mathrm{M},\mathscr{P}_+) \longrightarrow \mathbb{Z}, \qquad a \longmapsto \text{deg}\big(\Phi_{\mathscr{P}^c_+}(a)\big), \\
\text{deg}_-: A^r(\mathrm{M},\mathscr{P}_-) \longrightarrow \mathbb{Z}, \qquad a \longmapsto \text{deg}\big(\Phi_{\mathscr{P}^c_-}(a)\big), 
\end{align*}
see Definition~\ref{DefinitionIntermediateDegree}.
We omit the subscripts $+$ and $-$ from the notation when there is no danger of confusion.
The goal of this subsection is to prove the following.

\begin{proposition}\label{BlowupHR}
Let $\ell_-$, $\ell_Z$, and $\ell_+(t)$ be as above, and suppose that
\begin{enumerate}[(1)]\itemsep 5pt
\item the Chow ring of $\Sigma_{\mathrm{M},\mathscr{P}_-}$ satisfies $\text{HR}(\ell_-)$, and
\item the Chow ring of $\Sigma_{\mathrm{M}_Z}$ satisfies $\text{HR}(\ell_Z)$.
\end{enumerate}
Then the Chow ring of $\Sigma_{\mathrm{M},\mathscr{P}_+}$ satisfies $\text{HR}(\ell_+(t))$ for all sufficiently small positive $t$.
\end{proposition}

Hereafter we suppose   $\text{HR}(\ell_-)$ and  $\text{HR}(\ell_Z)$.
We introduce the main characters appearing in the proof of Proposition~\ref{BlowupHR}:
\begin{enumerate}[(1)]\itemsep 5pt
\item A Poincar\'e duality algebra  of dimension $r$:
\[
\hspace{-30mm}
A^*_+:=\bigoplus_{q=0}^r\ A^q_+, \qquad A_+^q:=A^q(\mathrm{M},\mathscr{P}_+) \otimes_\mathbb{Z} \mathbb{R}.
\]
\item A Poincar\'e duality algebra of dimension $r$:
\[
\hspace{-35mm} 
A^*_- := \bigoplus_{q =0}^r\ A^q_-, \qquad A^q_-:= \Big(\text{im}\ \Phi_Z^q\Big) \otimes_\mathbb{Z} \mathbb{R}.
\]
\item A Poincar\'e duality algebra of dimension $r-2$:
\[
T^*_Z:=\bigoplus_{q=0}^{r-2}  T^q_Z, \ \qquad T^q_Z:=\Big(\mathbb{Z}[x_Z]/(x_Z^{\text{rk}(Z)-1}) \otimes_\mathbb{Z} A^*(\mathrm{M}_Z)\Big)^q \otimes_\mathbb{Z} \mathbb{R}.
\]
\item A graded subspace of $A^*_+$, the sum of the images of the Gysin homomorphisms:
\[
\hspace{-25mm}G^*_Z:=\bigoplus_{q = 1}^{r-1} G^q_Z, \qquad G^q_Z:=\bigoplus_{p=1}^{\text{rk}(Z)-1} \Big(\text{im}\ \Psi_Z^{p,q}\Big) \otimes_\mathbb{Z} \mathbb{R}.
\]
\end{enumerate}
The truncated polynomial ring in the definition of $T^*_Z$ is given the degree map
\[
(-x_Z)^{\text{rk}(Z)-2} \longmapsto 1,
\]
so that the truncated polynomial ring satisfies $\text{HR}(-x_Z)$.
The tensor product $T^*_Z$ is given the induced degree map
\[
(-x_Z)^{\text{rk}(Z)-2} x_\mathscr{Z}\longmapsto 1,
\]
where $\mathscr{Z}$ is any maximal flag of nonempty proper flats of $\mathrm{M}_Z$.
It follows from Proposition~\ref{TensorHR} that the tensor product satisfies $\text{HR}( 1 \otimes \ell_Z -x_Z \otimes 1)$.

\begin{definition}
For nonnegative $q \le \frac{r}{2}$, we write the Poincar\'e duality pairings for $A^*_-$ and $T^*_Z$ by
\begin{align*}
&\big\langle -,- \big\rangle^q_{A^*_-}: A_-^q \times A^{r-q}_- \longrightarrow \mathbb{R}, \\
&\big\langle -,- \big\rangle^{q-1}_{T^*_Z}: T^{q-1}_Z \times T^{r-q-1}_Z \longrightarrow \mathbb{R}.
\end{align*}
We omit the superscripts $q$ and $q-1$ from the notation when there is no danger of confusion.
\end{definition}

Theorem~\ref{DecompositionTheorem} shows that $\Phi_Z$ defines an isomorphism between the graded rings
\[
A^*(\mathrm{M},\mathscr{P}_-) \otimes_\mathbb{Z} \mathbb{R} \simeq A_-^*,
\]
and that there is a decomposition into a direct sum 
\[
A_+^* = A_-^* \oplus G_Z^*.
\]
In addition, it shows that $x_Z \cdot -$ is an isomorphism between the graded vector spaces
\[
T_Z^{*} \simeq G_Z^{*+1}.
\]
The inverse of the isomorphism $x_Z \cdot -$ will be denoted $x_Z^{-1} \cdot -$.

We equip the above graded vector spaces with the following symmetric bilinear forms.

\begin{definition}
Let $q$ be a nonnegative integer $\le \frac{r}{2}$.
\begin{enumerate}[(1)]\itemsep 5pt
\item $\big(A^q_+,Q^q_- \oplus Q^q_Z\big)$: $Q^q_- $ and $Q^q_Z$ are the  bilinear forms on $A^q_-$ and $G^q_Z$ defined below.
\item $\big(A_-^q,Q^q_- \big)$: $Q^q_-$ is the restriction of the Hodge-Riemann form $Q^q_{\ell_+(0)}$ to $A_-^q$.
\item $\big(T_Z^q,Q_\mathscr{T}^q\big)$: $Q_\mathscr{T}^q$ is the Hodge-Riemann form associated to $\mathscr{T}:=\Big(1 \otimes \ell_Z - x_Z \otimes 1\Big) \in T^1_Z$.
\item $\big(G_Z^q,Q_Z^q\big)$: $Q_Z^q$ is the bilinear form defined by saying that  $x_Z \cdot -$ gives  an isometry
\[
\Big(T_Z^{q-1}, Q^{q-1}_{\mathscr{T}}\Big) \simeq \Big(G^{q}_Z\hspace{0.5mm},\hspace{0.5mm}  Q_Z^{q}\Big).
\]
\end{enumerate}
\end{definition}

We observe that $Q^q_- \oplus Q_Z^q$ satisfies the following version of Hodge-Riemann relations:

\begin{proposition}\label{SumOfSignatures}
The bilinear form $Q^q_- \oplus Q_Z^q$ is nondegenerate on $A^q_+$ and has signature
\[
\sum_{p=0}^q (-1)^{q-p}\Big(\text{dim}_\mathbb{R} A^p_+-\text{dim}_\mathbb{R} A^{p-1}_+ \Big) \ \ \text{for all nonnegative $q \le \frac{r}{2}$.}
\]
\end{proposition}

\begin{proof}
Theorem~\ref{DecompositionTheorem} shows that
 $\Phi_Z \otimes_\mathbb{Z} \mathbb{R}$ defines an isometry
 \[
 \Big( A^q(\mathrm{M},\mathscr{P}_-)_\mathbb{R} \hspace{0.5mm},\hspace{0.5mm}  Q^q_{\ell_-}\Big) \simeq \Big(A^q_- \hspace{0.5mm},\hspace{0.5mm} Q^q_{-}\Big).
\]
It follows from the assumption on $\Sigma_{\mathrm{M},\mathscr{P}_-}$ that $Q^q_-$ is nondegenerate on $A^q_-$ and has signature
\[
\sum_{p=0}^q (-1)^{q-p}\Big(\text{dim}_\mathbb{R} A^p_- - \text{dim}_\mathbb{R} A^{p-1}_- \Big).
\]
It follows from the assumption on $\Sigma_{\mathrm{M}_Z}$ that $Q^q_Z$ is nondegenerate on $G^q_Z$ and has signature
\begin{align*}
\sum_{p=0}^{q-1} (-1)^{q-p-1}\Big(\text{dim}_\mathbb{R} T^p_Z-  \text{dim}_\mathbb{R} T^{p-1}_Z \Big)&=
\sum_{p=0}^{q-1} (-1)^{q-p-1}\Big(\text{dim}_\mathbb{R} G^{p+1}_Z- \text{dim}_\mathbb{R} G^{p}_Z \Big)\\
&=\sum_{p=0}^{q} (-1)^{q-p}\Big(\text{dim}_\mathbb{R} G^p_Z- \text{dim}_\mathbb{R} G^{p-1}_Z \Big).
\end{align*}
The assertion is deduced from the fact that the signature of the sum is the sum of the signatures.
\end{proof}

We now construct a continuous family of symmetric bilinear forms $Q^q_t$ on $A^q_+$ parametrized by positive real numbers $t$.
This family $Q^q_t$ will shown to have the following properties:
\begin{enumerate}[(1)]\itemsep 5pt
\item For every positive real number $t$, there is  an isometry
\[
\Big(A_+^q\hspace{0.5mm} ,\hspace{0.5mm} Q^q_t\Big) \simeq \Big(A_+^q\hspace{0.5mm} ,\hspace{0.5mm} Q^q_{\ell_+(t)}\Big).
\]
\item The sequence $Q^q_t$ as $t$ goes to zero converges to the sum of $Q^q_{-}$ and $Q^q_{Z}$:
\[
\lim_{t \to 0} Q^q_t=Q^q_{-} \oplus Q^q_{Z}.
\]
\end{enumerate}
For positive real numbers $t$, we define a graded linear transformation
\[
S_t: A_+^* \longrightarrow A_+^*
\]
to be the sum of the identity on $A_-^*$ and the linear transformations
\[
\Big(\text{im}\ \Psi_Z^{p,q}\Big) \otimes_\mathbb{Z} \mathbb{R} \longrightarrow \Big(\text{im}\ \Psi_Z^{p,q} \Big) \otimes_\mathbb{Z} \mathbb{R}, \qquad a \longmapsto t^{-\frac{\text{rk}(Z)}{2}+p}  \ a.
\]
The inverse transformation $S_t^{-1}$ is the sum of the identity on $A^*_-$ and the linear transformations
\[
\Big(\text{im}\ \Psi_Z^{p,q}\Big) \otimes_\mathbb{Z} \mathbb{R} \longrightarrow \Big(\text{im}\ \Psi_Z^{p,q} \Big) \otimes_\mathbb{Z} \mathbb{R}, \qquad a \longmapsto t^{\frac{\text{rk}(Z)}{2}-p}  \ a.
\]

\begin{definition}
The symmetric bilinear form $Q^q_t$  is defined  so that $S_t$ defines an isometry
\[
\Big(A_+^q\hspace{0.5mm} ,\hspace{0.5mm} Q^q_t \Big)\simeq \Big(A_+^q\hspace{0.5mm} ,\hspace{0.5mm}  Q^q_{\ell_+(t)}\Big) \ \ \text{for all nonnegative integers $q \le \frac{r}{2}$.}
\]
In other words, for any elements $a_1,a_2 \in A_+^q$, we set
\[
Q^q_t(a_1,a_2):=(-1)^q \hspace{0.5mm} \text{deg} \big( S_t(a_1) \cdot \ell_+(t)^{r-2q} \cdot S_t(a_2)\big).
\]
\end{definition}

The first property of $Q^q_t$ mentioned above is built into the definition.
We verify the assertion on the limit of $Q^q_t$ as $t$ goes to zero.

\begin{proposition}\label{PropositionLimit}
For all nonnegative integers $q \le \frac{r}{2}$, we have
\[
\lim_{t \to 0} Q^q_t=Q^q_{-} \oplus Q^q_{Z}.
\]
\end{proposition}

\begin{proof}
We first construct a deformation of the Poincar\'e duality pairing $A_+^q \times A^{r-q}_+ \longrightarrow \mathbb{R}$:
\[
\big\langle a_1,a_2 \big\rangle^q_t:= \text{deg}\big(S_t(a_1),S_t(a_2)\big), \qquad t>0.
\]
We omit the upper index $q$ when there is no danger of confusion.

\begin{claim}[1]
For any $ b_1,b_2 \in A^*_-$ and  $c_1,c_2 \in G^*_Z$ and $a_1=b_1+c_1,a_2=b_2+c_2 \in A^*_+$, 
\[
\big\langle a_1,a_2\big\rangle_0:=\lim_{t \to 0}\big\langle a_1,a_2\big\rangle_t= \big\langle b_1,b_2\big\rangle_{A^*_-} - \big\langle x_Z^{-1} c_1, x_Z^{-1} c_2\big\rangle_{T^*_Z}.
\]
\end{claim}

We write $z:=\text{rk}(Z)$ and choose bases of $A_+^q$ and $A^{r-q}_+$ that respect the decompositions
\begin{align*}
A^q_+&=A^q_- \oplus\Big( \text{im} \  \Psi_Z^{1,q} \oplus  \text{im} \  \Psi_Z^{2,q} \oplus \cdots \oplus  \text{im} \  \Psi_Z^{z-1,q}\Big) \otimes_\mathbb{Z} \mathbb{R}, \ \ \text{and}
\\[4pt]
A^{r-q}_+&= A_-^{r-q}\oplus\Big( \text{im} \  \Psi_Z^{z-1,r-q} \oplus  \text{im} \  \Psi_Z^{z-2,r-q} \oplus \cdots  \oplus  \text{im} \  \Psi_Z^{1,r-q}\Big) \otimes_\mathbb{Z} \mathbb{R}.
\end{align*}
Let $\mathscr{M}_-$ be the matrix of the Poincar\'e duality pairing between $A^q_-$ and $A^{r-q}_-$,
and let $\mathscr{M}_{p_1,p_2}$ is the matrix of 
 the Poincar\'e duality pairing between $\text{im}\ \Psi_Z^{p_1,q} \otimes_\mathbb{Z}\mathbb{R}$ and 
 $ \text{im}\ \Psi_Z^{p_2,r-q} \otimes_\mathbb{Z}\mathbb{R}$.
Lemma~\ref{LemmaUpperTriangular} shows that the matrix of the deformed Poincar\'e pairing on $A^*_+$ is
\[
\left[\begin{array}{cccccc}
\mathscr{M}_- & 0 &0 & 0& \cdots & 0 \\
0 & \mathscr{M}_{1,z-1} & t  \mathscr{M}_{2,z-1} & t^2  \mathscr{M}_{3,z-1} & \cdots & t^{z-2} \mathscr{M}_{z-1,z-1} \\
0 &0& \mathscr{M}_{2,z-2} & t  \mathscr{M}_{3,z-2} &  &  t^{z-3} \mathscr{M}_{z-1,z-2}\\
0 & 0 &0 & \mathscr{M}_{3,z-3} & \cdots &  t^{z-4} \mathscr{M}_{z-1,z-3} \\
\vdots &\vdots&\vdots&\vdots& \ddots &  \vdots\\
0 & 0 &0& 0& 0 & \mathscr{M}_{z-1,1}\\
\end{array}\right].
\]
The claim on the limit of the deformed Poincar\'e duality pairing follows.
The minus sign on the right-hand side of the claim comes from the following computation made in Proposition~\ref{GysinTopIsomorphism}:
\[
\text{deg}\big(x_Z^{\text{rk}(Z)}x_\mathscr{Z}\big)=(-1)^{\text{rk}(Z)-1}.
\]

We use the deformed Poincar\'e duality pairing to understand the limit of the  bilinear form $Q^q_t$.
For an element $a$ of $A^1_+$, we write the multiplication with $a$ by
\[
M^a:A^*_+\longrightarrow A^{*+1}_+,\qquad x \longmapsto a \cdot x,
\]
and define its deformation $M^{a}_t:=S_t^{-1} \circ M^a \circ S_t$.
In terms of the operator $M^{\ell_+(t)}_t$, the bilinear form $Q^q_t$ can be written 
\begin{align*}
Q^q_t(a_1,a_2)&=(-1)^q \hspace{0.5mm}\text{deg}\Big(S_t(a_1)\cdot M^{\ell_+(t)} \circ \cdots \circ M^{\ell_+(t)} \circ S_t \ (a_2)\Big)\\
&=(-1)^q \hspace{0.5mm}\text{deg}\Big(S_t(a_1)\cdot S_t \circ M^{\ell_+(t)}_t \circ \cdots \circ M^{\ell_+(t)}_t \ (a_2)\Big)\\
&=(-1)^q \hspace{0.5mm} \Big\langle a_1 \hspace{0.5mm} ,\hspace{0.5mm}   M^{\ell_+(t)}_t \circ \cdots \circ M^{\ell_+(t)}_t \ (a_2) \Big\rangle_t
\end{align*}
Define linear operators $M^{1 \otimes \ell_Z}$, $M^{x_Z \otimes 1}$, and $M^\mathscr{T}$ on $G^*_Z$ 
by the isomorphisms
\begin{align*}
&\Big(G^*_Z,M^{1 \otimes \ell_Z}\Big) \simeq \Big(T^{*-1},1 \otimes \ell_Z \cdot -\Big),\\
&\Big(G^*_Z,M^{x_Z \otimes 1}\Big) \simeq \Big(T^{*-1},x_Z \otimes 1 \cdot -\Big),\\
&\Big(G^*_Z,M^{\mathscr{T}}\Big) \simeq \Big(T^{*-1},\mathscr{T} \cdot -\Big).
\end{align*}
Note that the linear operator $M^\mathscr{T}$ is the difference $M^{1 \otimes \ell_Z}-M^{x_Z \otimes 1}$.

\begin{claim}[2]
The limit of the operator $M^{\ell_+(t)}_t$as $t$ goes to zero decomposes into the sum
\[
\Big(A^*_+\hspace{0.5mm}  ,\hspace{0.5mm}  \lim_{t\to 0} M^{\ell_+(t)}_t \Big)=\Big(A^*_- \oplus G^*_Z \hspace{0.5mm}  ,\hspace{0.5mm}  M^{\ell_+(0)} \oplus M^\mathscr{T} \Big).
\]
\end{claim}

Assuming the second claim, we finish the proof as follows: 
We have
\[
\lim_{t\to 0} Q_t^q(a_1,a_2)=(-1)^q \hspace{0.5mm} \lim_{t \to 0}  \Big\langle a_1 \hspace{0.5mm} ,\hspace{0.5mm}   M^{\ell_+(t)}_t \circ \cdots \circ M^{\ell_+(t)}_t \ (a_2) \Big\rangle_t
\]
and from the first and the second claim, we see that the right-hand side is
\[
(-1)^q \hspace{0.5mm} \Big\langle a_1 \hspace{0.5mm} ,\hspace{0.5mm}   (M^{\ell_+(0)} \oplus M^\mathscr{T} ) \circ \cdots \circ   (M^{\ell_+(0)} \oplus M^\mathscr{T} ) \ (a_2) \Big\rangle_0 \\
=Q^q_-(b_1,b_2) + Q^q_Z(c_1,c_2),
\]
where $a_i=b_i+c_i$ for $ b_i \in A^*_-$ and $c_i \in G^*_Z$.
Notice that the minus sign in the first claim cancels with $(-1)^{q-1}$ in the Hodge-Riemann form 
\[
\Big(T_Z^{q-1}, Q^{q-1}_{\mathscr{T}}\Big) \simeq \Big(G^{q}_Z\hspace{0.5mm},\hspace{0.5mm}  Q_Z^{q}\Big).
\]

We now prove the second claim made above.  Write $M_t^{\ell_+(t)}$ as the difference 
\[
M_t^{\ell_+(t)}= S_t^{-1} \circ M^{\ell_+(t)} \circ S_t=S_t^{-1} \circ \Big(M^{\ell_+(0)}- M^{tx_Z}\Big) \circ S_t=M_t^{\ell_+(0)}- M_t^{tx_Z}.
\]
By Lemma~\ref{LemmaUpperTriangular}, the operators $M^{\ell_+(0)}$ and $S_t$ commute, and hence
\[
\Big( A^*_+\hspace{0.5mm} ,\hspace{0.5mm}  M^{\ell_+(0)}_t\Big)=\Big(A^*_+\hspace{0.5mm} ,\hspace{0.5mm}  M^{\ell_+(0)}\Big)=\Big( A^*_- \oplus G^*_Z\hspace{0.5mm} ,\hspace{0.5mm}  M^{\ell_+(0)} \oplus M^{1 \otimes \ell_Z}\Big).
\]
Lemma~\ref{LemmaUpperTriangular}  shows that 
the matrix of  $M^{x_Z}$ in the chosen bases of $A^q_+$ and $A^{q+1}_+$ is of the form
\[
\left[
\begin{array}{cccccc}
0& 0 &0 & \cdots & 0 & B_0\\
C &0& 0 & \cdots & 0 & B_1\\
0  &\text{Id} &0& \cdots & 0 & B_2\\
 \vdots  &\vdots & \ddots &\ddots& \vdots & \vdots \\
0   &0 &\cdots &\text{Id} &0 & B_{z-2} \\
 0   &0 &\cdots & 0 &\text{Id} &B_{z-1}
\end{array}
\right],
\]
where ``$\text{Id}$'' are the identity matrices representing 
\[
A^{q-p}(\mathrm{M}_Z)_\mathbb{R} \simeq \text{im} \ \Psi_Z^{p,q} \longrightarrow \text{im} \ \Psi_Z^{p+1,q+1} \simeq A^{q-p}(\mathrm{M}_Z)_\mathbb{R}.
\]
Note that  the matrix of the deformed operator $M^{tx_Z}_t$ can be written
\[
\left[
\begin{array}{cccccc}
0& 0 &0 & \cdots & 0 & t^{\frac{\text{rk}(Z)}{2}} B_0\\
t^{\frac{\text{rk}(Z)}{2}}C &0& 0&\cdots & 0 & t^{\text{rk}(Z)-1}B_1\\
0  &\text{Id} &0&\cdots  &0 &  t^{\text{rk}(Z)-2}B_2\\
  \vdots & \vdots & \ddots &\ddots& \vdots & \vdots \\
0   & 0&\cdots &\text{Id} &0 &t^2B_{z-2} \\
0    &0 & \cdots & 0 &\text{Id} &tB _{z-1}
\end{array}
\right].
\]
At the limit $t= 0$,  the matrix represents the sum  $0\oplus M^{x_Z\otimes 1}$,
and therefore
\begin{align*}
\Big(A^*_+\hspace{0.5mm} ,\hspace{0.5mm}  \lim_{t\to 0} M^{\ell_+(t)}_t\Big)&=\Big( A^*_- \oplus G^*_Z\hspace{0.5mm} ,\hspace{0.5mm}  M^{\ell_+(0)} \oplus M^{1 \otimes \ell_Z}\Big)-
\Big( A^*_- \oplus G^*_Z\hspace{0.5mm} ,\hspace{0.5mm} 0\oplus M^{x_Z\otimes 1} \Big) \\
& =\Big( A^*_- \oplus G^*_Z\hspace{0.5mm} ,\hspace{0.5mm}  M^{\ell_+(0)}\oplus M^{\mathscr{T}}\Big). 
\end{align*}
This completes the proof of the second claim.
\end{proof}

\begin{proof}[Proof of Proposition~\ref{BlowupHR}]
By Proposition~\ref{SumOfSignatures} and Proposition~\ref{PropositionLimit}, 
we know that $\lim_{t \to 0} Q^q_t$ is nondegenerate on $A^q_+$ and has signature
\[
\sum_{p=0}^q (-1)^{q-p}\Big(\text{dim}_\mathbb{R} A^p_+-\text{dim}_\mathbb{R} A^{p-1}_+ \Big) \ \ \text{for all nonnegative $q \le \frac{r}{2}$.}
\]
Therefore the same must be true for $Q^q_t$ for all sufficiently small positive $t$.
By construction, there is an isometry 
\[
\Big(A_+^q\hspace{0.5mm} ,\hspace{0.5mm} Q^q_t \Big)\simeq \Big(A_+^q\hspace{0.5mm} ,\hspace{0.5mm}  Q^q_{\ell_+(t)}\Big),
\]
and thus $A^*_+$ satisfies $\text{HR}(\ell_+(t))$ for all sufficiently small positive $t$.
\end{proof}

\subsection{}

We are now ready to prove the main theorem.
We write ``$\text{deg}$'' for the degree map of $\mathrm{M}$ and, for an order filter $\mathscr{P}$ of $\mathscr{P}_\mathrm{M}$, fix an isomorphism
\[
A^r(\mathrm{M},\mathscr{P}) \longrightarrow \mathbb{Z}, \qquad a \longmapsto \text{deg}\big(\Phi_{\mathscr{P}^c}(a)\big).
\]

\begin{theorem}[Main Theorem]\label{MainTheoremBody}
Let $\mathrm{M}$ be a loopless matroid, and let $\mathscr{P}$ be an order filter of $\mathscr{P}_\mathrm{M}$.
\begin{enumerate}[(1)]\itemsep 5pt
\item The Bergman fan $\Sigma_{\mathrm{M},\mathscr{P}}$ satisfies the hard Lefschetz property.
\item The Bergman fan $\Sigma_{\mathrm{M},\mathscr{P}}$ satisfies the Hodge-Riemann relations.
\end{enumerate}
\end{theorem}

When $\mathscr{P}=\mathscr{P}_\mathrm{M}$, the above implies Theorem \ref{MainTheoremIntroduction} in the introduction because any strictly submodular function defines a strictly convex piecewise linear function on $\Sigma_\mathrm{M}$.

\begin{proof}
We prove  by lexicographic induction on the rank of $\mathrm{M}$ and the cardinality of $\mathscr{P}$.
The base case of the induction is when $\mathscr{P}$ is empty, where we have
\[
A^*(\mathrm{M},\varnothing)_\mathbb{R}\simeq \mathbb{R}[x]/(x^{r+1}), \qquad x_i \longmapsto x.
\]
Under the above identification, the ample cone  is the set of positive multiples of $x$, and the degree $x^r$ is $1$.  It is straightforward to check  in this case that the Bergman fan satisfies the hard Lefschetz property and the Hodge-Riemann relations.

For the general case, we set $\mathscr{P}=\mathscr{P}_+$
and consider the matroidal flip from $\mathscr{P}_-$ to $\mathscr{P}_+$ with center $Z$.
By Proposition~\ref{PropositionSimplification} and Proposition~\ref{CombinatorialGeometry},
we may replace $\mathrm{M}$ by the associated combinatorial geometry $\overline{\mathrm{M}}$.
In this case, Proposition~\ref{Star} shows that the star of every ray in $\Sigma_{\mathrm{M},\mathscr{P}}$ is a product of at most two Bergman fans of matroids (one of which may not be a combinatorial geometry) to which the induction hypothesis on the rank of matroid applies.
We use Proposition~\ref{TensorHR} and  Proposition ~\ref{PropositionAmplePullback} to deduce that the star of every ray in
$\Sigma_{\mathrm{M},\mathscr{P}}$ satisfies the Hodge-Riemann relations, that is,
 $\Sigma_{\mathrm{M},\mathscr{P}}$ satisfies the local Hodge-Riemann relations.
  By Proposition~\ref{lHRtoHL}, this implies that  $\Sigma_{\mathrm{M},\mathscr{P}}$ satisfies the hard Lefschetz property.

Next we show that $\Sigma_{\mathrm{M},\mathscr{P}}$ satisfies the Hodge-Riemann relations.
 Since $\Sigma_{\mathrm{M},\mathscr{P}}$ satisfies the hard Lefschetz property, Proposition~\ref{ForSomeForAll} shows that it is enough to prove  that the Chow ring of $\Sigma_{\mathrm{M},\mathscr{P}}$ satisfies $\text{HR}(\ell)$  for some $\ell \in \mathscr{K}_{\mathrm{M},\mathscr{P}}$.
 Since the induction hypothesis on the size of order filter applies to both $\Sigma_{\mathrm{M},\mathscr{P}_-}$ and $\Sigma_{\mathrm{M}_Z}$, this follows from Proposition~\ref{BlowupHR}.
\end{proof}

We remark that the same inductive approach can be used to prove the following stronger statement (see \cite{Cattani} for an overview of the analogous facts in the context of convex polytopes and compact K\"ahler manifolds).  We leave details  to the interested reader.

\begin{theorem}\label{TheoremMixed}
Let $\mathrm{M}$ be a loopless matroid on $E$, and let $\mathscr{P}$ be an order filter of $\mathscr{P}_\mathrm{M}$.
\begin{enumerate}[(1)]\itemsep 5pt
\item The Bergman fan $\Sigma_{\mathrm{M},\mathscr{P}}$ satisfies the \emph{mixed hard Lefschetz theorem}: For any multiset
\[
\mathscr{L}:=\big\{\ell_1,\ell_2, \ldots, \ell_{r-2q} \big\}  \subseteq \mathscr{K}_{\mathrm{M},\mathscr{P}}, 
\]
 the linear map  given by the multiplication
\[
L_\mathscr{L}^q: A^q(\mathrm{M},\mathscr{P})_{\mathbb{R}} \longrightarrow A^{r-q}(\mathrm{M},\mathscr{P})_{\mathbb{R}}, \qquad a \longmapsto \big( \ell_1 \ell_2 \cdots \ell_{r-2q} \big)  \cdot a
\]
is an isomorphism for all nonnegative integers $q \le \frac{r}{2}$.
\item
The Bergman fan $\Sigma_{\mathrm{M},\mathscr{P}}$ satisfies the \emph{mixed Hodge-Riemann Relations}: For any multiset
\[
\mathscr{L}:=\big\{\ell_1,\ell_2, \ldots, \ell_{r-2q} \big\}  \subseteq \mathscr{K}_{\mathrm{M},\mathscr{P}} \ \ \text{and any} \ \ \ell \in \mathscr{K}_{\mathrm{M},\mathscr{P}},
\]
the symmetric bilinear form given by the multiplication
\[
Q^q_\mathscr{L}: A^q(\mathrm{M},\mathscr{P})_{\mathbb{R}} \times A^q(\mathrm{M},\mathscr{P})_{\mathbb{R}} \longrightarrow \mathbb{R}, \qquad (a_1,a_2) \longmapsto (-1)^q \text{deg}\  \big( a_1 \cdot L_\mathscr{L}^q(a_2)\big)
\]
is positive definite on the kernel of $\ell \cdot L^q_\mathscr{L}$ for all nonnegative integers $q \le \frac{r}{2}$.
\end{enumerate}
\end{theorem}

\section{Log-concavity conjectures}\label{SectionLCConjectures}

\subsection{}

Let $\mathrm{M}$ be a loopless matroid of rank $r+1$ on the ground set $E=\{0,1,\ldots,n\}$.
The \emph{characteristic polynomial}  of $\mathrm{M}$ is defined to be
\[
\chi_\mathrm{M}( \lambda)=\sum_{I \subseteq E} (-1)^{|I|}\  \lambda^{\text{crk}(I)},
\]
where the sum is over all subsets $I \subseteq E$ and $\text{crk}(I)$ is the corank of $I$ in $\mathrm{M}$.
Equivalently, 
\[
\chi_\mathrm{M}( \lambda)  = \sum_{F \subseteq E} \mu_\mathrm{M}(\varnothing,F)\   \lambda^{\text{crk}(F)},
\]
where the sum is over all flats $F \subseteq E$ and $\mu_\mathrm{M}$ is the M\"obius function of the lattice of flats of $\mathrm{M}$.
Any one of the two descriptions clearly shows that
\begin{enumerate}[(1)] \itemsep 2pt
\item the degree of the characteristic polynomial is $r+1$,
\item the leading coefficient of the characteristic polynomial is $1$, and
\item the characteristic polynomial satisfies $\chi_\mathrm{M}(1)=0$.
\end{enumerate}
See \cite{Zaslavsky,AignerEncyclopedia} for basic properties of the characteristic polynomial and its coefficients.

\begin{definition}
The \emph{reduced characteristic polynomial} $\overline{\chi}_\mathrm{M}( \lambda)$ is
\[
\overline{\chi}_\mathrm{M}( \lambda):=\chi_\mathrm{M}( \lambda)/ ( \lambda-1).
\]
We define a sequence of integers $\mu^0(\mathrm{M}),\mu^1(\mathrm{M}),\ldots,\mu^r(\mathrm{M})$ by the equality
\[
\overline{\chi}_\mathrm{M}( \lambda)=\sum_{k=0}^r (-1)^k\mu^k(\mathrm{M}) \hspace{0.5mm} \lambda^{r-k}.
\]
\end{definition}

The first number in the sequence is $1$, and the last number in the sequence is the absolute value of the M\"obius number $\mu_\mathrm{M}(\varnothing,E)$.
In general, $\mu^k(\mathrm{M})$ is the alternating sum of the absolute values of the coefficients of the characteristic polynomial
\[
\mu^k(\mathrm{M})=w_k(\mathrm{M})-w_{k-1}(\mathrm{M})+\cdots+(-1)^k w_0(\mathrm{M}).
\]
We will show that the Hodge-Riemann relations for $A^*(\mathrm{M})_\mathbb{R}$ imply the log-concavity 
\[
\mu^{k-1}(\mathrm{M})\mu^{k+1}(\mathrm{M}) \le \mu^k(\mathrm{M})^2 \quad \text{for} \quad 0<k<r.
\]  
Because the convolution of two log-concave sequences is log-concave, the above implies  the log-concavity of the sequence $w_k(\mathrm{M})$.

\begin{definition}
Let $\mathscr{F}=\{F_1\subsetneq F_2\subsetneq\cdots\subsetneq F_k\}$ be a $k$-step flag of nonempty proper flats of $\mathrm{M}$.
\begin{enumerate}[(1)]\itemsep 5pt
\item The flag $\mathscr{F}$ is said to be \emph{initial} if $r(F_m)=m$ for all indices $m$.
\item The flag $\mathscr{F}$ is said to be \emph{descending} if $\text{min}(F_1)>\text{min}(F_2)>\cdots>\text{min}(F_k)>0$.
\end{enumerate}
We write $\text{D}_k(\mathrm{M})$ for the set of initial descending $k$-step flags of nonempty proper flats of $\mathrm{M}$.
\end{definition}

Here, the usual ordering of the ground set $E=\{0,1,\ldots,n\}$ is used to define $\text{min}(F)$.

For inductive purposes it will be useful to consider the \emph{truncation} of $\mathrm{M}$, denoted $\text{tr}(\mathrm{M})$.
This is the matroid on $E$ whose rank function is defined by
\[
\text{rk}_{\text{tr}(\mathrm{M})}(I):=\text{min}(\text{rk}_\mathrm{M}(I),r).
\]
The lattice of flats of $\text{tr}(\mathrm{M})$ is obtained from the lattice of flats of $\mathrm{M}$ by removing all the flats of rank $r$.
It follows that, for any nonnegative integer $k<r$, there is a bijection
\[
\text{D}_k(\mathrm{M}) \simeq \text{D}_k(\text{tr}(\mathrm{M})),
\]
and an equality between the coefficients of the reduced characteristic polynomials
\[
\mu^k(\mathrm{M})=\mu^k(\text{tr}(\mathrm{M})).
\]
The second equality shows that all the integers $\mu^k(\mathrm{M})$ are  positive, see \cite[Theorem 7.1.8]{Zaslavsky}.

\begin{lemma} \label{LemmaShelling} 
For every positive integer $k \le r$, we have 
\[
\mu^k(\mathrm{M})=|\text{D}_k(\mathrm{M})|.
\]
\end{lemma}

\begin{proof}
The assertion for $k=r$ is  the known fact that 
$\mu^r(\mathrm{M})$ is the number of facets of  $\Delta_\mathrm{M}$ that are glued along their entire boundaries in its lexicographic shelling; see \cite[Proposition 7.6.4]{Bjorner}.
The general case is obtained from the same equality applied to repeated truncations of $\mathrm{M}$.  See \cite[Proposition 2.4]{Huh-Katz} for an alternative approach using Weisner's theorem.
\end{proof}

We now show that $\mu^k(\mathrm{M})$ is the degree of the product $\alpha_\mathrm{M}^{r-k} \hspace{0.5mm} \beta_\mathrm{M}^k$. 
See Definition~\ref{DefinitionAlphaBeta} for the elements $\alpha_\mathrm{M}, \beta_\mathrm{M} \in A^1(\mathrm{M})$, and
Definition~\ref{MatroidDegreeMap} for the degree map of $\mathrm{M}$.

\begin{lemma}\label{LemmaBetaPower}
For every positive integer $k \le r$, we have
\[
\beta^k_\mathrm{M}=\sum_{\mathscr{F}} x_\mathscr{F} \in A^*(\mathrm{M}),
\]
where the sum is over all descending $k$-step flags of nonempty proper flats of $\mathrm{M}$.
\end{lemma}

\begin{proof}
We prove by induction on the positive integer $k$.  
When $k=1$, the assertion is precisely that $\beta_{\mathrm{M},0}$ represents $\beta_{\mathrm{M}}$ in the Chow ring of $\mathrm{M}$:
\[
\beta_\mathrm{M} =\beta_{\mathrm{M},0}= \sum_{0 \notin F} x_F \in A^*(\mathrm{M}).
\]
In the general case, we use the induction hypothesis for $k$ to write
\[
\beta^{k+1}_\mathrm{M}=  \sum_\mathscr{F} \beta_\mathrm{M}\hspace{0.5mm} x_\mathscr{F},
\]
where the sum is over all descending $k$-step flags of nonempty proper flats of $\mathrm{M}$.
For each of the summands $\beta_\mathrm{M}\hspace{0.5mm} x_\mathscr{F}$, we write
\[
\mathscr{F}=\big\{F_1\subsetneq F_2\subsetneq\cdots\subsetneq F_k\big\}, \ \ \text{and set} \ \ i_\mathscr{F}:= \text{min}(F_1).
\]
By considering the representative  of $\beta_\mathrm{M}$ corresponding to the element $i_\mathscr{F}$, we see that
\[
\beta_\mathrm{M}\hspace{0.5mm} x_\mathscr{F}=\Big(\sum_{i_\mathscr{F} \notin F} x_F \Big)x_\mathscr{F}=\sum_{\mathscr{G}} x_{\mathscr{G}},
\]
where the second sum is over all descending flags of nonempty proper flats of $\mathrm{M}$ of the form
\[
\mathscr{G}=\big\{F \subsetneq F_1 \subsetneq \cdots \subsetneq F_k\big\}.
\]
This complete the induction.
\end{proof}

Combining Lemma~\ref{LemmaShelling}, Lemma~\ref{LemmaBetaPower}, and Proposition~\ref{PropositionFundamentalClass},
we see that the coefficients of the reduced characteristic polynomial of $\mathrm{M}$ are given by the degrees of the products $\alpha_\mathrm{M}^{r-k}\hspace{0.5mm} \beta_\mathrm{M}^k$:

\begin{proposition}\label{PropositionIntersectionDegree}
For every nonnegative integer $k \le r$, we have
\[
\mu^k(\mathrm{M})=\deg(\alpha_\mathrm{M}^{r-k} \hspace{0.5mm}\beta_\mathrm{M}^k).
\]
\end{proposition}

We illustrate the proof of the above formula 
for the rank $3$ uniform matroid $\mathrm{U}$ on $\{0,1,2,3\}$ with  flats
\[
\varnothing, \ \{0\}, \ \{1\}, \ \{2\}, \ \{3\},\ \{0,1\},\ \{0,2\},\ \{0,3\},\ \{1,2\},\ \{1,3\},\ \{2,3\}, \ \{0,1,2,3\}.
\]
The constant term $\mu^2(\mathrm{U})$ of the reduced characteristic polynomial of $\mathrm{U}$ is $3$, which is the size of the set of initial descending $2$-step flags of nonempty proper flats,
\[
\text{D}_2(\mathrm{U})=\Big\{\{2\} \subseteq \{1,2\},\ \{3\} \subseteq \{1,3\},\ \{3\} \subseteq \{2,3\} \Big\}.
\]
In the Chow ring of $\mathrm{U}$, we have $\beta_{\mathrm{U},1}=\beta_{\mathrm{U},2}=\beta_{\mathrm{U},3}$ by the linear relations, and hence
\begin{align*}
\beta_{\mathrm{U}}^2 &=\beta_\mathrm{U}(x_1+x_2+x_3+x_{12}+x_{13}+x_{23}) \\
&=\beta_{\mathrm{U},1} (x_1+x_{12}+x_{13})+\beta_{\mathrm{U},2} (x_2+x_{23})+\beta_{\mathrm{U},3} (x_3)\\
&=(x_0+x_2+x_3+x_{02}+x_{03}+x_{23}) (x_1+x_{12}+x_{13})\\
&\quad+(x_0+x_1+x_3+x_{01}+x_{03}+x_{13}) (x_2+x_{23})\\
&\quad +(x_0+x_1+x_2+x_{01}+x_{02}+x_{12}) (x_3).
\end{align*}
Using the incomparability relations, we see that there are only three nonvanishing terms in the expansion of the last expression, each corresponding to one of the three initial descending flag of flats:
\[
\beta^2=x_2x_{12}+x_3x_{13}+x_3x_{23}.
\]
\subsection{}

Now we explain why the Hodge-Riemann relations imply the log-concavity of the reduced characteristic polynomial.   
We first state a lemma involving inequalities among degrees of products:

\begin{lemma} \label{lem:ampleinequality}
Let $\ell_1$ and $\ell_2$ be elements of $A^1(\mathrm{M})_\mathbb{R}$.
If $\ell_2$ is nef, then 
\[
\text{deg}(\ell_1 \hspace{0.5mm} \ell_1 \hspace{0.5mm} \ell_2^{r-2})\ \text{deg}(\ell_2\hspace{0.5mm}  \ell_2 \hspace{0.5mm}  \ell_2^{r-2})\leq \text{deg}(\ell_1\hspace{0.5mm}  \ell_2 \hspace{0.5mm}  \ell_2^{r-2})^2.
\]
\end{lemma}

\begin{proof}
We first prove the statement when $\ell_2$ is ample.
Let $Q^1_{\ell_2}$ be the Hodge-Riemann form 
\[
Q^1_{\ell_2}:A^1(\mathrm{M})_\mathbb{R} \times A^1(\mathrm{M})_\mathbb{R} \longrightarrow \mathbb{R}, \qquad (a_1,a_2) \longmapsto -\text{deg}(a_1\hspace{0.5mm}\ell_2^{r-2} \hspace{0.5mm} a_2 ).
\]
Theorem \ref{MainTheoremBody} for $\mathscr{P}=\mathscr{P}_\mathrm{M}$ shows that the Chow ring $A^*(\mathrm{M})$ satisfies $\text{HL}(\ell_2)$ and $\text{HR}(\ell_2)$.
The property $\text{HL}(\ell_2)$ gives the Lefschetz decomposition
\[
A^1(\mathrm{M})_\mathbb{R}=\langle \ell_2 \rangle \oplus P^1_{\ell_2}(\mathrm{M}),
\]
which is orthogonal with respect to the Hodge-Riemann form $Q^1_{\ell_2}$. 
The property $\text{HR}(\ell_2)$ says that $Q^1_{\ell_2}$ is negative definite on 
$\langle \ell_2 \rangle$ and positive definite on its orthogonal complement $P^1_{\ell_2}(\mathrm{M})$.

Consider the restriction of $Q^1_{\ell_2}$ to the subspace $\langle \ell_1,\ell_2 \rangle \subseteq A^1(\mathrm{M})_\mathbb{R}$. 
Either $\ell_1$ is a multiple of $\ell_2$ or the restriction of $Q^1_{\ell_2}$ is indefinite,
and hence
\[
\text{deg}(\ell_1 \hspace{0.5mm} \ell_1 \hspace{0.5mm} \ell_2^{r-2})\ \text{deg}(\ell_2\hspace{0.5mm}  \ell_2 \hspace{0.5mm}  \ell_2^{r-2})\leq \text{deg}(\ell_1\hspace{0.5mm}  \ell_2 \hspace{0.5mm}  \ell_2^{r-2})^2.
\]

Next we prove the statement when $\ell_2$ is nef.
The discussion below Proposition~\ref{PropositionAmplePullback} shows that the ample cone $\mathscr{K}_\mathrm{M}$ is nonempty. 
Choose any ample class $\ell$, and use the assumption that $\ell_2$ is nef to deduce that
\[
\ell_2(t):=\ell_2+t \hspace{0.5mm} \ell  \ \text{is ample for all positive real numbers $t$.}
\]
Using the first part of the proof, we get, for any positive real number $t$, 
\[
\text{deg}(\ell_1 \hspace{0.5mm} \ell_1 \hspace{0.5mm} \ell_2(t)^{r-2})\ \text{deg}(\ell_2(t)\hspace{0.5mm}  \ell_2(t) \hspace{0.5mm}  \ell_2(t)^{r-2})\leq \text{deg}(\ell_1\hspace{0.5mm}  \ell_2(t) \hspace{0.5mm}  \ell_2(t)^{r-2})^2.
\]
By taking the limit $t \to 0$, we obtain the desired inequality.
\end{proof}

\begin{lemma}\label{AlphaNefBetaNef}
Let $\mathrm{M}$ be a loopless matroid.
\begin{enumerate}[(1)]\itemsep 5pt
\item The element $\alpha_\mathrm{M}$ is the class of a convex piecewise linear function on $\Sigma_\mathrm{M}$.
\item The element $\beta_\mathrm{M}$ is the class of a convex piecewise linear function on $\Sigma_\mathrm{M}$.
\end{enumerate}
In other words, $\alpha_\mathrm{M}$ and $\beta_\mathrm{M}$ are nef.
\end{lemma}

\begin{proof}
For the first assertion, it is enough to show that $\alpha_\mathrm{M}$ is the class of a nonnegative piecewise linear function that is zero on a given cone $\sigma_{\varnothing < \mathscr{F}}$ in $\Sigma_\mathrm{M}$.
For this we choose an element $i$ not in any of the flats in $\mathscr{F}$.
The representative $\alpha_{\mathrm{M},i}$ of  $\alpha_\mathrm{M}$ has the desired property.

Similarly, for the second assertion, it is enough to show that $\beta_\mathrm{M}$ is the class of a nonnegative piecewise linear function that is zero on a given cone $\sigma_{\varnothing < \mathscr{F}}$ in $\Sigma_\mathrm{M}$.
For this we choose an element $i$ in the flat $\text{min}\ \mathscr{F}$.
The representative $\beta_{\mathrm{M},i}$ of  $\beta_\mathrm{M}$ has the desired property.
\end{proof}

\begin{proposition}\label{PropositionLogConcave}
For every positive integer $k < r$, we have 
\[
\mu^{k-1}(\mathrm{M})\mu^{k+1}(\mathrm{M}) \le \mu^{k}(\mathrm{M})^2.
\]
\end{proposition}

\begin{proof}
We prove by induction on the rank of $\mathrm{M}$.
When $k$ is less than $r-1$, the induction hypothesis applies to the truncation of $\mathrm{M}$.
When $k$ is $r-1$, Proposition~\ref{PropositionIntersectionDegree} shows that the assertion is equivalent to the inequality
\[
\text{deg}(\alpha_\mathrm{M}^2\hspace{0.5mm} \beta_\mathrm{M}^{r-2})\text{deg}(\beta_\mathrm{M}^2\hspace{0.5mm} \beta_\mathrm{M}^{r-2})\leq\text{deg}(\alpha^1_\mathrm{M}  \hspace{0.5mm}\beta^{r-1}_\mathrm{M})^2.
\]
This follows from Lemma \ref{lem:ampleinequality} applied to  $\alpha_\mathrm{M}$ and $\beta_\mathrm{M}$, because $\beta_\mathrm{M}$ is nef by Lemma \ref{AlphaNefBetaNef}.
\end{proof}

We conclude with the proof of the announced log-concavity results.

\begin{theorem}\label{ReducedLogConcave}
Let $\mathrm{M}$ be a matroid, and let $G$ be a graph.
\begin{enumerate}[(1)]\itemsep 5pt
\item The coefficients of the reduced characteristic polynomial of $\mathrm{M}$ form a log-concave sequence. 
\item The coefficients of the characteristic polynomial of $\mathrm{M}$ form a log-concave sequence. 
\item The number of independent subsets of size $i$ of $\mathrm{M}$ form a log-concave sequence in $i$.
\item The coefficients of the chromatic polynomial of $G$ form a log-concave sequence.
\end{enumerate}
\end{theorem}

The second item proves the aforementioned conjecture of Heron \cite{Heron}, Rota \cite{Rota}, and Welsh \cite{WelshBook}.
The third item proves the conjecture of Mason \cite{Mason} and Welsh \cite{Welsh}.
The last item proves the conjecture  of Read \cite{Read} and Hoggar \cite{Hoggar}.

\begin{proof}
It follows from Proposition~\ref{PropositionLogConcave} that the coefficients of the reduced characteristic polynomial of $\mathrm{M}$ form a log-concave sequence. 
Since the convolution of two log-concave sequences is a log-concave sequence,
the coefficients of the characteristic polynomial of $\mathrm{M}$ also form a log-concave sequence.

To justify the third assertion, we use the result of Brylawski \cite{Brylawski, Lenz} that the number of independent subsets of size $k$ of $\mathrm{M}$ is the absolute value of the coefficient of $\lambda^{r-k}$ of the reduced characteristic polynomial of another matroid. 
It follows that the number of independent subsets of size $k$ of $\mathrm{M}$ form a log-concave sequence in $k$.

For the last assertion, we recall that the chromatic polynomial of a graph is given by the characteristic polynomial of the associated graphic matroid \cite{WelshBook}. More precisely, we have
\[
\chi_G(\lambda) = \lambda^{n_G} \cdot \chi_{\mathrm{M}_G}(\lambda),
\]
where $n_G$ is the number of connected components of $G$. It follows that the coefficients of the chromatic polynomial of $G$ form a log-concave sequence. 
\end{proof}

\end{document}